\documentclass[a4paper,12pt]{amsart}
\usepackage[utf8]{inputenc}
\usepackage{fullpage}
\usepackage{tikz-cd}
\usepackage[T1]{fontenc}
\usepackage{amssymb}
\usepackage[all]{xy}
\usepackage{graphicx}
\usepackage{amsmath}
\usepackage{amsthm}
\usepackage{float}
\usepackage{makecell}
\usepackage{bookmark}
\usepackage[top=3cm,bottom=3cm,left=3cm,right=3cm]{geometry}
\usepackage{tikz}
\usepackage{tikzit}
\tikzstyle{whitedot}=[fill=white, draw=black, shape=circle, inner sep=2pt]
\tikzstyle{blackdot}=[fill=black, draw=black, shape=circle, inner sep=2pt]
\tikzstyle{smallblackdot}=[fill=black, draw=black, shape=circle, inner sep=0pt]

\tikzstyle{red}=[-, draw=red]
\tikzstyle{green}=[-, draw={rgb,255: red,0; green,217; blue,0}]
\tikzstyle{blue}=[-, draw=blue]
\tikzstyle{thick}=[-, thick]

\usetikzlibrary{shapes.geometric}
\usepackage{mathtools}
\usepackage{marginnote}
\usepackage{multicol}
\usepackage{multirow}
\usepackage{rotating}
\usepackage{tabularx}
\usepackage{lscape}
\usepackage{stackengine}
\stackMath
\usepackage{arydshln}
\usepackage{array}
\usepackage{mathtools}
\usepackage{babel}
\usepackage{amsaddr}

\DeclareMathOperator{\Pic}{Pic}

\DeclareMathOperator{\rank}{rk}

\DeclareMathOperator{\NS}{NS}
\DeclareMathOperator{\MW}{MW}

\DeclareMathOperator{\Aut}{Aut}

\DeclareMathOperator{\id}{id}

\DeclareMathOperator{\Fix}{Fix}

\DeclareMathOperator{\W}{W}

\title{Jacobian elliptic fibrations on K3s with a non-symplectic automorphism of order 3}
\author{Felipe Zingali Meira}
\address{Universidade Federal do Rio de Janeiro, Rijksuniversiteit Groningen}
\email{f.zingali.meira@rug.nl}
\date{06/2024}

\newtheorem{thm}{Theorem}[section]
\newtheorem{lema}[thm]{Lemma}
\newtheorem{corollary}[thm]{Corollary}
\newtheorem{assumption}[thm]{Assumption}
\newtheorem*{acknowledge}{Acknowledgements}
\newtheorem{prop}[thm]{Proposition}

\theoremstyle{definition}
\newtheorem{dfn}[thm]{Definition}
\newtheorem{remark}[thm]{Remark}
\newtheorem{exemp}[thm]{Example}

\newcommand{\cB}{\mathcal B}         \newcommand{\cO}{\mathcal O}
\newcommand{\cC}{\mathcal C}         
         
\newcommand{\cE}{\mathcal E}         
\newcommand{\cF}{\mathcal F}         
\newcommand{\cG}{\mathcal G}

\newcommand{\cJ}{\mathcal J}

\newcommand{\ZZ}{\mathbb{Z}}
\newcommand{\QQ}{\mathbb{Q}}

\newcommand{\CC}{\mathbb{C}}

\newcommand{\PP}{\mathbb{P}}

\begin{document}

\maketitle

\begin{abstract}
We classify Jacobian elliptic fibrations on K3 surfaces with a non-symplectic automorphism $\sigma$ of order 3 according to the action of $\sigma$ on their fibres, building on work by Garbagnati and Salgado for non-symplectic involutions. We determine the possible reducible fibres types and give Weierstrass equations for Jacobian elliptic fibration which are preserved by $\sigma$. For a K3 surface $X$ with Picard number at least 12 and $\sigma$ acting trivially on $\NS X$, we apply the Kneser--Nishiyama method to find its Jacobian elliptic fibrations, and use our method to classify them in relation to every non-symplectic automorphism of order 3 in $\Aut X$.
\end{abstract}

\section{Introduction}

A regular algebraic surface with trivial canonical bundle is called a K3 surface. This class of surfaces forms a central topic in modern algebraic geometry, lying between the well understood class of rational surfaces and surfaces of general type. A surjective map from a K3 surface to $\PP^1$ is an elliptic fibration if all but finitely many of its fibres are elliptic curves. Moreover, we say that an elliptic fibration is Jacobian if it has a section. One of the many distinguishing properties of K3 surfaces is that they are the only class of surfaces which might admit more than one relatively minimal Jacobian elliptic fibration which is not of product type. It is thus natural to classify the Jacobian elliptic fibrations in K3 surfaces.

\sloppy There are different ways of defining an equivalence between two fibrations, each leading to a different classification. This is discussed in detail in \cite{BKW}. Furthermore, each classification involves different methods. Building on  work by Oguiso in \cite{Oguiso}, and expanded in \cite{Kloosterman} and \cite{Comparin-Garbagnati}, Garbagnati and Salgado developed a classification method for K3 surfaces $X$ with an involution $\iota$ which is non-symplectic, i.e. $\iota$ acts non-trivially in $H^{2,0}(X)$ (see \cite{GS1}, \cite{GS2}, \cite{GS3}). When the fixed locus of $\iota$ in nonempty, $X/\iota$ is a rational surface and every Jacobian elliptic fibration of $X$ corresponds to a linear system in the quotient. The method consists of separating the Jacobian elliptic fibrations of $X$ in three distinct types according to the action of $\iota$ on their fibres. For each type, we deduce properties of the induced linear system on $X/\iota$. Thus, we effectively reduce the classification of fibrations on $X$ to particular kinds of linear systems on a rational surface.

\subsection{Our contribution}
In this paper, we work torwards a generalization of the results of Garbagnati and Salgado to K3 surfaces $X$ with non-symplectic automorphisms $\sigma$ of any prime order. We focus on the case where the order of $\sigma$ is 3, as surfaces admitting such automorphisms usually allow for more types of elliptic fibrations compared to those with non-symplectic automorphisms of higher prime orders (see Proposition \ref{condition_type_1}). When $\cE \colon X \to \PP^1$ is a Jacobian elliptic fibration in $X$ and its fibre class $F$ is fixed by $\sigma$ in $\NS X$, this method allows us to find a Weierstrass equation for $\cE$ (see Propositions \ref{equation_type2} and \ref{equation_type1}).

We apply the Kneser--Nishiyama method to K3 surfaces with Picard number at least 12 and a non-symplectic automorphism of order 3 acting trivially on the Nerón--Severi group. Then, we use our methods to classify each fibration in relation to \textit{every} non-symplectic automorphism of the surface.

\subsection{Relevance and advantages of the method}
The reinterpretation of elliptic fibrations on well-known K3 surfaces, e.g. the most algebraic K3 according to Vinberg (see \cite{Vinberg}), in the light of a chosen non-symplectic automorphism of order 3 allows us to produce Weierstrass equations for fibrations that could not be obtained by realizing the surface as a degree two cover of a rational surface.
The classification of surfaces of degree 3 with relation to this automorphism also paves the way towards loosening the conditions on the constructions in \cite{Ohashi-Taki} and \cite{DilliesThesis} to the existence of fibrations of type 1 and 2 respectively (see Subsection \ref{subsec: applications}).

\subsection{Other Applications}\label{subsec: applications}
Jacobian elliptic fibrations on K3s with a non-symplectic automorphism $\sigma$ of order 3 were used for geometric constructions of other classes of varieties.
In \cite{Ohashi-Taki}, with the assumption that $\sigma$ fixes a curve of genus $g \geq 2$, they were used to construct log del Pezzo surfaces of index 3. In \cite{DilliesThesis}, with the assumption that $\sigma$ fixes a curve of genus 1, they were used to generalize the Borcea--Voisin construction of Calabi--Yau threefolds (see \cite{Borcea},\cite{Voisin}). 

\subsection{Organization of the paper}
In Section \ref{section_preliminaries}, we introduce elliptic fibrations on algebraic surfaces and discuss their possible reducible fibres. Then, we turn our discussion to K3 surfaces, stating facts about their elliptic fibrations and automorphisms. We give a brief summary of the distinct ways of classifying Jacobian elliptic fibrations on K3 surfaces. Finally, we present different kinds of linear systems which appear on rational elliptic surfaces and play a key role in the classification exhibited in \cite{GS1}.

In Section \ref{K3_surfaces_ns_order3}, we classify the Jacobian elliptic fibrations in K3 surfaces with regard to a fixed non-symplectic automorphism of order 3. We study a rational surface given by the quotient of our K3 by this non-symplectic automorphism, where each fibration induces a linear system. We prove that the kind of induced linear systems depends exclusively on the type of the fibration (see Theorem \ref{class_linearsys}). Furthermore, for fibrations in which the fibre class is fixed by the non-symplectic automorphism, we determine the possible singular fibres for each type (see Proposition \ref{condition_type_1} and Proposition \ref{fibres_type_2}), and prove that we can use their induced linear system to find a Weierstrass equation for the Jacobian elliptic fibration (see Propositions \ref{equation_type2} and \ref{equation_type1}).

In Section \ref{section_nishiyama}, we apply the Kneser--Nishiyama method (see \cite{Nishiyama1996}, \cite{Kneser}) to all K3 surfaces with Picard number at least 12 admitting a non-symplectic automorphism of order 3, with the condition that the automorphism acts trivially on its Néron--Severi group. Our restriction to such surfaces has a two-folded justification: the condition on the Picard number ensures that the Kneser--Nishiyama method provides a full $\cJ_2$-classification of the Jacobian elliptic fibrations, and the fact that the Trivial and Néron--Severi lattices of K3 surfaces with non-symplectic automorphisms of order 3 acting trivially on the Néron--Severi group are known thanks to the work of Artebani, Sarti and Taki (see \cite{AS},\cite{Taki2008}). We exhibit explicitly the choices of lattice and calculations of orthogonal complements necessary for this application. 

In Section \ref{section_X3_surface} we showcase our method described in Section \ref{K3_surfaces_ns_order3} by applying it to the $X_3$ surface, studied by Shioda and Inose in \cite{Shioda-Inose} and Vinberg in \cite{Vinberg}. The $\cJ_2$-classification of Jacobian elliptic fibrations of the $X_3$ surface was done in \cite{Nishiyama1996} by Nishiyama and exhibit Weierstrass equations for each class (see Theorem \ref{weierstrass_eq}, \ref{weierstrass_eq_2}).

\begin{acknowledge}
   The author thanks Cecília Salgado for suggesting the topic, which will be part of his thesis, and Alice Garbagnati for hosting him in Milan. The author also thanks both Cecília and Alice as well as Jaap Top for many fruitful discussions.
   The author was financially supported by a scholarship from Capes and the University of Groningen.
   
\end{acknowledge}

\section{Preliminaries}\label{section_preliminaries}

Let $k$ be an algebraically closed field of characteristic 0, which we fix once and for all.

\subsection{Elliptic fibrations}

\begin{dfn}
    A surjective map $\cE \colon S \to C$ from a smooth projective surface over $k$ to a smooth projective curve is called an \textit{elliptic fibration} if the generic fibre is a smooth curve of genus 1. We say that an elliptic fibration is \textit{relatively minimal} if no fibre contains a $(-1)$-curve, and \textit{Jacobian} if it admits a section $s_0 \colon C \to S$, which we call the zero-section. A triple $(S,\cE,s_0)$ is an \textit{elliptic surface} if $S$ is a surface admitting a relatively minimal Jacobian elliptic fibration $\cE \colon S \to C$ with at least one singular fibre and section $s_0 \colon C \to S$.
\end{dfn}

\begin{prop}\label{RES}
Let $\cF$ and $\cG$ be two cubics in $\PP^2$. Then, the resolution of the rational map $\varphi \colon \PP^2 \dashrightarrow \PP^1$ given by $P \mapsto [\cF(P) : \cG(P)]$ is a rational surface $R$ endowed with a Jacobian elliptic fibration $\cE \colon R \rightarrow \PP^1$.
\end{prop}

Over algebraically closed fields, every rational elliptic surface can be obtained by such a resolution (see \cite{miranda}, Lemma IV.1.2).

In what follows, let $(S, \cE, s_0)$ be an elliptic surface. There is a bijection between sections $s \colon C \to S$ of $\cE$ and $k(C)$-points of the generic fibre $E$ of $\cE$. Let $O$ be the $k(C)$-point induced by the zero section. Then $(E,O)$ is an elliptic curve over $k(C)$, and the set of sections of $\cE$ inherits the group structure of $E(k(C))$ (see \cite{silverman2}, Proposition 3.10).

\begin{dfn}
    The group of sections of $\cE$ is called the Mordell--Weil group of $\cE$ and denoted by $\MW (\cE)$.
\end{dfn}

The Néron--Severi group of $S$, denoted by $\NS S$ has a lattice structure with the bilinear pairing given by the intersection form (see \cite{schutt-shioda}, Theorem 6.4). The classes $F$ of a fibre and $O$ of the zero-section of $\cE$ in $\NS S$ determine a sublattice $\langle F, O\rangle \subseteq \NS S$.

\begin{dfn}
The lattice $W_{\cE} := \langle F, O \rangle^{\perp \NS S}$ is called the \textit{frame lattice} of $(S,\cE,s_0)$.
\end{dfn}

Let $F_v$ denote the fibre $\cE^{-1}(v)$ for every $v \in C$. Then, $F_v$ meets $O$ at a single point. Let $\Theta_{0,v}$ be the component of $F_v$ which intersects $O$, and let $m_v$ be the number of distinct components of $F_v$. We write $F_v$ in $\NS S$ as $F_v = \Theta_{0,v} + \sum_{i=1}^{m_v-1} \mu_{i,v} \Theta_{i,v}$, where $\Theta_{i,v}$ are the distinct component of $F_v$ and $\mu_{i,v}$ their multiplicities. The possible fibres on a minimal elliptic fibration were classified by Kodaira and Néron (see \cite{kodaira1963compact}, \cite{Neron1964}). Through this classification, we see that each type of reducible fibre is related to a lattice of $ADE$-type through the following.

\begin{dfn} \label{typelattice}
    The $ADE$\textit{-type} of a fibre $F_v$ is defined as the sublattice $$T_v = \bigoplus_{i=1}^{m_v-1} \Theta_{i,v}.$$ The sum of the $ADE$-types of all fibres of $\cE$ is called the $ADE$-type of $\cE$, and denoted by $T$.
\end{dfn}

 The $ADE$-type of a reducible fibre $F_v$ is in 1-1 correspondence with its Kodaira type, except when $T_v$ is equal to $A_1$ or $A_2$. For $T_v = A_1$, $F_v$ can be of type $I_2$ or $III$; for $T_v = A_2$, $F_v$ can be of type $I_3$ or $IV$. See Table IV.3.1 \cite{miranda} for a list of the $ADE$-types for reducible fibres of each possible Kodaira type in a Jacobian elliptic fibration, along with other properties.

\begin{prop}
    Let $(S,\cE, s_0)$ be an elliptic surface, and denote the Euler characteristic of $S$ by $\chi(S)$ (see \cite{beauville}). If $\chi(S) > 1$, then the following affirmations are true.
    \begin{itemize}
        \item[i)] The root type of the frame lattice determines the $ADE$-type of $\cE$, that is, $(W_\cE)_\mathrm{root} = T$.
        \item[ii)] The Mordell--Weil group of $\cE$ is isomorphic to $W_{\cE}/T$.
    \end{itemize}
\end{prop}

\begin{proof}
    See \cite{schutt-shioda}, Proposition 6.42 for item (i). Item (ii) is a consequence of (i) and the Shioda-Tate Theorem (see \cite{schutt-shioda}, Theorem 6.5). 
\end{proof}

\subsection{Elliptic fibrations in K3 surfaces}\label{elliptic_fibrations_k3}

In this section, we recall the definition of K3 surfaces, as well as some of their properties. More specifically, we present distinct ways of classifying elliptic fibrations on K3 surfaces.

\begin{dfn}
A \textit{K3 surface} is a smooth, projective surface $X$ such that $K_X = 0$ and $q(X) = h^1(X, \cO_X) = 0$.
\end{dfn}

If $X$ is a K3 surface, then the lattice given by $H^2(X,\ZZ)$ with the cup product is always isometric to $\Lambda_{K3} = U^{\oplus 3} \oplus E_8(-1)^{\oplus 2}$, where $U$ is the hyperbolic plane lattice (see \cite{huybrechts}, Chapter 1, Proposition 3.5). The Néron--Severi group $\NS X$ can be isometrically embedded as a lattice in $H^2(X,\ZZ)$. 

\begin{dfn}
    The transcendental lattice of $X$ is defined as $T_X := (\NS X)^{\perp H^2(X,\ZZ)}$.
\end{dfn}

On a K3 surface $X$, Jacobian elliptic fibrations are equivalent to embeddings $U \hookrightarrow \NS X$. Indeed, the fibre class $F$ and the zero-section $O$ of a Jacobian elliptic fibration $\cE$ determine a sublattice $\langle F, F + O\rangle \subseteq \NS X$ isometric to $U$. On the other hand, the following holds

\begin{prop}\label{K3_elliptic_fibration}
    Let $X$ be a K3 surface and $\psi \colon U \hookrightarrow \NS X$ an embedding of lattices. Then, there is an elliptic fibration $\cE \colon X \to \PP^1$ such that its frame lattice $W_{\cE}$ is isometric to $\psi(U)^{\perp \NS X}$.
\end{prop}

\begin{proof}
    See \cite{Kondo2}, Lemma 2.1. 
\end{proof}

Since $U$ is unimodular, we can write $\NS X = U \oplus W_{\cE}$. Consequently, $W_{\cE}$ is an even lattice with signature $(0,\rho(X)-2)$ and its discriminant lattice is isomorphic to that of $\NS X$.

For a K3 surface $X$, it is possible that there are multiple distinct embeddings of $U$ into $\NS X$. In this case, $X$ admits multiple Jacobian elliptic fibrations. As such, it is useful to define ways in which two Jacobian elliptic fibrations $\cE$ and $\cE'$ of $X$ are equivalent. The classifications presented here are studied in depth in \cite{BKW}.

\begin{dfn}
    Let $\cE, \cE'$ be Jacobian elliptic fibrations on $X$ with respective sections $s, s'$. Then
    \begin{itemize}
        \item[i)] $\cE$ and $\cE'$ are $\cJ_0$-equivalent if there is $g \in \Aut(\PP^1)$ such that $\cE' = g \circ \cE$ and $s' = s \circ g$.
        \item[ii)] $\cE$ and $\cE'$ are $\cJ_1$-equivalent if there is $g \in \Aut(\PP^1)$ and $\sigma \in \Aut(X)$ such that $\cE' = g \circ \cE \circ \sigma$ and $s' = \sigma \circ s \circ g$.
    \end{itemize}
    The sets of Jacobian elliptic fibrations of $X$ modulo $\cJ_0$ and $\cJ_1$-equivalence are denoted by as $\cJ_0(X)$ and $\cJ_1(X)$, respectively.
\end{dfn}

Notice that if $\cE$ and $\cE'$ are $\cJ_0$-equivalent, then they are $\cJ_1$-equivalent by taking $\id_X$ as the automorphism. The $\cJ_1$ classification is particularly important because of the following theorem.

\begin{thm}
    On any given K3 surface $X$, there are finitely many elliptic fibrations up to $\cJ_1$-equivalence. 
\end{thm}

\begin{proof}
    See \cite{Sterk}, Proposition 2.6 and Corollary 2.7.
\end{proof}

If $\cE$ and $\cE'$ are $\cJ_1$-equivalent, then the frame lattices $W_{\cE}$ and $W_{\cE'}$ are isomorphic. This motivates a third type of classification.

\begin{dfn}
    Let $\cE, \cE'$ be Jacobian elliptic fibrations of $X$. We say that they are $\cJ_2$-equivalent if $W_{\cE} \cong W_{\cE'}$. The set of elliptic fibrations of $X$ modulo $\cJ_2$-equivalence is denoted by $\cJ_2(X).$
\end{dfn}

This definition allows us to translate the classification problem to pure lattice theory.

\begin{dfn}
    Let $L$ be an even lattice. The \textit{discriminant group} of $L$ is defined as $ G_L := L^{\vee}/L$, where $L^{\vee}$ is the dual lattice of $L$. The \textit{discriminant form} of $L$ is a map $q_L \colon G_L \to \QQ/2\ZZ$, given by $x \mapsto \langle x, x\rangle \mod 2\ZZ$. The pair $(G_L,q_L)$ is called the \textit{discriminant lattice} of $L$.
\end{dfn}

Let $\cJ_2'(X)$ be the set of all even lattices, modulo isometries, with signature $(0,\rho(X) - 2)$ and discriminant lattice isomorphic to $(G_{\NS X}, q_{\NS X})$. Then, $\cJ_2(X)$ is a subset of $\cJ'_2(X)$.

\begin{prop}\label{picard_geq_12_works}
    Let $X$ be a K3 surface with $\rho(X) \geq 12$. Then, $\cJ'_2(X) = \cJ_2(X)$.
\end{prop}

\begin{proof}
    See \cite{SSES}, Lemma 12.21.
\end{proof}

\subsection{The Kneser--Nishiyama method}\label{nishiyama-method}

Based on lattice theoretic techniques developed by Kneser in \cite{Kneser}, Nishiyama developed a method of obtaining $\cJ'_2(X)$ for K3 surfaces $X$ with known $\NS X$ and $T_X$ (see \cite{Nishiyama1996}). This is known in the literature as the Kneser--Nishiyama method. In what follows we present a brief overview of it. See \cite{bertin}, Section 4.1, or \cite{BKW}, Section 4.1, for a similar overview.

\begin{dfn}
    A Niemeier Lattice is an even, unimodular, negative definite lattice of rank $24$.
\end{dfn}

\begin{thm}
Niemeier lattices are uniquely defined by their root types up to isometry, of which there only 24 possibilities.
\end{thm}

\begin{proof}
    See \cite{niemeier}, Theorem 8.5.
\end{proof}

Let $X$ be a K3 surface with transcendental lattice $T_X$.

\begin{thm}\label{nishiyama_thm}
    Let $T_0$ be a lattice of root type such that $\rank T_0 = \rank T_X + 4$, $G_{T_0} = G_{T_X}$ and $q_{T_0} = q_{T_X}$. Then, every $W \in \cJ'_2(X)$ can be written as $\phi(T_0)^{\perp L}$, for $\phi \colon T_0 \to L$ a primitive embedding into a Niemeier lattice $L$. Furthermore, if $\rho(X) \geq 12$, then there is a Jacobian elliptic fibration $\cE$ such that $W = W_{\cE}$, and the following holds.
    \begin{itemize}
        \item[i)] Let $M = \phi(T_0)^{\perp L_\mathrm{root}}$. Then the $ADE$-type of $\cE$ is $M_\mathrm{root} = W_\mathrm{root}$.
        \item[ii)] The rank of $\MW(\cE)$ is given by $\rank M - \rank M_\mathrm{root}$.
        \item[iii)] The torsion part of $\MW(\cE)$ is isomorphic to $\overline{W_\mathrm{root}}/W_\mathrm{root}$, where $\overline{W_\mathrm{root}}$ is the primitive closure of $W_\mathrm{root}$.
    \end{itemize}
\end{thm}

\begin{proof}
    See \cite{Nishiyama1996}, Sections 6.1 and 6.2.
\end{proof}

The Kneser--Nishiyama method consists of the application of the previous result to obtain $\cJ'_2(X)$. We describe it in te following steps.

\begin{itemize}
    \item[1.] Find a suitable lattice $T_0$ of root type.
    \item[2.] For each Niemeier lattice $L$, determine every possible primitive embedding ${\phi \colon T_0 \to L_\mathrm{root}}$, up to actions of $\W(L)$.
    \item[3.] For each embedding $\phi$, compute the orthogonal lattices $M := \phi(T_0)^{\perp L_\mathrm{root}}$ and the root type $M_\mathrm{root} = W_\mathrm{root}$.
    \item[4.] Compute the rank of the Mordell--Weil group, given by $\rank M - \rank M_\mathrm{root}$, and the torsion, given by $\overline{W_\mathrm{root}}/W_\mathrm{root}$.
\end{itemize} 

\subsection{Non-symplectic automorphisms of K3 surfaces}

Let $X$ be a K3 surface. Since $K_X = 0$, we have $h^0(X,K_X) = h^0(X,\cO_X) = 1$, and we can choose a generator $\omega_X$ of $H^0(X,K_X)$. If $\sigma: X \rightarrow X$ is an automorphism of $X$, $\sigma$ induces an action on $H^0(X,K_X)$ with the map $\omega_X \mapsto \sigma^*(\omega_X)$.

\begin{dfn}
    We say that an automorphism $\sigma$ is \textit{symplectic} if $\sigma^*(\omega_X) = \omega_X$ and \textit{non-symplectic} if $\sigma^*(\omega_X) = \zeta\omega_X$ for some root of unity $\zeta \neq 1$.
\end{dfn}

Let $X$ be a K3 surface endowed with a non-symplectic automorphism $\sigma$ of prime order $p$. By \cite{nikulin1979}, Theorem 0.1, $p$ is at most equal to 19. Denote the fixed locus of $\sigma$ by $\Fix(\sigma)$.

\begin{thm}\label{fixed_locus}
    For $X$ as above, $\Fix(\sigma)$ is either empty or the disjoint union of isolated points and smooth curves. Moreover, on the nonempty case, either $p =2$ and $\Fix(\sigma)$ is the disjoint union of two curves of genus 1, or
    \begin{equation*}
        \Fix(\sigma) = C \cup C_1 \cup {.}{.}{.} \cup C_{k-1} \cup \{P_1, {.}{.}{.}, P_n\}
    \end{equation*}
    where $C$ is a curve of genus $g$, $C_i$ are rational curves and $P_i$ isolated fixed points.
\end{thm}

\begin{proof}
    See \cite{AST}, Lemma 2.2.
\end{proof}

\begin{thm}\label{quotientK3}
    Let $X$ be a K3 surface and $\sigma \in \Aut X$ of finite order $n > 1$. Let $Y$ be the resolution of the quotient $X/\sigma$.
    \begin{itemize}
        \item[i)] If $\sigma$ is symplectic, then $X/\sigma$ is a K3 surface.
        \item[ii)] If $\sigma$ is non-symplectic with $n = 2$ and $\Fix(\sigma) = \varnothing$, then $Y$ is an Enriques surface.
        \item[iii)] If $\sigma$ is non-symplectic and either $n > 2$ or $\Fix(\sigma) \neq \varnothing$, then $Y$ is rational. 
    \end{itemize}
\end{thm}

\begin{proof}
    See \cite{Kondo}, Lemma 4.1 for item (i), \cite{zhang_quotients}, Lemma 1.2, and \cite{Xiao}, Lemma 2, for items (ii) and (iii).
\end{proof}

\subsection{Classification of fibrations in relation to a non-symplectic automorphism} \label{section-classification-prime}

In \cite{GS1}, Garbagnati and Salgado define a classification of elliptic fibrations on a K3 surface in relation to a non-symplectic involution $\iota$. This definition generalizes nicely to any non-symplectic automorphism $\sigma$ of prime order. In this section, we reproduce this definition and show some of its main properties.

\begin{dfn}\label{classification}
    Let $(X,\sigma)$ denote a pair consisting of $X$ a K3 surface, and $\sigma$ a fixed non-symplectic automorphism of $X$ of  prime order $p$. If $\cE \colon X \to \PP^1$ is a Jacobian elliptic fibration, we classify it in relation to $\sigma$ as follows.
    \begin{itemize}
        \item[1)] $\cE$ is of \textit{type 1} if $\sigma(F_v) = F_v$ for every fibre $F_v = \cE^{-1}(v)$.
        \item[2)] $\cE$ is of \textit{type 2} if $\sigma$ preserves the class of the fibre $[F]$ in $\NS X$, but $\sigma$ is not of type 1 (i.e. there is $v \in \PP^1$ such that $\sigma(F_v) = F_{v'}$, for $v \neq v'$).
        \item[3)] $\cE$ is of \textit{type 3} if $\sigma$ does not preserve the class of the fibre $[F]$ in $\NS X$.   
    \end{itemize}
\end{dfn}

\begin{prop}
    Any Jacobian elliptic fibration in $X$ is of type 1, 2 or 3 in relation to $\sigma$.
\end{prop}

\begin{proof}
    Suppose $\sigma$ fixes the fibre class, that is, $\sigma([F]) = [F]$, and let $\PP^1 \cong C \subset X$ be a section of $\cE$. Then the restriction of $\sigma$ to $C$ induces an automorphism $\sigma_{\PP^1} \in \Aut(\PP^1)$. We know that $\sigma_{\PP^1}^p = \id_{\PP^1}$ since $\sigma^p = \id_X$. Then, the automorphism $\sigma_{\PP^1}$ acts on $\PP^1$ either as the identity or cyclically with order $p$. If $\sigma_{\PP^1}$ acts as the identity, then $\sigma$ preserves each fibre $F_v$ and $\cE$ is of type 1, otherwise $\sigma(F_v) = F_{\sigma_{\PP^1}(v)}$ and $\cE$ is of type 2.
    Finally, if $[F]$ is not fixed by $\sigma$, then $\sigma([F]) = [F']$ and $\cE$ is of type 3.
\end{proof}

\begin{remark}\label{no_type_3}
    If $\sigma$ acts trivially on $\NS X$, then $(X,\sigma)$ does not admit elliptic fibrations of type 3.
\end{remark}

The following two propositions are adapted from Proposition 2.5 and Theorem 2.6 in \cite{GS2}, respectively, considering non-symplectic automorphisms of any prime order. We present the proofs for completeness.

\begin{prop}\label{type_1_rank_0}
    Suppose $\cE$ is a Jacobian elliptic fibration of type 1 in $(X, \sigma)$, and $\sigma$ acts trivially in $\NS X$. Then every section of $\cE$ is fixed by $\sigma$. Consequently, $\rank MW(\cE) = 0$.
\end{prop}

\begin{proof}    
    Let $C \subset X$ be a rational curve. Since $\sigma$ acts trivially on $\NS X$, and rational curves are unique in their class on K3 surfaces, $\sigma(C) = C$. Then, if $\Sigma$ is a section of $\cE$, since $\sigma$ acts trivially on the base of the fibration, $\Sigma$ must be fixed by $\sigma$. The number of curves fixed by $\sigma$ is finite, so $\rank MW(\cE) = 0$.
\end{proof}

\begin{prop}\label{fixes_high_genus}
If $(X,\sigma)$ admits a Jacobian elliptic fibration of type 2, then $\sigma$ fixes curves of genus at most 1.
\end{prop}

\begin{proof}
Let $\cE \colon X \to \PP^1$ be a type 2 elliptic fibration of $(X,\sigma)$, and $C \subset X$ a curve fixed by $\sigma$. Suppose $C$ is a multi-section of $\cE$. Then, $C$ intersects any fibre with positive multiplicity, and the intersection points must be fixed. However, since $\cE$ is of type 2, there are distinct fibres $F_v$ and $F_u$ such that $\sigma(F_v) = F_u$, so $F_v$ does not have fixed points. Therefore, $C$ must be a fibre component of $\cE$, so $g(C) \leq 1$.
\end{proof}

\subsection{Linear systems on rational elliptic surfaces}\label{res_linear_systems}

In section \ref{Induced_systems}, we study the relation between elliptic fibrations on K3 surfaces and linear systems on a related rational elliptic surface. With this in mind, it is important to define certain classes of linear systems over a general rational elliptic surface $\cE_R: R \rightarrow \PP^1$.

\begin{dfn}\label{cb_class}
    A \textit{conic bundle class} on $R$ is defined as a nef class $D$ on $\NS R$ such that
    \begin{itemize}
        \item[i)] $D \cdot (-K_{R}) = 2$.
        \item[ii)] $D^2 = 0$.
    \end{itemize}
\end{dfn}

If $D$ is a conic bundle class, then the induced morphism $\varphi_{|D|}: R \rightarrow \PP^1$ is a conic bundle on $R$, that is, a surjective morphism with a smooth genus 0 generic fibre (see \cite{renato}, Propositions 3.2 and 3.6).

Let $\eta: \tilde{R} \rightarrow R$ be the blow-up of $R$ in a finite number of points. A conic bundle class $D$ in $\tilde{R}$ only induces a conic bundle class in $R$ if $D$ does not intersect any of the exceptional divisors of $\eta$. Otherwise, the induced pencil $\eta_*|D|$ contains base points.

\begin{dfn}
    Let $R$ be a rational elliptic surface and $\eta:\tilde{R} \rightarrow R$ a blow-up map. We say that $D \in \NS R$ is a \textit{generalized conic bundle class} (with respect to $\eta$) if $D = \eta_* D'$, where $D'$ is a conic bundle class in $\tilde{R}$.
\end{dfn}

If a rational elliptic surface $R$ is relatively minimal, then the elliptic fibration $\cE_R: R \rightarrow \PP^1$ is unique. After taking a blow-up $\tilde{R}$, this is no longer necessarily true.

\begin{dfn}
    A \textit{splitting genus 1 pencil} on $\tilde{R}$ is a proper morphism $\varphi: \tilde{R} \rightarrow \PP^1$ such that
    \begin{itemize}
        \item[i)] $C_s := \varphi^{-1}(s)$ is a genus 1 curve for almost all $s \in \PP^1$.
        \item[ii)] $(C_s)(-K_{\tilde{R}}) = 0$ for all $s \in \PP^1$.
    \end{itemize}
    The pencil of curves $\{C_s\}_{s\in \PP^1}$ is also called a splitting genus 1 pencil.
\end{dfn}

\section{K3 surfaces with non-symplectic automorphisms of order 3}\label{K3_surfaces_ns_order3}

\subsection{Elliptic fibrations on $(X,\sigma)$}

In Section \ref{section-classification-prime}, we have seen how to classify Jacobian elliptic fibrations of a K3 surface in relation to a non-symplectic automorphism of prime order $p$. The case of $p=2$ was studied extensively by Garbagnati and Salgado in \cite{GS1} and \cite{GS2}. In what follows, we deal with the case $p = 3$. The following proposition describes a way in which the case of order 3 is special when compared to that of higher primes.

\begin{prop}\label{condition_type_1}
    Let $X$ be a K3 surface and $\sigma \in \Aut(X)$ a non-symplectic automorphism of prime order $p$. If $(X,\sigma)$ admits an elliptic fibration $\cE$ of  type 1, then $p = 2$ or $3$. Furthermore, if $p = 3$, the singular fibres of $\cE$ must be of type  $I_0^*, II, IV, II^*$ or $IV^*$.
\end{prop}

\begin{remark}
    This list of possible fibres for an elliptic fibrations of type 1 with respect to a non-symplectic automorphism of order 3 proven in \cite{Ohashi-Taki}, with the added assumption that $\sigma$ fixes a curve of genus $g \geq 2$. The proof is the same, and we present it here for completion.
\end{remark}

\begin{proof}
    If $\cE$ is an elliptic fibration of type 1, then $\sigma$ acts as an automorphism of order $p$ on each fibre. In particular, for smooth fibres, $\sigma$ is an automorphism on an elliptic curve. The only admissible automorphism groups of elliptic curves are $\ZZ/2$, $\ZZ/4$ and $\ZZ/6$ (\cite{silverman}, Chapter III, Theorem 10.1), so $p$ can only be 2 or 3.
    If $p = 3$, then there is a $\ZZ/6$ action on each smooth fibre of $\cE$, and their short Weierstrass form must be $y^2 = x^3 + B$. Consequently, the $J$-function of $\cE$ is constant and equal to zero. The only types of singular fibres with $J(F) = 0$ are $I_0^*, II, IV, II^*$ and $IV^*$ (see \cite{miranda}, Table IV.3.1).
\end{proof}

\begin{remark}
    If $p=2$ and the quotient $X/ \sigma$ is a relatively minimal rational elliptic surface, the singular fibres of a fibration of type 1 were classified in \cite{GS1}, Theorem 5.3. 
\end{remark}

Sufficient conditions for the existence of a Jacobian elliptic fibration of type 1 were given in \cite{AS}. 

\begin{prop}
    If $\sigma$ acts trivially in $\NS(X)$ and fixes at least 2 curves, then $(X,\sigma)$ admits a Jacobian elliptic fibration of type 1.
\end{prop}

\begin{proof}
    See Propostion 4.2 in \cite{AS}.
\end{proof}

A natural way to exhibit explicitly a pair $(X,\sigma)$ with $X$ a K3 surface and $\sigma$ a non-symplectic automorphism is to start from a rational elliptic surface $\cE \colon R \to \PP^1$, and consider its base change by a Galois covering $\tau \colon \PP^1 \to \PP^1$. Then, by resolving singularities and contracting $(-1)$-curves of the resulting surface $R \times_{\PP^1} \PP^1$, we obtain a surface $X$ with a relatively minimal Jacobian elliptic fibration $\cE_X \colon X \to \PP^1$. If $X$ is a K3 surface, the Galois morphism on $\PP^1$ induced by $\tau$ lifts to an automorphism $\sigma$ of $X$ with order equal to the degree of $\tau$. Since the quotient by $\sigma$ is birational to $R$, by \ref{quotientK3} $\sigma$ is non-symplectic.

\begin{prop}\label{basechange_is_k3}
    Let $\cE \colon R \to \PP^1$ be a rational elliptic surface, $\tau \colon \PP^1 \to \PP^1$ a cubic Galois covering ramified at $a,b \in \PP^1$ and $F_a, F_b$ the fibres of $\cE$ above $a$ and $b$. The surface $X$ obtained taking the base change of $\cE$ by $\tau$ is a K3 if and only if one of $F_a$, $F_b$ is of type $I_n^*$ or $IV$, while the other is of type $I_n, II$ or $III$.
\end{prop}

\begin{proof}
    By the canonical divisor formula for elliptic surfaces (see \cite{schutt-shioda}, Theorem 5.28), Noether's Formula (see \cite{beauville}, I.14) and Serre duality (\cite{beauville}, Theorem I.11), a Jacobian elliptic surface $X$ is K3 if and only if $e(X) = 24$. By \cite{schutt-shioda}, Theorem 5.31, $e(X) = \sum_{v \in \PP^1} e(F^X_v)$, where $F^X_v = \cE_X^{-1}(v)$. Let $u,v$ be points in $\PP^1$ such that $\tau(v) = u$. Then, the Euler number of $F^X_{v}$ is known in terms of the Kodaira type of $F_u$ and the ramification index $r(v|u)$ (see \cite{miranda}, Table VI.4.1). Knowing that $e(R) = \sum_{u \in \PP^1} e(F_u) = 12$, we obtain the result by checking for which Kodaira types of $F_a$ and $F_b$ we obtain $e(X) = 24$.
\end{proof}

\subsection{Singular fibres on fibrations of type 2}

Let $(X,\sigma)$ be a K3 surface with an automorphism of order 3. The goal of this section is to describe the configuration of singular fibres on a Jacobian elliptic fibration $\cE$ of type 2 in $(X,\sigma)$. Notice that if $\cE \colon X \to \PP^1$ comes from a degree 3 base change of a rational elliptic surface and $\sigma$ is the induced Galois automorphism, then $\cE$ is of type 2 in $(X,\sigma)$. Then, Proposition \ref{basechange_is_k3} limits the possible configurations of singular fibres in $\cE$. In what follows, we show that \textit{every} Jacobian elliptic fibration of type 2 in $(X,\sigma)$ comes from such a base change. As a first step, we study the properties of the quotient $X/\sigma$.

By Theorem \ref{quotientK3}, we know $X/\sigma$ is rational, but in general, it is not a rational elliptic surface. Since $\sigma$ is an automorphism of order 3, it admits isolated fixed points (Theorem \ref{fixed_locus}).  If $x\in X$ is an isolated fixed point of $\sigma$, then the local action around $x$ is given by the linear map
\begin{equation*}
    A_{3,1} = \begin{pmatrix}
        \zeta_3^2 & 0 \\
        0 & \zeta_3^2
    \end{pmatrix}.
\end{equation*}

Let $\pi \colon X \to X/\sigma$ be the quotient map. Analyzing the action of $A_{3,1}$, we can infer that $\pi(x)$ is a singularity of type $\tfrac{1}{3}(1,1)$ (see \cite{Reid-singularities}). In order to circumvent this, we can first blow-up the isolated fixed points of $\sigma$.

\begin{prop}\label{blow-up-fixed-pts}
    Let $\eta_X \colon \tilde{X} \to X$ be the blow-up of the isolated fixed points of $\sigma$. Then, the following statements are true.
    \begin{itemize}
        \item[i)] Every Jacobian elliptic fibration $\cE \colon X \to \PP^1$ lifts to a Jacobian elliptic fibration $\tilde{\cE} \colon \tilde{X} \to \PP^1$, and both fibrations are isomorphic on an open set of $\PP^1$.
        \item[ii)] We can lift $\sigma$ to an automorphism $\tilde{\sigma}$ of $\tilde{X}$ which fixes the exceptional curves of $\eta_X$.
        \item[iii)] Let $\tilde{R} := \tilde{X}/ \tilde{\sigma}$, and $\tilde{\pi} \colon \tilde{X} \to \tilde{R}$ be the quotient map. Then, $\tilde{R}$ is isomorphic to the minimal resolution $\phi$ of $X/\sigma$, and the following diagram commutes.
    \end{itemize}
    \begin{center}
    \begin{tikzcd}
        \tilde{R} \arrow[d, swap, "\phi"] & \tilde{X} \arrow[d, "\eta_{X}"] \arrow[l, "\tilde{\pi}"] \\
        X/\sigma & X \arrow[l,"\pi"] 
    \end{tikzcd}
\end{center}
\end{prop}

\begin{proof}
    i) Let $[F]$ be the class of a fibre and $\Sigma$ a section of $\cE$ in $\NS X$, and write $\tilde{F} := \eta_X^*(F)$ and $\tilde{\Sigma} := \eta_X^*(\Sigma)$. Since a general smooth fibre of $\cE$ does not intersect the isolated fixed points of $\sigma$,  the pencil $|\tilde{F}|$ induces an elliptic fibration on $\tilde{X}$ with $\tilde{\Sigma}$ as a section.

    ii) $\tilde{X}$ can be described locally around an exceptional divisor by coordinates $(z_1,z_2,t)$ with $z_2 = tz_1$. We define $\tilde{\sigma}$ by the map $(z_1,z_2,t) \mapsto (\zeta_3^2z_1,\zeta_3^2z_2,t)$. Then, $\tilde{\sigma}$ agrees with $\sigma$ outside of $z_1 = z_2 = 0$, and fixes the exceptional curve.

    iii) Let $Z = \{P_1,...,P_n\}$ be the set of isolated fixed points of $\sigma$, and ${\tilde{Z} = \{E_1,...,E_n\}}$ the set of corresponding exceptional curves on $\tilde{X}$. The open subsets $X\setminus Z$ and $\tilde{X}\setminus \tilde{Z}$ are isomorphic, and $\sigma$ and $\tilde{\sigma}$ agree under this identification. By taking the respective quotients, we obtain that the open sets $Y \setminus \phi^{-1}(\pi(Z))$ and $\tilde{R} \setminus \tilde{\pi}(\tilde{Z})$ are isomorphic. For any isolated fixed point $P_i$, we know that $E_i^2 = -1$ and it is fixed by $\tilde{\sigma}$. Then $C_i = \tilde{\pi}(E_i)$ is rational and its self intersection is $C_i^2 = C_i \cdot \tilde{\pi}_*(E_i) = E_i \cdot \tilde{\pi}^*(C_i) = 3E_i^2 = -3$. By \cite{Reid-singularities} (Example 3.1) the minimal resolution of a singularity of type $\tfrac{1}{3}(1,1)$ is a rational curve with self-intersection $-3$, so $Y$ and $\tilde{R}$ are isomorphic as claimed.
\end{proof}

\begin{prop} \label{type2_induces}
    Let $\cE \colon X \to \PP^1$ be a Jacobian elliptic fibration of type 2 in $(X,\sigma)$. Then, $\cE$ induces a Jacobian elliptic fibration $\cE_{\tilde{R}} \colon \tilde{R} \to \PP^1$, so $\tilde{R}$ is a rational elliptic surface.
\end{prop}

\begin{proof}
    By Proposotion \ref{blow-up-fixed-pts}, $\cE$ induces an elliptic fibration $\tilde{\cE}$ on $\tilde{X}$. Let $\tilde{F}$ be the fibre class and $\tilde{\Sigma}$ a section of $\tilde{\cE}$, and $D, C$ their respective pushforwards by $\tilde{\pi}$. We claim that the pencil $|D|$ induces an elliptic fibration on $\tilde{R}$ with section $C$.

    Since $\cE$ is of type 2, every curve fixed by $\sigma$ is a fibre component. Then, for all but finitely many choices of $v_1 \in \PP^1$, there are three distinct smooth fibres $\tilde{F}_{v_1}$, $\tilde{F}_{v_2}$ and $\tilde{F}_{v_3}$ in an orbit of $\tilde{\sigma}$. Consequently, for a generic choice of $D_v$ in the pencil $|D|$, the map $\tilde{\pi}$ defines a cubic covering of $D_v$ by 3 disjoint smooth genus 1 curves. By the Riemann-Hurwitz Theorem, $D_v$ must also be smooth of genus 1, and $|D|$ a genus 1 pencil. Furthermore, $D^2 = \tilde{\pi}_* \tilde{F} \cdot D_v = \tilde{F} \cdot \tilde{\pi}^*(D_v) = \tilde{F} \cdot (\tilde{F}_{v_1} + \tilde{F}_{v_2} + \tilde{F}_{v_3}) = 0$,  so $|D|$  induces an elliptic fibration $\cE_{\tilde{R}} \colon \tilde{R} \to \PP^1$. By the intersection product $D \cdot C = \tilde{F} \cdot \tilde{\Sigma} = 1$, $C$ is a section and we conclude that $\cE_{\tilde{R}}$ is a Jacobian elliptic fibration.
\end{proof}

Let $R$ be the relatively minimal model of $\tilde{R}$ with respect to the elliptic fibration $\cE_{\tilde{R}}$. There is a blow-down $\eta_R: \tilde{R} \rightarrow R$ and $R$ is endowed with a relatively minimal elliptic fibration $\cE_R: R \rightarrow \PP^1$ such that $\cE_{\tilde{R}} = \eta_R \circ \cE_R$.

\begin{prop}\label{basechange_R}
    Let $\sigma_{\PP^1} \colon \PP^1 \to \PP^1$ be the action induced by $\sigma$ on the base of $\cE$, and $\tau \colon \PP^1 \to \PP^1$ the quotient by $\sigma_{\PP^1}$. Then, the relatively minimal elliptic surface induced by the base change of $\cE_{R} \colon R \to \PP^1$ by $\tau$ is $\cE_X \colon X \to \PP^1$, and the induced automorphism is $\sigma$.
\end{prop}

\begin{proof}
    Let $U \subset \PP^1$ be an open set such that $\cE_{\tilde{R}}^{-1}(v)$ is a smooth fibre for every $v \in U$. Then, $\cE_{\tilde{R}}^{-1}(U)$ is isomorphic to $\cE_R^{-1}(U)$, and the base change of $\cE_R$ by $\tau$ must be birational to $\cE_{\tilde{X}} \colon \tilde{X} \to \PP^1$. After resolving singularities and contracting $(-1)$-curves on the fibres, by the uniqueness of the relatively minimal model of elliptic surfaces, we obtain the fibration $\cE_X \colon X \to \PP^1$. Furthermore, since the automorphism induced by this base change agrees with $\sigma$ on the dense open set $\cE_R^{-1}(U)$, they must agree everywhere.
\end{proof}

As a consequence of this, we can use Proposition \ref{basechange_is_k3} to prove the following.

\begin{prop}\label{fibres_type_2}
    Let $\cE \colon X \to \PP^1$ be a Jacobian elliptic fibration of type 2 in $(X,\sigma)$. Then, $\sigma$ preserves two fibres $F^X_a$ and $F^X_b$, and every other fibre is in an orbit $F^X_{v_1},F^X_{v_2},F^X_{v_3}$ of $\sigma$. Furthermore:
    \begin{itemize}
        \item[i)] $F^X_a$ is of type $I_0$ or $I_n^*$ for $n = 0,3,6,9,12$.
        \item [ii)] $F^X_b$ is of type $I_0^*,III^*$ or $I_m$ for $m = 0,3,6,9,12,15,18$. 
        \item[iii)] $F^X_{v_1}, F^X_{v_2}$ and $F^X_{v_3}$ have the same type, which can be $II, III, IV, IV^*, I_n^*$ for $n = 0,1$ or $I_m$ for $m = 0,1,\ldots, 6$.
        \item[iv)] There are no fibres of type $II^*$, $I_n^*$ for $n = 2,4,5,7,8,10,11,13$ or $I_n$ for $n = 7,8,10,11,13,14,16,17,19$. 
    \end{itemize}
\end{prop}

\begin{proof}
    By Proposition \ref{basechange_R}, $\cE$ is the base change of a rational elliptic surface $\cE_R \colon R \to \PP^1$ by a Galois covering $\tau \colon \PP^1 \to \PP^1$ of degree 3 ramified over $a,b \in \PP^1$. By Proposition \ref{basechange_is_k3}, we know $F_a$ is of type $IV$ or $I_n^*$ and $F_b$ is of type $II, III$ or $I_m$. Since $R$ is rational, we know by the Shioda--Tate formula (see \cite{shioda1989mordell}, Theorem 1.1) that fibres of $\cE_R$ have at most 9 components, if $F_u$ is of type $I_m$ or $I_n^*$, we know $m \leq 9$ and $n \leq 4$. 

    Let $F^X_a, F^X_b$ be the fibres of $\cE$ above $F_a, F_b$ respectively. Then, by \cite{miranda}, Table VI.4.1, we know $F^X_a$ is of type $I_0$ or $I_{n}^*$ for $n = 0,3,6,9,12$, and $F^X_b$ is of type $I_0^*,III^*$ or $I_m$, for $m = 0,3,6,9,12,15,18$. This proves (i) and (ii).
    
    Let $F_u$ be a fibre of $\cE_R$ for $u \neq a,b$. Since $u$ is not ramified by $\tau$, there are three distinct points $v_1,v_2,v_3 \in \PP^1$ such that $\tau(v_i) = u$, and $F^X_{v_i}$ has the same type as $F_u$. By \cite{Persson1990}, we know that any of the fibre types listed in (iii) are possible for $F_u$.

    It remain to prove that the types listed in (iv) are not possible for $F_u$. Firstly, notice that $F_a$ has at least 3 components, so by Shioda--Tate $F_u$ has at most 7 components. Therefore, it remains to check for $I_2^*$ and $I_7$. Since both have 7 components, we know $F_a$ is of type $IV$ and $F_b$ is irreducible. However, these configurations are impossible (see \cite{perssonlist}, Table 2.1, configurations 22 and 27).
\end{proof}

\begin{corollary}
    Assume that $\sigma$ acts trivially on $\NS X$. Then, every fibre other than $F^X_a$ and $F^X_b$ is irreducible. 
\end{corollary}

\begin{proof}
    Assume $F^X_v$ is not irreducible, for $v \neq a,b$. Then, $\sigma$ takes the components of $F^X_v$ to the components of another fibre. Since fibre components are independent in $\NS X$, this constitutes a non-trivial action.
\end{proof}

\subsection{Classification by induced linear systems}\label{Induced_systems}

Let $X$ be a K3 surface with a non-symplectic involution $\iota$. In work by Garbagnati and Salgado, the elliptic fibrations of $X$ are directly related to linear systems on the quotient $X/\iota$, which is shown to be a rational elliptic surface with the assumption that $\iota$ fixes curves of genus at most 1 (see \cite{GS1}, \cite{GS2}). 
Our goal is to study the linear systems induced by Jacobian elliptic fibrations on the resolution $\tilde{R}$ of the quotient $X/\sigma$ (see \ref{blow-up-fixed-pts}). In order to do this, we work with the following assumption through the rest of this section.

\begin{assumption}\label{assump}
    Let $X$ be a K3 surface and $\sigma \in \Aut(X)$ a non-symplectic automorphism of order 3. We assume that $X$ admits a Jacobian elliptic fibration $\cE_X \colon X \to \PP^1$ of type 2 with relation to $\sigma$.
\end{assumption}

Let $(X, \sigma, \cE_X)$ be a K3 surface with Assumption \ref{assump}. Let $\cE$ be an elliptic fibration on $X$ (possibly distinct from $\cE_X$). Then $\cE$ induces a pencil of curves $\Lambda$ on $\tilde{R}$ by pulling back $|F|$ by $\eta_X$, and then applying the pushforward by $\tilde{\pi}$. The following theorem describes the relation between the type of an elliptic fibrations in \ref{classification} and which kind of pencil it induces on $\tilde{R}$ (see section \ref{res_linear_systems}).

\begin{thm} \label{class_linearsys}
    The induced pencil $\Lambda$ is determined by the type of $\cE$.
    \begin{itemize}
        \item[i)] $\cE$ is of type 1 if and only if $\Lambda$ is a conic bundle class of $\tilde{R}$.
        \item[ii)] $\cE$ is of type 2 if and only if $\Lambda$ is a splitting genus 1 pencil of $\tilde{R}$.
        \item[iii)] $\cE$ is of type 3 if and only if $\Lambda$ is a non-complete linear system.
    \end{itemize}
\end{thm}

The proof of this theorem is a direct adaptation of Theorem 4.2 in \cite{GS1}, in which the automorphism $\sigma$ is a non-symplectic involution. The biggest change is the necessity of pulling back the linear system to $\tilde{X}$ before taking the quotient.

\begin{proof}
    Let $[\tilde{F}]$ be the class of a fibre of the elliptic fibration induced in $\tilde{X}$ by $\cE$ and $\tilde{F}_v$ the pullback of $F_v$. 
    
    Suppose that $\cE$ is of type 1. Then the action of $\sigma$ on each $F_v$ lifts to an action of $\tilde{\sigma}$ on $\tilde{F}_v$. The pencil $\Lambda$ is given by the system of curves $\{D_v\}_{v \in \PP^1}$, where $D_v = \tilde{\pi}(\tilde{F_v}) = \tilde{F}_v/\tilde{\sigma}$. Since $\sigma$ has a finite number of isolated fixed points, for all but finitely many $v \in \PP^1$, $F_v$ is smooth and $D_v = F_v/\sigma$. Applying the Riemann-Hurwitz Theorem to the quotient map $F_v \to D_v$, we know that $g(D_v) = 0$ if and only if the map ramifies in two distinct point with index 3, and  $g(D_v) = 1$ if and only it is unramified. If we assume $g(D_v) = 1$, then $\sigma$ acts as the translation of a torsion point of $F_v$ as an elliptic curve, fixing its period. Furthermore, since this is true for all but finitely many $v \in \PP^1$ and $\sigma$ acts as the identity on the base of $\cE$, then $\sigma$ must preserve the period of $X$. That is not possible due to the assumption that $\sigma$ is non-symplectic, so $g(D_v) = 0$. We can calculate the self intersection as $D_v^2 = D_v \cdot \tilde{\pi}(\tilde{F}_v) = \tilde{\pi}^*(D_v) \cdot \tilde{F}_v = 3\tilde{F}_v \cdot \tilde{F}_v = 0$. By the adjunction formula, $D_v \cdot K_{\tilde{R}} = -2$. We conclude that $\Lambda = |D_v|$ is a generalized conic bundle of $R$ (with respect to $\eta_R \colon \tilde{R} \to R$).
    
    Suppose $\cE$ is of type 2. By \ref{type2_induces}, we know that $\Lambda$ consists of the system of fibres $\{D_v\}_{v\in \PP^1}$ in a Jacobian elliptic fibration. Consequently, $D_v^2 = 0$ and $g(D_v) = 1$, and by the adjunction formula, $D_v \cdot K_{\tilde{R}} = 0$. Therefore, $\Lambda$ is a splitting genus 1 pencil.

    Suppose $\cE$ is of type 3. Then $\sigma([F]) = [F']$ and $\sigma([F']) = [F'']$, for $[F], [F'], [F'']$ three distinct classes on $\NS X$, each respectively inducing distinct elliptic fibrations $\cE$, $\cE'$, $\cE''$. Pulling back these classes by $\eta_X$, we obtain $[\tilde{F}], [\tilde{F}'], [\tilde{F}'']$ distinct classes in $\NS \tilde{X}$. Since they are supported on smooth curves of $\tilde{X}$, the intersection products $\tilde{F}\tilde{F}'$, $\tilde{F}\tilde{F}''$ and $\tilde{F}'\tilde{F}''$ are all greater than 0. Let $\tilde{F}\tilde{F}' + \tilde{F}\tilde{F}'' + \tilde{F}'\tilde{F}'' = m > 0$, then $(\tilde{F}+\tilde{F}'+\tilde{F}'')^2 = 2m$. Since $\tilde{F}^2 = \tilde{F}'^2 = \tilde{F}''^2 = 0$, the linear system $|\tilde{F}+\tilde{F}'+\tilde{F}''|$ is base point free. In particular, there is a smooth curve $C_X$ of genus $m+1$ whose class is $[\tilde{F} + \tilde{F}' + \tilde{F}'']$. As a consequence, $|C_X| = |\tilde{F}+\tilde{F}'+\tilde{F}''|$ is an $m+1$ dimensional linear system with smooth general elements (see \cite{SaintDonat}, Proposition 2.6). On the other hand, the family of curves $\tilde{F}_v + \tilde{F}'_v + \tilde{F}''_v$, given by $\eta_X^{-1}(\cE^{-1}(v) + \cE'^{-1}(v) + \cE''^{-1}(v))$ for each $v \in \PP^1$, has dimension 1 and reducible general element. Taking $D_v = \tilde{\pi}(\tilde{F}_v) = \tilde{\pi}(\tilde{F}'_v) = \tilde{\pi}(\tilde{F}''_v)$, we conclude $\Lambda = \{ D_v \}_{v \in \PP^1}$ is a non-complete sub-linear system of $|\tilde{\pi}(C_X)|$.
\end{proof}

\subsection{Equations for elliptic fibrations of type 1 and 2} \label{weierstrass_section}

Let $\cE_X$ be a Jacobian fibration of type 2 in $(X,\sigma)$. By Proposition \ref{basechange_R}, $\cE_X$ is the base change of a rational elliptic fibration $\cE_R \colon R \to \PP^2$ by a map $\tau \colon \PP^1 \to \PP^1$. The fibration $\cE_R$ is constructed as a resolution $\eta$ of a rational map $\varphi \colon \PP^2 \dashrightarrow \PP^1$ given by $[x{:}y{:}z] \mapsto [\cF(x,y,z) : \cG(x,y,z)]$, and after a change of coordinates, we can assume $\tau$ is given by $[s{:}t] \mapsto [s^3{:}t^3]$. Thus, the generic fibre of $\cE_X$ can be written as

\begin{align*}
    \cE_X \colon \cF(x,y,z) + t^3 \cG(x,y,z) = 0.
\end{align*}

Now, let $\cE \colon X \to \PP^1$ be an elliptic fibration distinct from $\cE_X$, and $\Lambda$ the induced linear system in $\tilde{R}$. Through the contractions $\eta_{\tilde{R}} \colon \tilde{R} \to R$ and $\eta \colon R \to \PP^2$, $\Lambda$ induces a pencil of curves $\Gamma$ in $\PP^2$. In this section, we show how to use $\Gamma$ to deduce an equation for the generic fibre of $\cE$, when $\cE$ is of type 1 or 2 in relation to $\sigma$.

Let $\cE \colon X \to \PP^1$ be an elliptic fibration of type 2 in $(X,\sigma)$ distinct from $\cE_X$. By Proposition \ref{basechange_R}, we know that $\cE$ is obtained with the base change of a rational elliptic surface by a cubic Galois covering $\tau_{\cE} \colon \PP^1 \to \PP^1$.

\begin{prop}\label{equation_type2}
    Let $F_a$, $F_b$ be the fibres above the ramification points of $\tau_{\cE}$, and $\cC_a$, $\cC_b$ the induced curves in $\Gamma$. Then, we can write the generic fibre of $\cE$ by the following equation
    \begin{align*}
        \cE \colon \cC_a(x,y,z) + t^3 \cC_b(x,y,z) = 0.
    \end{align*}
\end{prop}

\begin{proof}
    By Theorem \ref{class_linearsys}, we know that $\Lambda$ is a genus 1 pencil inducing a Jacobian elliptic fibration $\cE' \colon \tilde{R} \to \PP^1$. Then, $\Gamma$ must be a pencil of genus 1 curves in $\PP^2$ generated by $\cC_a$ and $\cC_b$. For all but finitely many $t \in \PP^1$, the fibre $(\cE')^{-1}(t)$ is isomorphic to $\cC_a(x,y,z) + t \cC_b(x,y,z) = 0.$
    By a change of coordinates, we can suppose that $\tau_{\cE}$ is given by the map $t \mapsto t^3$. Thus, applying the base change by $\tau_{\cE}$, we obtain the wanted equation for the generic fibre of $\cE$.
\end{proof}

\begin{remark}\label{remark_ADE}
    In order to use this proposition, we need to know what are the fibres of $\cE \colon X \to \PP^1$ above the ramification locus of $\tau_{\cE}$. We can deduce this from the $ADE$-type of $\cE$. By Proposition \ref{fibres_type_2}, every reducible fibre which is the only one of its Kodaira type must be ramified by the base change, otherwhise it would have 3 copies. For instance, if the $ADE$-type of $\cE$ is $D_7 \oplus E_7$, then the only reducible fibres are of type $I_3^*$ and $III^*$, and both need to be ramified by $\tau_{\cE}$.
\end{remark}

Now, suppose $\cE \colon X \to \PP^1$ is an elliptic fibration of type 1 in $(X,\sigma)$. Then, $\Lambda$ is a pencil of rational curves in  $\tilde{R}$.

\begin{prop}\label{equation_type1}
    Let the restriction of $\cE_R \colon \tilde{R} \to \PP^1$ to $D_v$ be given by the map
    \begin{align*}
        f_v \colon & D_v \to \PP^1 \\
        & P \mapsto [x_v(P):y_v(P)].
    \end{align*}
    Then, we can write the generic fibre of $\cE \colon X \to \PP^1$ as
    \begin{align*}
        \cE \colon s^3 x_v = t^3 y_v.
    \end{align*}
\end{prop}

\begin{proof}
    Let $F_v$ be a smooth fibre of $\cE$ such that $F_v$ is isomorphic to $\tilde{F}_v = \eta_X^{-1}(F_v)$. Notice that the following diagram commutes

    \begin{center}
        \begin{tikzcd}
            D_v \arrow[d, swap, "f_v"] & F_v \arrow[d,"\cE"] \arrow[l, "\tilde{\pi}"] \\
            \PP^1 & \PP^1. \arrow[l, "\tau"] 
        \end{tikzcd}
    \end{center}
    
    \sloppy
    By the universal property of fibre products, there is a 1-1 map $F_v \to {D_v \times_{\PP^1} \PP^1}$. Since $F_v$ is smooth, this must be an isomorphism. Then, we can write $F_v$ in coordinates $(P,[s:t])$, where $P$ is a point in $D_v$ and $[s:t]$ on $\PP^1$ such that $f_v(P) = \tau([s{:}t])$, that is, $[x_v(P){:}y_v(P)] = [s^3{:}t^3]$. This gives rise to the proposed equation for the generic fibre.
\end{proof}

\subsection{Conic bundles inducing elliptic fibrations}

By Theorem \ref{class_linearsys}, a Jacobian elliptic fibration of type 1 in $(X,\sigma)$ induces a conic bundle in $\tilde{R}$. On the other hand, assume $\cE_X \colon X \to \PP^1$ is obtained as the base change of a relatively minimal rational elliptic surface $\cE_R \colon R \to \PP^1$ by a cubic covering $\tau \colon \PP^1 \to \PP^1$ ramified at $a,b \in \PP^1$. Then, any conic bundle $\cB \colon R \to \PP^1$ induces a fibration $\cB_X \colon X \to \PP^1$. However, in general, $\cB_X$ is not an elliptic fibration. In fact, we can check that in general the fibres of $\cB_X$ have genus 2.

If $B_v$ is a smooth fibre of $\cB_X$, and $D_v$ is the corresponding fibre of $\cB$, then $B_v$ is isomorphic to the base change of $f_v \colon D_v \to \PP^1$ (given by the restriction of $\cE_R$ to $D_v$) by $\tau \colon \PP^1 \to \PP^1$. We can calculate the ramification of the map $B_v \to D_v$ in order to find the genus $g(B_v)$. By the definition of conic bundle class, we know that $D_v \cdot F = 2$, where $F$ is the class of fibres of $\cE_R$. Then, $f_v$ has degree 2, and ramifies at two distinct points $c,d \in \PP^1$. Assume that $a,b,c,d$ are all distinct. Then, there are distinct points $a_1,a_2,b_1,b_2 \in D_v$ such that $f_v(a_i) = a, f_v(b_i) = b$ for $i = 1,2$. Furthermore, let $a_0, b_0$ be points of $\PP^1$ such that $\tau(a_0) = a$ and $\tau(b_0) = b$. Then, $a_1, a_2, b_1, b_2$, are the ramification point of $B_v \to D_v$, each having a single point in its pre-image, given by $(a_1,a_0),(a_2,a_0),(b_1,b_0),(b_2,b_0)$ respectively. Using the Riemann-Hurwitz theorem, we can calculate that $g(D_v) = 2$. As a consequence, in general, a conic bundle in $R$ does not induce an elliptic fibration in $X$.

\begin{remark}
    In the case of quadratic base changes, every conic bundle in $R$ induces a Jacobian elliptic fibration in $X$. 
\end{remark}

\begin{prop}
    Let $\cB \colon R \to \PP^1$ be a conic bundle in $R$, and let $a,b$ be the points in $\PP^1$ where $\tau \colon \PP^1 \to \PP^1$ ramifies. Then, $\cB$ induces a Jacobian elliptic fibration in $X$ if and only if the map $f_v \colon C_v \to \PP^1$ given by the restriction of $\cE_R$ to the fibre $D_v := \cB^{-1}(v)$ ramifies in either $a$ or $b$ for every $v \in \PP^1$.
\end{prop}

\begin{proof}
    Suppose $f_v \colon D_v \to \PP^1$ ramifies in $a$ for every $v \in \PP^1$, and assume that the other ramification point is distinct from $b$. Then, the map $D_v \times_{\PP^1} \PP^1 =: B_v \to D_v$ ramifies in $a', b_1, b_2$, where $f_v(a') = a$ and $f_v(b_1) = f_v(b_2) = b$. Then applying the Riemann-Hurwitz theorem, we have $g(B_v) = 1$. On the other hand, if $\cB$ induces an elliptic fibration in $X$, then for all but finitely many $v \in \PP^1$ it is true that $g(B_v) = 1$, and by the Riemann-Hurwitz theorem the map $B_v \to D_v$ must ramify in 3 points. Let $c_v,d_v$ be the ramification points of $f_v$. If $c_v,d_v$ are distinct from $a,b$, then $B_v \to D_v$ would ramify in 4 distinct points, and $g(C'_v) = 2$. Then we can assume without loss of generality that $c_v = a$ for every $v \in \PP^1$.
\end{proof}

\begin{exemp}
    Let $\cE \colon R \to \PP^1$ be the rational elliptic surface induced by the pencil of cubics $\Lambda = s\cF + t \cG$ in $\PP^2$, where $\cF = y^2z - x^3 + xz^2 - 4z^3$ and $\cG = (x+z)(x-z)z$. The pencil $\Lambda_P = \alpha x - \beta z$ describes the lines of $\PP^2$ through the point $P = [0{:}1{:}0]$. Since $P$ is a base point of $\Lambda$, this induces a conic bundle $\cB \colon R \to \PP^1$. We want to show that $\cB$ defines an elliptic fibration on the base change of $X$ through the map $\tau([s{:}t]) = [s^3{:}t^3]$.

    Firstly, notice that $\tau$ ramifies in the point $[0{:}1], [1{:}0]$. For every $v = [\alpha{:}\beta] \in \PP^1$, the line $\alpha x = \beta z$ in $\Lambda_P$ is the image of the following map
    \begin{align*}
        \rho_{v} \colon \PP^1 &\longrightarrow \PP^2\\
        [u_1{:}u_2] &\longmapsto [\beta u_1 {:} u_2 {:} \alpha u_1].
    \end{align*}
    We can evaluate the map $f_v \colon C_v \to \PP^1$ as the resolution of the composition of $\rho_v$ with the cubic map $[x{:}y{:}z] \mapsto [\cF(x,y,z){:}\cG(x,y,z)]$, thus obtaining
    \begin{align*}
        f_v \colon \PP^1 &\longrightarrow \PP^1\\
        [u_1{:}u_2] & \longmapsto [\alpha u_2^2{+}(\alpha^2\beta{-}\beta^3{-}4\alpha^3)u_1^2 {:} (\alpha\beta^2{-}\alpha^3)u_1^2].
    \end{align*}
    Then, for each $v = [\alpha{:}\beta]$ the ramification points of $f_v$ are $[1{:}0]$ and $[\alpha^2{-}\beta^3{-}4\alpha^3{:}ab^2{-}a^3]$. Since $[1{:}0]$ is a ramification point for $\tau$ and for every $f_v$, $\varphi$ induces an elliptic fibration on the base change $X$ of $R$ by $\tau$.
\end{exemp}

\section{A \texorpdfstring{$\cJ_2$}{J2}-classification of K3 surfaces with non-symplectic automorphisms of order 3}\label{section_nishiyama}

In this section, our goal is to provide a $\cJ_2$-classification to K3 surfaces $X$ with non-symplectic automorphisms $\sigma$ of order 3. In order to make this classification complete, we add two restrictions to $X$. Firstly, we assume that the Picard number of $X$ is at least 12. Then, by Proposition \ref{picard_geq_12_works}, we can apply the Kneser--Nishiyama method (see Section \ref{nishiyama-method}). Secondly, we assume that $\sigma$ acts trivially in $\NS X$. Then, by work of Artebani and Sarti, $\NS X$ must be equal to one of 10 possible lattices, already assuming $\rho(X) \geq 12$ (see \cite{AS}, Proposition 3.2). For each possibility, Table \ref{List_K3_applied_nishiyama} shows explicitly the Néron--Severi and transcendental lattices, as well as the number $n$ of isolated fixed points of $\sigma$, the number $k$ of fixed curves of $\sigma$, and $g$ the greatest genus amongst the fixed curves.

\begin{table}[ht]\label{List_K3_applied_nishiyama}
    \centering
    \begin{tabular}{|c|c|c|c|c|c|}
        \hline
        No. & $\NS X $ & $T_X$ & $n$ & $k$ & $g$\\
        \hline
        1 & $U \oplus A_2^{\oplus 5}$ & $U \oplus U(3) \oplus A_2^{\oplus 3}$ & $5$ & $2$ & $0$\\
        \hline
        2 & $U \oplus E_6 \oplus A_2^{\oplus 2}$ & $U \oplus U(3) \oplus E_6$ & $5$ & $3$ & $1$\\
        \hline
        3 & $U \oplus E_8 \oplus A_2$ & $U^{\oplus 2} \oplus E_6$ & $5$ & $4$ & $2$\\
        \hline
        4 & $U \oplus E_6 \oplus A_2^{\oplus 3}$ & $U \oplus U(3) \oplus A_2^{\oplus 2}$ & $6$ & $3$ & $0$\\
        \hline
        5 & $U \oplus E_6^{\oplus 2}$ & $U^{\oplus 2} \oplus A_2^{\oplus 2}$ & $6$ & $4$ & $1$\\
        \hline
        6 & $U \oplus E_6^{\oplus 2} \oplus A_2$ & $U \oplus U(3) \oplus A_2$ & $7$ & $4$ & $0$\\
        \hline
        7 & $U \oplus E_6 \oplus E_8$ & $U^{\oplus 2} \oplus A_2$ & $7$ & $5$ & $1$\\
        \hline
        8 & $U \oplus E_6 \oplus E_8 \oplus A_2$ & $U \oplus U(3)$ & $8$ & $5$ & $0$\\
        \hline
        9 & $U \oplus E_8^{\oplus 2}$ & $U^{\oplus 2}$ & $8$ & $6$ & $1$\\
        \hline
        10 & $U \oplus E_8^{\oplus 2} \oplus A_2$ & $A_2(-1)$ & $9$ & $6$ & $0$\\
        \hline
    \end{tabular}
    \caption{Genera and lattices for each pair $(n,k)$}
    \label{order3_lattices}
\end{table}

\begin{thm}
    Let $X$ be a K3 surface with $\rho(X) \geq 12$ and $\sigma \in \Aut(X)$ a non-symplectic automorphism acting trivially on $\NS X$. Then, Table \ref{all_fibrations} describes the $\cJ_2$-classification of elliptic fibrations of $X$. Each fibration is given with its respective $ADE$-type $T$, and Mordell--Weil group $\MW(\cE)$.
\end{thm}

\begin{landscape}

\begin{table}[ht]\label{all_fibrations}
    \centering
        \caption{All possible elliptic fibrations on K3 surfaces with non-symplectic automorphisms of order 3}

        \textbf{1.} $\NS X = U \oplus A_2^{\oplus 5}$, $T_X = U \oplus U(3) \oplus A_2^{\oplus 3} $, $T_0 = E_6 \oplus A_2^{\oplus 4}$\\
        \vspace{3pt}
        \begin{tabular}{|>{\centering\arraybackslash}p{0.9cm}|>{\centering\arraybackslash}p{2.3cm}|>{\centering\arraybackslash}p{2.3cm}|>{\centering\arraybackslash}p{3cm}|>{\centering\arraybackslash}p{2.5cm}|>{\centering\arraybackslash}p{4.5cm}|>{\centering\arraybackslash}p{2.3cm}|>{\centering\arraybackslash}p{1.5cm}|}
        \hline
        No. & $L_\mathrm{root}$ & \multicolumn{3}{|c|}{Embeddings} & $M$ & $T = M_\mathrm{root}$ & $\MW(\cE)$\\
        \hline
        1.1 & $E_8^{\oplus 3}$ & $E_6 \subset E_8$ & $(A_2^{\oplus 2})^{\oplus 2} \subset E_8^{\oplus 2}$ & & $A_2^{\oplus 5}$ & $A_2^{\oplus 5}$ & 0\\
        \hline
        1.2 & $E_8 \oplus D_{16} $ & $E_6 \subset E_8$ & $A_2^{\oplus 4} \subset D_{16}$ & & $A_2 \oplus (A_2^{\oplus 4})^{\perp D_{16}}$ & $D_4 \oplus A_2$ & $\ZZ^4$\\
        \hline
        1.3 & $E_7^{\oplus 2} \oplus D_{10}$ & $E_6 \subset E_7$ & $A_2 \subset E_7$ & $A_2^{\oplus 3} \subset D_{10}$ & $(-6) \oplus A_5 \oplus (A_2^{\oplus 3})^{\perp D_{10}} $ & $A_5$ & $\ZZ^5$\\
        \hline
        1.4 & $E_7^{\oplus 2} \oplus D_{10}$ & $E_6 \subset E_7$ & $A_2^{\oplus 2} \subset E_7$ & $A_2^{\oplus 2} \subset D_{10}$ & $(-6)^{\oplus 2} \oplus A_2 \oplus (A_2^{\oplus 2})^{\perp D_{10}}$ & $D_4 \oplus A_2$ & $\ZZ^4$ \\
        \hline
        1.5 & $E_7 \oplus A_{17} $ & $E_6 \subset E_7$ & $A_2^{\oplus 4} \subset A_{17}$ & & $(-6) \oplus (A_2^{\oplus 4})^{\perp A_{17}} $ & $A_5$ & $\ZZ^5$\\
        \hline
        1.6 & $E_6^{\oplus 4}$ & $E_6 \subset E_6$ & $(A_2^{\oplus 2})^{\oplus 2} \subset E_6^{\oplus 2}$ & & $E_6 \oplus A_2^{\oplus 2} $ & $E_6 \oplus A_2^{\oplus 2}$ & 0\\
        \hline
        1.7 & $E_6^{\oplus 4}$ & $E_6 \subset E_6$ & $A_2^{\oplus 2} \subset E_6$ & $(A_2)^{\oplus 2} \subset E_6^{\oplus 2}$ & $A_2^{\oplus 5}$ & $A_2^{\oplus 5}$ & 0\\
        \hline
        1.8 & $E_6 \oplus D_7 \oplus A_{11}$ & $E_6 \subset E_6$ & $A_2^{\oplus 4} \subset A_{11}$ & & $D_7 \oplus (A_2^{\oplus 4})^{\perp A_{11}} $ & $D_7$ & $\ZZ^3$\\
        \hline
        1.9 & $E_6 \oplus D_7 \oplus A_{11}$ & $E_6 \subset E_6$ & $A_2 \subset D_7$ & $A_2^{\oplus 3} \subset A_{11}$ & $A_2^{\perp D_{7}} \oplus (A_2^{\oplus 3})^{\perp A_{11}} $ & $D_4 \oplus A_2$ & $\ZZ^4$\\
        \hline
        1.10 & $E_6 \oplus D_7 \oplus A_{11}$ & $E_6 \subset E_6$ & $A_2^{\oplus 2} \subset D_7$ & $A_2^{\oplus 2} \oplus A_{11}$ & $(A_2^{\oplus 2})^{\perp D_7} \oplus (A_2^{\oplus 2})^{\perp A_{11}}$ & $A_5$ & $\ZZ^5$\\
        \hline
        \end{tabular}\\
        \vspace{5pt}

        \textbf{2.} $\NS X = U \oplus E_6 \oplus A_2^{\oplus 2}$, $T_X = U \oplus U(3) \oplus E_6$, $T_0 = E_6^{\oplus 2} \oplus A_2$\\
        \vspace{3pt}
        \begin{tabular}{|>{\centering\arraybackslash}p{0.9cm}|>{\centering\arraybackslash}p{2.3cm}|>{\centering\arraybackslash}p{2.3cm}|>{\centering\arraybackslash}p{3cm}|>{\centering\arraybackslash}p{2.5cm}|>{\centering\arraybackslash}p{4.5cm}|>{\centering\arraybackslash}p{2.3cm}|>{\centering\arraybackslash}p{1.5cm}|}
        \hline
        No. & $L_\mathrm{root}$ & \multicolumn{3}{|c|}{Embeddings} & $M$ & $T = M_\mathrm{root}$ & $\MW(\cE)$\\
        \hline
        2.1 & $E_8^{\oplus 3}$ & $E_6^{\oplus 2} \subset E_8^{\oplus 2}$ & $A_2 \subset E_8$ & & $E_6 \oplus A_2^{\oplus 2}$ & $E_6 \oplus A_2^{\oplus 2}$ & 0\\
        \hline
        2.2 & $E_7^{\oplus 2} \oplus D_{10}$ & $E_6^{\oplus 2} \subset E_7^{\oplus 2}$ & $A_2 \subset D_{10}$ & & $(-6)^{\oplus 2} \oplus A_2^{\perp D_{10}}$ & $D_7$ & $\ZZ^3$\\
        \hline
        2.3 & $E_6^{\oplus 4}$ & $E_6^{\oplus 2} \subset E_6^{\oplus 2}$ & $A_2 \subset E_6$ & & $E_6 \oplus A_2^{\oplus 2}$ & $E_6 \oplus A_2^{\oplus 2}$ & 0\\
        \hline
        \end{tabular}\\
        \vspace{5pt}

        \textbf{3.} $\NS X = U \oplus E_8 \oplus A_2$, $T_X = U^{\oplus 2} \oplus E_6$, $T_0 = E_8 \oplus E_6$\\
        \vspace{3pt}
        \begin{tabular}{|>{\centering\arraybackslash}p{0.9cm}|>{\centering\arraybackslash}p{2.3cm}|>{\centering\arraybackslash}p{2.3cm}|>{\centering\arraybackslash}p{3cm}|>{\centering\arraybackslash}p{2.5cm}|>{\centering\arraybackslash}p{4.5cm}|>{\centering\arraybackslash}p{2.3cm}|>{\centering\arraybackslash}p{1.5cm}|}
        \hline
        No. & $L_\mathrm{root}$ & \multicolumn{3}{|c|}{Embeddings} & $M$ & $T = M_\mathrm{root}$ & $\MW(\cE)$\\
        \hline
        3.1 & $E_8^{\oplus 3}$ & $E_8 \subset E_8$ & $E_6 \subset E_8$ & & $E_8 \oplus A_2$ & $E_8 \oplus A_2$ & 0\\
        \hline
        \end{tabular}\\
        \vspace{5pt}

        \textbf{4.} $\NS X = U \oplus E_6 \oplus A_2^{\oplus 3}$, $T_X = U \oplus U(3) \oplus A_2^{\oplus 2}$, $T_0 = E_6 \oplus A_2^{\oplus 3}$\\
        \vspace{5pt}
        \begin{tabular}{|>{\centering\arraybackslash}p{0.9cm}|>{\centering\arraybackslash}p{2.3cm}|>{\centering\arraybackslash}p{2.3cm}|>{\centering\arraybackslash}p{3cm}|>{\centering\arraybackslash}p{2.5cm}|>{\centering\arraybackslash}p{4.5cm}|>{\centering\arraybackslash}p{2.3cm}|>{\centering\arraybackslash}p{1.5cm}|}
        \hline
        No. & $L_\mathrm{root}$ & \multicolumn{3}{|c|}{Embeddings} & $M$ & $T = M_\mathrm{root}$ & $\MW(\cE)$\\
        \hline
        4.1 & $E_8^{\oplus 3}$ & $E_6 \subset E_8$ & $A_2^{\oplus 2} \subset E_8$ & $A_2 \subset E_8$ & $E_6 \oplus A_2^{\oplus 3}$ & $E_6 \oplus A_2^{\oplus 3}$ & 0\\
        \hline
        4.2 & $E_8 \oplus D_{16}$ & $E_6 \subset E_8$ & $A_2^{\oplus} \subset D_{16}$ & & $A_2 \oplus (A_2^{\oplus 3})^{\perp D_{16}} $ & $D_7 \oplus A_2$ & $\ZZ^3$\\
        \hline
        4.3 & $E_7^{\oplus 2} \oplus D_{10}$ & $E_6 \oplus E_7$ & $A_2 \oplus D_{10}$ & & $(-6) \oplus E_7 \oplus (A_2^{\oplus 3})^{D_{10}}$ & $E_7$ & $\ZZ^5$\\
        \hline
        \end{tabular}\\
\end{table}

\newpage

\begin{table}[ht]
    \centering    
        \begin{tabular}{|>{\centering\arraybackslash}p{0.9cm}|>{\centering\arraybackslash}p{2.3cm}|>{\centering\arraybackslash}p{2.3cm}|>{\centering\arraybackslash}p{3cm}|>{\centering\arraybackslash}p{2.5cm}|>{\centering\arraybackslash}p{4.5cm}|>{\centering\arraybackslash}p{2.3cm}|>{\centering\arraybackslash}p{1.5cm}|}
        \hline
        4.4 & $E_7^{\oplus 2} \oplus D_{10}$ & $E_6 \subset E_7$ & $A_2 \subset E_7$ & $A_2^{\oplus 2} \subset D_{10}$ & $(-6) \oplus A_5 \oplus (A_2^{\oplus 2})^{\perp D_{10}}$ & $D_4 \oplus A_5$ & $\ZZ^3$\\
        \hline
        4.5 & $E_7^{\oplus 2} \oplus D_{10}$ & $E_6 \subset E_7$ & $A_2^{\oplus 2} \subset E_7$ & $A_2 \subset D_{10}$ & $(-6)^{\oplus} \oplus A_2 \oplus A_2^{\perp D_{10}}$ & $D_7 \oplus A_2$ & $\ZZ^3$\\
        \hline
        4.6 & $E_7 \oplus A_{17}$ & $E_6 \subset E_7$ & $A_2^{\oplus 3} \subset A_{17} $ & & $(-6) \oplus (A_2^{\oplus 3})^{\perp A_{17}} $ & $A_8$ & $\ZZ^4$\\
        \hline
        4.7 & $E_6^{\oplus 4}$ & $E_6 \subset E_6$ & $A_2^{\oplus 2} \subset E_6$ & $A_2 \subset E_6$ & $E_6 \oplus A_2^{\oplus 3}$ & $E_6 \oplus A_2^{\oplus 3}$ & 0\\
        \hline
        4.8 & $E_6^{\oplus 4}$ & $E_6 \subset E_6$ & $(A_2)^{\oplus 3} \subset E_6^{\oplus 3}$ & & $A_2^{\oplus 6}$ & $A_2^{\oplus 6}$ & $\ZZ/3$\\
        \hline
        4.9 & $E_6 \oplus D_7 \oplus A_{11}$ & $E_6 \subset E_6$ & $A_2^{\oplus 3} \subset A_{11}$ & & $D_7 \oplus (A_2^{\oplus})^{\perp A_{11}} $ & $D_7 \oplus A_2$ & $\ZZ^3$\\
        \hline
        4.10 & $E_6 \oplus D_7 \oplus A_{11}$ & $E_6 \subset E_6$ & $A_2 \subset D_7$ & $A_2^{\oplus 2} \subset A_{11}$ & $A_2^{\perp D_{7}} \oplus (A_2^{\oplus 2})^{\perp A_{11}}$ & $D_4 \oplus A_5$ & $\ZZ^3$\\
        \hline
        4.11 & $E_6 \oplus D_7 \oplus A_{11}$ & $E_6 \subset E_6$ & $A_2^{\oplus 2} \subset D_7$ & $A_2 \subset A_{11}$ & $(A_2^{\oplus 2})^{\perp D_7} \oplus A_2^{\perp A_{11}}$ & $A_8$ & $\ZZ^4$\\
        \hline
        \end{tabular}\\
        \vspace{5pt}

        \textbf{5.} $\NS X = U \oplus E_6^{\oplus 2}$, $T_X = U^{\oplus 2} \oplus A_2^{\oplus 2}$, $T_0 = E_8 \oplus A_2^{\oplus 2}$\\
        \vspace{3pt}
        \begin{tabular}{|>{\centering\arraybackslash}p{0.9cm}|>{\centering\arraybackslash}p{2.3cm}|>{\centering\arraybackslash}p{2.3cm}|>{\centering\arraybackslash}p{3cm}|>{\centering\arraybackslash}p{2.5cm}|>{\centering\arraybackslash}p{4.5cm}|>{\centering\arraybackslash}p{2.3cm}|>{\centering\arraybackslash}p{1.5cm}|}
        \hline
        No. & $L_\mathrm{root}$ & \multicolumn{3}{|c|}{Embeddings} & $M$ & $T = M_\mathrm{root}$ & $\MW(\cE)$\\
        \hline
        5.1 & $E_8^{\oplus 3}$ & $E_8 \subset E_8$ & $A_2^{\oplus 2} \subset E_8$ & & $E_8 \oplus A_2^{\oplus 2}$ & $E_8 \oplus A_2^{\oplus 2}$ & 0\\
        \hline
        5.2 & $E_8^{\oplus 3}$ & $E_8 \subset E_8$ & $(A_2)^{\oplus 2} \subset E_8^{\oplus 2}$ & & $E_6^{\oplus 2}$ & $E_6^{\oplus 2}$ & 0\\
        \hline
        5.3 & $E_8 \oplus D_{16}$ & $E_8 \subset E_8$ & $A_2^{\oplus 2} \subset D_{16}$ & & $(A_2^{\oplus 2})^{\perp D_{16}}$ & $D_{10}$ & $\ZZ^2$\\
        \hline
        \end{tabular}\\
        \vspace{5pt}

        \textbf{6.} $\NS X = U \oplus E_6^{\oplus 2} \oplus A_2$, $T_X = U \oplus U(3) \oplus A_2$, $T_0 = E_6 \oplus A_2^{\oplus 2}$\\
        \vspace{3pt}
        \begin{tabular}{|>{\centering\arraybackslash}p{0.9cm}|>{\centering\arraybackslash}p{2.3cm}|>{\centering\arraybackslash}p{2.3cm}|>{\centering\arraybackslash}p{3cm}|>{\centering\arraybackslash}p{2.5cm}|>{\centering\arraybackslash}p{4.5cm}|>{\centering\arraybackslash}p{2.3cm}|>{\centering\arraybackslash}p{1.5cm}|}
        \hline
        No. & $L_\mathrm{root}$ & \multicolumn{3}{|c|}{Embeddings} & $M$ & $T = M_\mathrm{root}$ & $\MW(\cE)$\\
        \hline
        6.1 & $E_8^{\oplus 3}$ & $E_6 \subset E_8$ & $A_2^{\oplus 2} \subset E_8$ & & $E_8 \oplus A_2^{\oplus 3}$ & $E_8 \oplus A_2^{\oplus 3}$ & 0\\
        \hline
        6.2 & $E_8^{\oplus 3}$ & $E_6 \subset E_8$ & $(A_2)^{\oplus 2} \subset E_8^{\oplus 2}$ & & $E_6^{\oplus 2} \oplus A_2$ & $E_6^{\oplus 2} \oplus A_2$ & 0\\
        \hline
        6.3 & $E_8 \oplus D_{16}$ & $E_6 \subset E_8$ & $A_2^{\oplus 2} \subset D_{16}$ & & $A_2 \oplus (A_2^{\oplus 2})^{\perp D_{16}}$ & $D_{10} \oplus A_2$ & $\ZZ^2$\\
        \hline
        6.4 & $E_7^{\oplus 2} \oplus D_{10}$ & $E_6 \subset E_7$ & $A_2^{\oplus 2} \subset E_7$ & & $(-6)^{\oplus 2} \oplus D_{10} \oplus A_2$ & $D_{10} \oplus A_2$ & $\ZZ^2$\\
        \hline
        6.5 & $E_7^{\oplus 2} \oplus D_{10}$ & $E_6 \subset E_7$ & $A_2 \subset E_7$ & $A_2 \subset D_{10}$ & $(-6)\oplus A_5 \oplus A_2^{\perp D_{10}} $ & $D_7 \oplus A_5$ & $\ZZ^2$\\
        \hline
        6.6 & $E_7^{\oplus 2} \oplus D_{10}$ & $E_6 \subset E_7$ & $A_2^{\oplus 2} \subset D_{10}$ & & $(-6) \oplus E_7 \oplus (A_2^{\oplus 2})^{\perp D_{10}} $ & $E_7 \oplus D_4$ & $\ZZ^3$ \\
        \hline
        6.7 & $E_7 \oplus A_{17}$ & $E_6 \subset E_7$ & $A_2^{\oplus 2} \subset A_{17}$ & & $(-6) \oplus (A_2^{\oplus 2})^{\perp A_{17}}$ & $A_{11}$ & $\ZZ^3$ \\
        \hline
        6.8 & $E_6^{\oplus 4}$ & $E_6 \subset E_6$ & $A_2^{\oplus} \subset E_6$ & & $E_6^{\oplus 2} \oplus A_2$ & $E_6^{\oplus 2} \oplus A_2$ & 0\\
        \hline
        6.9 & $E_6^{\oplus 4}$ & $E_6 \subset E_6$ & $A_2 \subset E_6$ & $A_2 \subset E_6$ & $E_6 \oplus A_2^{\oplus 4}$ & $E_6 \oplus A_2^{\oplus 4}$ & $\ZZ/3$\\
        \hline
        6.10 & $E_6 \oplus D_7 \oplus A_{11}$ & $E_6 \subset E_6$ & $A_2^{\oplus 2} \subset D_7$ & & $(A_2^{\oplus 2})^{\perp D_7} \oplus A_{11}$ & $A_{11}$ & $\ZZ^3$\\
        \hline
        6.11 & $E_6 \oplus D_7 \oplus A_{11}$ & $E_6 \subset E_6$ & $A_2 \subset D_7$ & $A_2 \subset A_{11}$ & $A_2^{\perp D_{7}} \oplus A_2^{\perp A_{11}}$ & $D_4 \oplus A_8$ & $\ZZ^2$\\
        \hline
        6.12 & $E_6 \oplus D_7 \oplus A_{11}$ & $E_6 \subset E_6$ & $A_2^{\oplus 2} \subset A_{11}$ & & $D_7 \oplus (A_2^{\oplus 2})^{\perp A_{11}} $ & $D_7 \oplus A_5$ & $\ZZ^2$\\
        \hline
        \end{tabular}
    \end{table}

    \newpage

\begin{table}
    \centering
        \textbf{7.} $\NS X = U \oplus E_6 \oplus E_8$, $T_X = U^{\oplus 2} \oplus A_2$, $T_0 = E_8 \oplus A_2$\\
        \vspace{3pt}
        \begin{tabular}{|>{\centering\arraybackslash}p{0.9cm}|>{\centering\arraybackslash}p{2.3cm}|>{\centering\arraybackslash}p{2.3cm}|>{\centering\arraybackslash}p{3cm}|>{\centering\arraybackslash}p{2.5cm}|>{\centering\arraybackslash}p{4.5cm}|>{\centering\arraybackslash}p{2.3cm}|>{\centering\arraybackslash}p{1.5cm}|}
        \hline
        No. & $L_\mathrm{root}$ & \multicolumn{3}{|c|}{Embeddings} & $M$ & $T = M_\mathrm{root}$ & $\MW(\cE)$\\
        \hline
        7.1 & $E_8^{\oplus 3}$ & $E_8 \subset E_8$ & $A_2 \subset E_8$ & & $E_6 \oplus E_8$ & $E_6 \oplus E_8$ & 0\\
        \hline
        7.2 & $E_8 \oplus D_{16}$ & $E_8 \subset E_8$ & $A_2 \subset D_{16}$ & & $A_2^{\perp D_{16}}$ & $D_{13}$ & $\ZZ$\\
        \hline
        \end{tabular}\\
        \vspace{5pt}

        \textbf{8.} $\NS X = U \oplus E_6 \oplus E_8 \oplus A_2$, $T_X = U \oplus U(3)$, $T_0 = E_6 \oplus A_2$\\
        \vspace{3pt}
        \begin{tabular}{|>{\centering\arraybackslash}p{0.9cm}|>{\centering\arraybackslash}p{2.3cm}|>{\centering\arraybackslash}p{2.3cm}|>{\centering\arraybackslash}p{3cm}|>{\centering\arraybackslash}p{2.5cm}|>{\centering\arraybackslash}p{4.5cm}|>{\centering\arraybackslash}p{2.3cm}|>{\centering\arraybackslash}p{1.5cm}|}
        \hline
        No. & $L_\mathrm{root}$ & \multicolumn{3}{|c|}{Embeddings} & $M$ & $T = M_\mathrm{root}$ & $\MW(\cE)$\\
        \hline
        8.1 & $E_8^{\oplus 3}$ & $E_6 \subset E_8$ & $A_2 \subset E_8$ & & $E_6 \oplus E_8 \oplus A_2$ & $E_6 \oplus E_8 \oplus A_2$ & 0\\
        \hline
        8.2 & $E_8 \oplus D_{16}$ & $E_6 \subset E_8$ & $A_2 \subset D_{16}$ & & $A_2 \oplus A_2^{\perp D_{16}}$ & $D_{13} \oplus A_2$ & $\ZZ$\\
        \hline
        8.3 & $E_7^{\oplus 2} \oplus D_{10}$ & $E_6 \subset E_7$ & $A_2 \subset E_7$ & & $(-6) \oplus A_5 \oplus D_{10} $ & $D_{10} \oplus A_5$ & $\ZZ \oplus \ZZ/2$\\
        \hline
        8.4 & $E_7^{\oplus 2} \oplus D_{10}$ & $E_6 \subset E_7$ & $A_2 \subset D_{10}$ & & $(-6) \oplus E_7 \oplus A_2^{\perp D_{10}}$ & $E_7 \oplus D_7$ & $\ZZ^2$\\
        \hline
        8.5 & $E_7 \oplus A_{17}$ & $E_6 \subset E_7$ & $A_2 \subset A_{17}$ & & $(-6) \oplus A_2^{\perp A_{17}}$ & $A_{14}$ & $\ZZ^2$\\
        \hline
        8.6 & $E_6^{\oplus 4}$ & $E_6 \subset E_6$ & $A_2 \subset E_6$ & & $E_6^{\oplus 2} \oplus A_2^{\oplus 2}$ & $E_6^{\oplus 2} \oplus A_2^{\oplus 2}$ & $\ZZ/3$\\
        \hline
        8.7 & $E_6 \oplus D_7 \oplus A_{11}$ & $E_6 \subset E_6$ & $A_2 \subset D_7$ & & $A_2^{\perp D_7} \oplus A_{11}$ & $D_4 \oplus A_{11} $ & $\ZZ$\\
        \hline
        8.8 & $E_6 \oplus D_7 \oplus A_{11}$ & $E_6 \subset E_6$ & $A_2 \subset A_{11}$ & & $D_7 \oplus A_2^{\perp A_{11}}$ & $D_7 \oplus A_8$ & $\ZZ$\\
        \hline
        \end{tabular}\\
        \vspace{5pt}

        \textbf{9.} $\NS X = U \oplus E_8^{\oplus 2}$, $T_X = U^{\oplus 2}$, $T_0 = E_8$\\
        \vspace{3pt}
        \begin{tabular}{|>{\centering\arraybackslash}p{0.9cm}|>{\centering\arraybackslash}p{2.3cm}|>{\centering\arraybackslash}p{2.3cm}|>{\centering\arraybackslash}p{3cm}|>{\centering\arraybackslash}p{2.5cm}|>{\centering\arraybackslash}p{4.5cm}|>{\centering\arraybackslash}p{2.3cm}|>{\centering\arraybackslash}p{1.5cm}|}
        \hline
        No. & $L_\mathrm{root}$ & \multicolumn{3}{|c|}{Embeddings} & $M$ & $T = M_\mathrm{root}$ & $\MW(\cE)$\\
        \hline
        9.1 & $E_8^{\oplus 3}$ & $E_8 \subset E_8$ & & & $E_8^{\oplus 2}$ & $E_8^{\oplus 2}$ & 0\\
        \hline
        9.2 & $E_8 \oplus D_{16}$ & $E_8 \subset E_8$ & & & $D_{16}$ & $D_{16}$ & $\ZZ/2$\\
        \hline
        \end{tabular}\\
        \vspace{5pt}

        \textbf{10.} $\NS X = U \oplus E_8^{\oplus 2} \oplus A_2$, $T_X = A_2(-1)$, $T_0 = E_6$\\
        \vspace{3pt}
        \begin{tabular}{|>{\centering\arraybackslash}p{0.9cm}|>{\centering\arraybackslash}p{2.3cm}|>{\centering\arraybackslash}p{2.3cm}|>{\centering\arraybackslash}p{3cm}|>{\centering\arraybackslash}p{2.5cm}|>{\centering\arraybackslash}p{4.5cm}|>{\centering\arraybackslash}p{2.3cm}|>{\centering\arraybackslash}p{1.5cm}|}
        \hline
        No. & $L_\mathrm{root}$ & \multicolumn{3}{|c|}{Embeddings} & $M$ & $T = M_\mathrm{root}$ & $\MW(\cE)$\\
        \hline
        10.1 & $E_8^{\oplus 3}$ & $E_6 \subset E_8$ & & & $E_8^{\oplus 2} \oplus A_2$ & $E_8^{\oplus 2} \oplus A_2$ & 0\\
        \hline
        10.2 & $E_8 \oplus D_{16}$ & $E_6 \subset E_8$ & & & $D_{16} \oplus A_2$ & $D_{16} \oplus A_2$ & $\ZZ/2$\\
        \hline
        10.3 & $E_7^{\oplus 2} \oplus D_{10}$ & $E_6 \subset E_7$ & & & $(-6) \oplus E_7 \oplus D_{10}$ & $E_7 \oplus D_{10}$ & $\ZZ \oplus \ZZ/2$\\
        \hline
        10.4 & $E_7 \oplus A_{17}$ & $E_6 \subset E_7$ & & & $(-6) \oplus A_{17}$ & $A_{17}$ & $\ZZ \oplus \ZZ/3$\\
        \hline
        10.5 & $E_6^{\oplus 4} $ & $E_6 \subset E_6$ & & & $E_6^{\oplus 3}$ & $E_6^{\oplus 3}$ & $\ZZ/3$\\
        \hline
        10.6 & $E_6 \oplus D_7 \oplus A_{11}$ & $E_6 \subset E_6$ & & & $D_7 \oplus A_{11}$ & $D_7 \oplus A_{11}$ & $\ZZ/4$\\
        \hline
        \end{tabular}\\

\end{table}

\end{landscape}

\newpage

\begin{proof}
    The proof of this theorem consists on a direct application of Nishiyama's method (see \cite{Nishiyama1996}). An overview of this method is described in the end of Section \ref{elliptic_fibrations_k3}, and the explicit computations for the required cases are presented in Section \ref{application_nishiyama}.
\end{proof}

\begin{remark}\label{remark_10isX3}
    The surface 10 in Table \ref{List_K3_applied_nishiyama} is denoted as the $X_3$ surface. It was first studied by Shioda and Inose in \cite{Shioda-Inose}, and subsequently by Vinberg in \cite{Vinberg}, where it was first denoted by $X_3$ and described as one of the most algebraic K3 surfaces. The $X_3$ surface is constructed as the minimal resolution of $(E_{\zeta_3} \times E_{\zeta_3})/{\rho}$, where $\zeta_3$ is a primitive cube root of unity, $E_{\zeta_3}$ the elliptic curve with torus $\CC/(\ZZ + \zeta_3\ZZ)$ and $\rho$ the automorphism given by $\rho(z_1,z_2) = (\zeta_3 z_1, \zeta_3^{-1} z_2)$. In \cite{Nishiyama1996}, the Kneser--Nishiyama method was applied to obtain a full $\cJ_2$-classification of $X_3$.
\end{remark}

\subsection{Application of the Kneser--Nishiyama Method}\label{application_nishiyama}

Table \ref{T_0-table} shows the choice of $T_0$ for each $T_X$ in Table \ref{order3_lattices}.
\begin{table}[ht]
    \centering
    \begin{tabular}{|c|c|c|c|c|}
        \hline
        No. & $T_X$ & $G_{T_X}$ & $T_0$\\
        \hline
        1 & $U \oplus U(3) \oplus A_2^{\oplus 3}$ & $(\ZZ/3\ZZ)^5$ & $E_6 \oplus A_2^{\oplus 4}$\\
        \hline
        2 & $U \oplus U(3) \oplus E_6$ & $(\ZZ/3\ZZ)^3$ & $E_6^{\oplus 2} \oplus A_2$\\
        \hline
        3 & $U^{\oplus 2} \oplus E_6$ & $\ZZ/3\ZZ$ & $E_8 \oplus E_6$\\
        \hline
        4 & $U \oplus U(3) \oplus A_2^{\oplus 2}$ & $(\ZZ/3\ZZ)^4$ & $E_6 \oplus A_2^{\oplus 3}$\\
        \hline
        5 & $U^{\oplus 2} \oplus A_2^{\oplus 2}$ & $(\ZZ/3\ZZ)^2$ & $E_8 \oplus A_2^{\oplus 2}$\\
        \hline
        6 & $U \oplus U(3) \oplus A_2$ & $(\ZZ/3\ZZ)^3$ & $E_6 \oplus A_2^{\oplus 2}$\\
        \hline
        7 & $U^{\oplus 2} \oplus A_2$ & $\ZZ/3\ZZ$ & $E_8 \oplus A_2$\\
        \hline
        8 & $U \oplus U(3)$ & $(\ZZ/3\ZZ)^2$ & $E_6 \oplus A_2$\\
        \hline
        9 & $U^{\oplus 2}$ & $\{e\}$ & $E_8$\\
        \hline
        10 & $A_2(-1)$ & $\ZZ/3\ZZ$ & $E_6$\\
        \hline
    \end{tabular}
    \caption{$T_0$ for each surface $X$}
    \label{T_0-table}
\end{table}

\begin{prop}
    For every $T_X$, $T_0$ in Table \ref{T_0-table}, $G_{T_X} = G_{T_0}$ and $q_{T_0} = q_{T_X}$.
\end{prop}

\begin{proof}
    Firstly, observe that for any $L_1, L_2$, $G_{L_1 \oplus L_2} = G_{L_1} \times G_{L_2}$ and $q_{L_1 \oplus L_2} = q_{L_1} + q_{L_2}$. This reduces the required calculations to the cases $T_X = A_2(-1), U^{\oplus 2}, U \oplus U(3)$ and $T_0 = E_6, E_8, E_6 \oplus A_2$, respectively. The lattices $G_{T_X}$ are given in \cite{AS}, Lemma 1.3 and Table 1, and both $G_{T_0}$ and $q_{T_0}$ can be calculated using Lemma 1.2 in \cite{Nishiyama1996}.

    For $T_X = A_2(-1)$, let $a_1, a_2$ generate $T_X$, with $a_1^2 = a_2^2 = 2$, $a_1 \cdot a_2 = -1$. Then the discriminant group $G_{T_X}$ is generated by $w = \tfrac{2}{3}a_1 + \tfrac{1}{3}a_2$, and $q_{T_X}(w) \equiv -\tfrac{4}{2} \mod 2\ZZ$.
    
    For $T_X = U^{\oplus 2}$, $G_{T_X}$ is trivial, so $q_{T_X} = 0$. 
    
    For $T_X = U \oplus U(3)$, let $u_1, u_2, u'_1, u'_2$ be its generators, with $u_i^2 = (u'_i)^2 = u_i \cdot u'_j = 0$, $u_1 \cdot u_2 = 1$, and $u'_1 \cdot u'_2 = 3$. Then $G_{T_X}$ is generated by $w_1 = \tfrac{1}{3} u'_1 + \tfrac{1}{3}u'_2$ and $w_2 = \tfrac{2}{3} u'_1 + \tfrac{1}{3} u'_2$. We check that $q_{T_X}(w_1) \equiv - \tfrac{4}{3} \mod 2\ZZ$ and $q_{T_X}(w_2) \equiv - \tfrac{2}{3} \mod 2\ZZ$. 
    
    Using Lemma 1.2 in \cite{Nishiyama1996}, we can check that the values of $G_{T_0}$ and $q_{T_0}$ agree with the ones calculated above.
\end{proof}

Next, we move to the calculation of primitive embeddings of $A^{\oplus \ell}$ into $A_n$ or $D_n$, for $\ell \geq 2$, and their orthogonal complements. Firstly, we prove the following lemma.

\begin{lema} \label{lattice_lemma}
    If $N_1, N_2$ are sublattices of a lattice $L$, $N_2$ is of root type and $N_2 \subseteq N_1^{\perp L}$, then $((N_1 \oplus N_2)^{\perp L})_\mathrm{root} = (N_2^{\perp (N_{1}^{\perp L})_\mathrm{root}})_\mathrm{root}.$
\end{lema}

\begin{proof}
    Since $N_2$ is of root type, $N_2 \subseteq (N_1^{\perp L})_\mathrm{root} \subseteq N_1^{\perp L}$. Then,
    \begin{equation*}
        N_2^{\perp (N_1^{\perp L})_\mathrm{root}} \subseteq N_2^{\perp (N_1^{\perp L})} = (N_1 \oplus N_2)^{\perp L}.
    \end{equation*}
    Taking the root type of both sides, we get $(N_2^{\perp (N_1^{\perp L})_\mathrm{root}})_\mathrm{root} \subseteq ((N_1 \oplus N_2)^{\perp L})_\mathrm{root}$.
    Now suppose $x \in ((N_1 \oplus N_2)^{\perp L})_\mathrm{root}$. Then, $x$ is generated by roots and $\langle x, n_1 \rangle = 0$ for all $n_1 \in N_1$, so by definition $x \in (N_1^{\perp L})_\mathrm{root}$. Since $\langle x, n_2 \rangle = 0$ for all $n_2 \in N_2$, it is true that $x \in (N_2^{\perp (N_1^{\perp L})_\mathrm{root}})_\mathrm{root}$. This gets us to the result.
\end{proof}

\begin{prop}\label{proposition-embedding}
    The primitive embeddings of $A_2^{\oplus \ell}$ in $A_n$ or $D_n$, up to an action of their Weyl group, are as follows.
    \begin{itemize}
        \item[i)] There is a unique embedding given by ${A_2^{\oplus \ell} = \bigoplus_{i=0}^{\ell-1} \langle a_{3i+1},a_{3i+2}\rangle \subset A_n}$ for $n \geq 3\ell-1$. Furthermore, the orthogonal of this embedding is 
        \begin{equation*}
            (A_2^{\oplus \ell})^{\perp A_n}_\mathrm{root} = \begin{cases}
                0 & \text{\normalfont if } 3\ell-1 \leq n \leq 3\ell,\\
                A_{n-3\ell} & \text{\normalfont if } n \geq 3\ell+1.
            \end{cases}
        \end{equation*}
        \item[ii)] There is a unique embedding given by $A_2^{\oplus \ell} = \bigoplus_{i=1}^{\ell} \langle d_{3\ell-1},d_{3\ell} \rangle \subset D_n$ for $n \geq 3\ell$. Furthermore, the orthogonal of this embedding is
        \begin{equation*}
            (A_2^{\oplus \ell})^{\perp D_n}_\mathrm{root} = \begin{cases}
                0 & \text{\normalfont if } 3\ell \leq n \leq 3\ell+1,\\
                A_1^{\oplus 2} & \text{\normalfont if } n = 3\ell+2,\\
                A_3 & \text{\normalfont if } n = 3\ell+3,\\
                D_{n-3\ell} & \text{\normalfont if } n \geq 3\ell+4.
            \end{cases}
        \end{equation*}
        
    \end{itemize}
\end{prop}

\begin{proof}
    We prove this by induction. The case $\ell = 1$ is proved in \cite{Nishiyama1996} (see Lemmas 4.1 and 4.2 for the primitive embeddings of $A_2$ in $A_n$ and $D_n$ respectively, and Corollary 4.4 for the orthogonal lattice). Suppose this is true for $\ell$. Then, for $n \geq 3\ell+2$, there is a unique primitive embedding of $A_2^{\oplus \ell}$ in $A_n$, and the orthogonal lattice is $A_{n-3\ell} = \langle a_{3\ell+1},{.}{.}{.}, a_n \rangle$. We know that $A_2 = \langle a_{3\ell+1}, a_{3\ell+2} \rangle$ is the unique primitive embedding of $A_2$ in $A_{n-3\ell}$, up to an action of $W(A_{n-3\ell}) \subset W(A_n)$. Glueing together both embeddings, we obtain a unique primitive embedding of $A_2^{\oplus \ell+1}$ in $A_n$ up to an action of $W(A_n)$. By Lemma \ref{lattice_lemma}, $(A_2^{\oplus \ell})^{\perp A_n}_\mathrm{root} = (A_2^{\perp A_{n-3\ell}})_\mathrm{root}$. This is equal to $0$ if $n = 3\ell+2$ or $3\ell+3$, and to $A_{n-3\ell-3}$ if $n \geq 3\ell+4$.

    For $n \geq 3\ell+3$, we know that there is a unique primitive embedding of $A_2^{\oplus \ell}$ in $D_n$, and the orthogonal is equal to $A_3$ if $n = 3\ell+3$, and to $D_{n-3\ell}$ if $n \geq 3\ell+4$. In either case, there is a unique primitive embedding of $A_2$ up to an action of the Weyl group, thus obtaining an embedding of $A_2^{\oplus \ell+1}$ in $D_n$. We can apply Lemma $\ref{lattice_lemma}$ to obtain the result.
\end{proof}

Using this result, we are able to explicitly show the generators of orthogonal complements of primitive embeddings of a lattice $N$ into a Niemeier lattice $L$, as well as their root types, when $N$ is a sum of $A_2, E_6$ and $E_8$ (see Table \ref{embeddings} in Appendix \ref{appendix}).

Finally, we move to the calculation of the torsion subgroup of the Mordell--Weil group of each fibration. In \cite{Shimada}, the possible torsion subgroups of each $ADE$-type are shown (see Table 1 on the Appendix). For most of the fibrations in Table \ref{all_fibrations}, there is only one admissible torsion subgroup, and thus there is no need to calculate $\overline{W_\mathrm{root}}/W_\mathrm{root}$ explicitly. The fibrations with multiple admissible torsion subgroups are 4.8, 6.9, 8.3, 8.6, 8.7, 9.2, 10.3 and 10.4. When $\rank \MW(\cE) = 0$, the torsion subgroup is isomorphic to $W/M$ (see \cite{Nishiyama1996}, Lemma 6.6), and the torsion can be determined by calculating $\det T_0 = \det W$ and $\det M$. This allows us to determine the torsion for fibrations 4.8, 6.9, 8.6 and 9.2. For 10.3 and 10.4, the torsion subgroups were already calculated in \cite{Nishiyama1996} (see Theorem 3.1, Table 1.1).

\begin{prop}
    The torsion subgroups of fibrations 8.3 and 8.7 in Table \ref{all_fibrations} are $\ZZ/2\ZZ$ and $0$, respectively.
\end{prop}

\begin{proof}
    \sloppy
    Firstly, consider fibration 8.3. By (\cite{Shimada}, Table 1), the torsion subgroup is either $\ZZ/2\ZZ$ or $0$. We determine that the subgroup must be $\ZZ/2\ZZ$ be explicitely showing an order 2 element in $\overline{W_\mathrm{root}}/W_\mathrm{root}$. The lattice $T_0$ is embedded in the Niemeier lattice $L_\mathrm{root} = E_7^{\oplus 2} \oplus D_{10}$. Let $e^{(1)}_1,{.}{.}{.},e^{(1)}_7$, $e^{(2)}_1,{.}{.}{.},e^{(2)}_7$, $d_1,{.}{.}{.},d_{10}$ be generators for $L_\mathrm{root}$. Then, $T_0$ is embedded isomorphically onto the sublattice $\langle e^{(1)}_1, {.}{.}{.}, e^{(1)}_6 \rangle \oplus \langle e^{(2)}_2, e^{(2)}_3 \rangle$ of $L_\mathrm{root}$. We calculate $W_\mathrm{root}$ by taking the root type of the orthogonal complement of $\phi(T_0)$, obtaining $W_\mathrm{root} =$ $\langle e^{(2)}_5, e^{(2)}_6, e^{(2)}_7, (2e^{(2)}_1{+}e^{(2)}_2{+}2e^{(2)}_3{+}3e^{(2)}_4{+}2e^{(2)}_5{+}e^{(2)}_6), e^{(2)}_1 \rangle \oplus \langle d_1, {.}{.}{.}, d_{10} \rangle$. Then, $\eta = \tfrac{e^{(2)}_1 + e^{(2)}_5 + e^{(2)}_7 + d_1 + d_4 + d_6 + d_8 + d_{10}}{2}$ is an element of $L \setminus L_\mathrm{root}$ such that $2\eta \in W_\mathrm{root}$. Therefore, $\eta$ lies in the primitive closure $\overline{W_\mathrm{root}}$, and the torsion subgroup is $\ZZ/2\ZZ$.

    Now consider fibration 8.7. By (\cite{Shimada}, Table 1), the torsion subgroup is also given by either $\ZZ/2\ZZ$ or $0$. We determine that the subgroup must be trivial by showing that no element of order 2 in $L/L_\mathrm{root}$ is in $\overline{W_\mathrm{root}}$. The lattice $T_0$ is embedded in $L_\mathrm{root} = E_6 \oplus D_7 \oplus A_{11}$. Let $e_1, {.}{.}{.}, e_6$, $d_1,{.}{.}{.},d_7$, $a_1,{.}{.}{.},a_{11}$ be generators for $L_\mathrm{root}$. The sublattice of $L_\mathrm{root}$ isomorphic to $T_0$ is $\langle e_1,{.}{.}{.}, e_6 \rangle \oplus \langle d_2, d_3 \rangle$. We calculate $W_\mathrm{root} = \langle (d_1{+}d_2{+}2d_3{+}d_4),d_5,d_6,d_7\rangle \oplus \langle a_1, {.}{.}{.},a_{11}\rangle$. By \cite{niemeier}, $L/L_\mathrm{root} = \ZZ/12\ZZ$, so there is a single element of order 2 modulo $L_\mathrm{root}$. Taking $\mu = \tfrac{d_1 + d_2 + a_1 + a_3 + a_5 + a_7 + a_9 + a_{11}}{2}$ to represent this class, we see that $2\mu \not\in W_\mathrm{root}$. Consequently, the torsion subgroup is trivial.
\end{proof}

\subsection{Classification with respect to automorphisms}

In this section, we apply the results in Section \ref{K3_surfaces_ns_order3} to the fibrations in Table \ref{all_fibrations}.

\begin{prop}
    The fibrations 1.2--1.5, 1.8--1.10, 2.2, 4.2--4.6, 4.9--4.11, 5.3, 6.3--6.7, 6.10--6.12, 7.2, 8.2--8.5, 8.7, 8.8, 9.2, 10.2, 10.3, 10.4, 10.6 in Table \ref{all_fibrations} are of type 2 with respect to any non-symplectic automorphism $\sigma'$ of order 3 fixing their fibre class.
\end{prop}

\begin{proof}
    Since $\sigma'$ fixes the fibre class, the fibrations are of type 1 or 2 in $(X,\sigma')$. All listed fibrations apart from 9.2, 10.2 and 10.6 have positive rank, so by Proposition \ref{type_1_rank_0} they cannot be of type 1. The remaining fibrations have fibres of type $I_n^*$ for $n > 0$, so by Proposition \ref{condition_type_1}, they also cannot be of type 1.
\end{proof}

\begin{prop}
    The fibrations 1.6, 2.1, 2.3, 3.1, 4.1, 4.7, 5.1, 5.2, 6.1, 6.2, 6.8, 6.9, 7.1, 8.1, 8.6, 9.1, 10.1 in Table \ref{all_fibrations} are of type 1 with respect to any non-symplectic automorphism $\sigma'$ of order 3 fixing their fibre class.
\end{prop}

\begin{proof}
    All of the fibrations listed either have a fibre of type $II^*$ or a non-multiple of 3 number of fibres of type $IV^*$. By Proposition \ref{fibres_type_2}, they cannot be of type 2. Since $\sigma'$ fixes their fibre class, they must be of type 1 in $(X,\sigma')$.
\end{proof}

The fibrations that remain to be classified are 1.1, 1.7, 4.8, and 10.5.

\begin{prop}
Fibrations 1.1, 1.7 and 4.8 are of type 1 with relation to any non-symplectic automorphism $\sigma'$ fixing their fibre class.
\end{prop}

\begin{proof}
    For fibrations 1.1 and 1.7, $T = A_2^{\oplus 5}$. Therefore, they have $n$ fibres of type $IV$ and $5-n$ fibres of type $I_3$. Assuming the fibration is of type 2, by Proposition \ref{fibres_type_2} $n$ is equal to either $0$ or $3$. Then, we there are $2$ or $5$ fibres of type $I_3$, which is not possible.

    Let $\cE$ be a Jacobian elliptic fibration of number 4.8. Then, $T = A_2^{\oplus 6}$, so $\cE$ has 6 reducible fibres of type $IV$ or $I_3$. Assume $\cE$ is of type 2 in $(X,\sigma')$, and let $\cE_R \colon R \to \PP^1$ be the rational elliptic surface and $\tau$ the base change determined by Proposition \ref{basechange_R}. Then, if $\tau$ is ramified at $a,b \in \PP^1$, the fibre $F_a$ is of type $IV$ and $F_b$ is a smooth fibre (see Proposition \ref{basechange_is_k3}). Furthermore, there are two other reducible fibres $F_{u_1}$ and $F_{u_2}$ of type $IV$ or $I_3$. By the Shioda--Tate formula (\cite{shioda1989mordell}, Theorem 1.1), the rank of $\cE_R$ is equal to 2. This is not possible since the sections of $\cE_R$ lift to $\cE$, and $\cE$ has rank 0. 
\end{proof}

The remaining fibration, namely number 10.5 in Table \ref{all_fibrations}, is treated in Section \ref{section_X3_surface}. We see that it is the only fibration that can be of type 1 with relation to one automorphism (see Theorem \ref{main_classification}) and of type 2 with relation to another (see Proposition \ref{ce5_type_2}).

\begin{remark}
    Notice that the surface number 3 in Table \ref{K3_surfaces_ns_order3} only admits one fibration, which is of type 1. This is consistent with the fact that the automorphism fixes a curve of genus 2 (cf. Proposition \ref{fixes_high_genus}). This is the only such case in Table \ref{List_K3_applied_nishiyama}.
\end{remark}

\section{Elliptic fibrations on $X_3$}\label{section_X3_surface}

In this section, we apply the classification method in Section \ref{K3_surfaces_ns_order3} to the $X_3$ surface (see \ref{remark_10isX3}). We have listed all Jacobian elliptic fibration of $X_3$ modulo $\cJ_2$-equivalence (see surface 10 in Table \ref{List_K3_applied_nishiyama}). These were originally calculated in \cite{Nishiyama1996} (see Remark \ref{remark_10isX3}). Furthermore, Corollary D of \cite{BKW} shows that $\cJ_1(X_3) = \cJ_2(X_3)$, that is, there is exactly one Jacobian elliptic fibration on $X_3$ in each $\cJ_2$-class modulo $\Aut(X_3)$. While these elliptic fibrations have been known in the literature (see \cite{Nishiyama1996}, \cite{Shioda-Inose}, \cite{Vinberg}), this method allows us to find explicit Weierstrass equations for each class. We start in Section \ref{section_X} by constructing $X_3$ as a base change of a rational elliptic surface $R$ by a cubic Galois cover, which endows the K3 surface with a non-symplectic automorphism $\sigma$. In Section \ref{section_ell_divisors}, we provide a Jacobian elliptic fibration in each class of $\cJ_1(X)$ by describing the fibre class and a section, and apply Theorem \ref{class_linearsys} to classify each in relation to $\sigma$. In Section \ref{section_pencils}, we apply Propositions \ref{equation_type2} and \ref{equation_type1} to provide Weierstrass equation for each fibration. In Section \ref{other_aut_X3}, we discuss different constructions of $X_3$ as base changes of rational elliptic surfaces, examining when the non-symplectic automorphism produced is equal to or distinct from $\sigma$, and providing an example of a Jacobian elliptic fibrations of type 3.

\subsection{Construction of $X_3$}\label{section_X}

In this section, we construct $X_3$ as the base change of a specific rational elliptic surface $R$. There are distinct possibilites for rational elliptic surfaces and base changes resulting in $X_3$. Here, we choose $R$ to be the rational elliptic surface obtained by the pencil of cubics $\Lambda = s\cF + t\cG$ in $\PP^2$, where $\cF$ is given by $xyz = 0$ and $\cG$ by $(x-y)(y-z)(z-x) = 0$ (see Proposition \ref{RES}).

\begin{center}
\begin{figure}[ht]
    \centering
    \begin{tikzpicture}
	\begin{pgfonlayer}{nodelayer}
		\node [style=none] (0) at (0.5, 5) {};
		\node [style=none] (1) at (-3.5, -3) {};
		\node [style=none] (2) at (4, -2) {};
		\node [style=none] (3) at (-4, -2) {};
		\node [style=none] (4) at (-0.5, 5) {};
		\node [style=none] (5) at (3.5, -3) {};
		\node [style=none] (6) at (2.25, 1.5) {};
		\node [style=none] (7) at (-2.25, 1.5) {};
		\node [style=none] (8) at (0, -3.25) {};
		\node [style=none] (9) at (-3.75, -2.5) {};
		\node [style=none] (10) at (3.75, -2.5) {};
		\node [style=none] (11) at (0, 5) {};
		\node [style=none] (14) at (0.75, 4) {\small $P_1$};
		\node [style=none] (18) at (3.5, -1.5) {\small $P_2$};
		\node [style=none] (19) at (-3.5, -1.5) {\small $P_3$};
		\node [style=none] (20) at (5, -2) {$\ell_1$};
		\node [style=none] (21) at (-3.75, -3.5) {$\ell_2$};
		\node [style=none] (22) at (3.75, -3.5) {$\ell_3$};
		\node [style=none] (23) at (-3, 1.75) {$\color{red}{r_2}$};
		\node [style=none] (24) at (3, 1.75) {$\color{red}{r_3}$};
		\node [style=none] (25) at (0, -3.75) {$\color{red}{r_1}$};
		\node [style=none] (26) at (-2.25, 0.75) {\small $Q_2$};
		\node [style=none] (27) at (2.25, 0.75) {\small $Q_3$};
		\node [style=none] (28) at (-0.5, -2.5) {\small $Q_1$};
		\node [style=none] (29) at (0.25, 0.5) {\footnotesize $O$};
		\node [style=smallblackdot] (40) at (-3, -2) {};
		\node [style=smallblackdot] (41) at (0, -2) {};
		\node [style=smallblackdot] (42) at (3, -2) {};
		\node [style=smallblackdot] (43) at (0, 0) {};
		\node [style=smallblackdot] (44) at (-1.5, 1) {};
		\node [style=smallblackdot] (45) at (1.5, 1) {};
		\node [style=smallblackdot] (46) at (0, 4) {};
	\end{pgfonlayer}
	\begin{pgfonlayer}{edgelayer}
		\draw (1.center) to (0.center);
		\draw (3.center) to (2.center);
		\draw (4.center) to (5.center);
		\draw [style=red] (11.center) to (8.center);
		\draw [style=red] (10.center) to (7.center);
		\draw [style=red] (9.center) to (6.center);
	\end{pgfonlayer}
\end{tikzpicture} \begin{tabular}{cc|c}
    $\ell_1:$ & $x = 0$ & $P_1 = [1{:}0{:}0]$\\
    $\ell_2:$ & $y = 0$ & $P_2 = [0{:}1{:}0]$\\
    $\ell_3:$ & $z = 0$ & $P_3 = [0{:}0{:}1]$\\
    $r_1:$ & $y-z = 0$ & $Q_1 = [0{:}1{:}1]$\\
    $r_2:$ & $z-x = 0$ & $Q_2 = [1{:}0{:}1]$\\
    $r_3:$ & $x-y = 0$ & $Q_3 = [1{:}1{:}0]$\\
    \ && $O_{\ } = [1{:}1{:}1]$
\end{tabular}
    \caption{Cubics generating $\Lambda$}
    \label{R_pencil}
\end{figure} 
\end{center}

\sloppy
The base points of $\Lambda$ lie on the scheme theoretic intersection $\cF \cap \cG,$ and consist of $P_1,P_2,P_3,Q_1,..,Q_6$. The points $Q_4, Q_5, Q_6$ are infinitely near to $P_1, P_2, P_3$, respectively, and correspond to the tangent directions of $r_1, r_2, r_3$. Blowing up the base points, we obtain the rational elliptic surface $R$, whose only reducible fibres are $F_a := \cE_R^{-1}([0{:}1])$ of type $IV$ and $F_b := \cE_R^{-1}([1{:}0])$ of type $I_6$, i.e. the strict transforms of $\cF$ and $\cG$. The exceptional divisors $H_1,...,H_6$ above $Q_1,...,Q_6$ determine sections of $\cE_R$.

\begin{figure}[ht]
    \centering
    \begin{tikzpicture}
	\begin{pgfonlayer}{nodelayer}
		\node [style=none] (6) at (-5, 6) {};
		\node [style=none] (7) at (5, 6) {};
		\node [style=none] (8) at (-3.5, 6.75) {};
		\node [style=none] (9) at (3.5, 6.75) {};
		\node [style=none] (10) at (-8.5, -0.75) {};
		\node [style=none] (11) at (-8.5, 0.75) {};
		\node [style=none] (12) at (8.5, -0.75) {};
		\node [style=none] (13) at (8.5, 0.75) {};
		\node [style=none] (14) at (-5, -6) {};
		\node [style=none] (15) at (-3.5, -6.75) {};
		\node [style=none] (16) at (5, -6) {};
		\node [style=none] (17) at (3.5, -6.75) {};
		\node [style=none] (18) at (0, 5) {};
		\node [style=none] (19) at (0, -5) {};
		\node [style=none] (20) at (-5, -2.5) {};
		\node [style=none] (21) at (5, 2.5) {};
		\node [style=none] (22) at (-5, 2.5) {};
		\node [style=none] (23) at (5, -2.5) {};
		\node [style=none] (25) at (-7, 3) {};
		\node [style=none] (26) at (-3.75, 1.5) {};
		\node [style=none] (27) at (-7, -3.5) {};
		\node [style=none] (28) at (-3.25, -2) {};
		\node [style=none] (29) at (0.25, -7) {};
		\node [style=none] (30) at (0.25, -3.75) {};
		\node [style=none] (31) at (6.75, -3) {};
		\node [style=none] (32) at (4, -1.75) {};
		\node [style=none] (33) at (3.75, 2.25) {};
		\node [style=none] (34) at (6.5, 3.5) {};
		\node [style=none] (35) at (-0.25, 4) {};
		\node [style=none] (36) at (-0.25, 6.75) {};
		\node [style=smallblackdot] (37) at (-4, 6) {};
		\node [style=smallblackdot] (38) at (4, 6) {};
		\node [style=smallblackdot] (39) at (-8, 0) {};
		\node [style=smallblackdot] (40) at (-4, -6) {};
		\node [style=smallblackdot] (41) at (4, -6) {};
		\node [style=smallblackdot] (42) at (8, 0) {};
		\node [style=smallblackdot] (43) at (0, 0) {};
		\node [style=none] (46) at (-2, 6.5) {$E_1$};
		\node [style=none] (52) at (5.75, 4.75) {$\ell_3$};
		\node [style=none] (53) at (7.75, -2) {$E_2$};
		\node [style=none] (54) at (2, -6.5) {$\ell_1$};
		\node [style=none] (55) at (-6, -4.25) {$E_3$};
		\node [style=none] (56) at (-7.5, 2) {$\ell_2$};
		\node [style=none] (57) at (1, 5) {$\color{blue}H_4$};
		\node [style=none] (58) at (5.75, 1.75) {$\color{blue}H_3$};
		\node [style=none] (59) at (4.5, -3.5) {$\color{blue}H_5$};
		\node [style=none] (60) at (-1.5, -4.5) {$\color{blue}H_1$};
		\node [style=none] (61) at (-5.75, -1.5) {$\color{blue}H_6$};
		\node [style=none] (62) at (-4.25, 3.5) {$\color{blue}H_2$};
		\node [style=none] (63) at (1.25, -5) {$\color{red}r_1$};
		\node [style=none] (64) at (-5.25, 1.75) {$\color{red}r_2$};
		\node [style=none] (65) at (4.25, 3.25) {$\color{red}r_3$};
	\end{pgfonlayer}
	\begin{pgfonlayer}{edgelayer}
		\draw (6.center) to (7.center);
		\draw (8.center) to (10.center);
		\draw (11.center) to (15.center);
		\draw (14.center) to (16.center);
		\draw (17.center) to (13.center);
		\draw (12.center) to (9.center);
		\draw [style=red] (20.center) to (21.center);
		\draw [style=red] (22.center) to (23.center);
		\draw [style=red] (19.center) to (18.center);
		\draw [style=blue, bend left=45] (27.center) to (28.center);
		\draw [style=blue, bend left=45, looseness=1.25] (25.center) to (26.center);
		\draw [style=blue, bend left=45] (34.center) to (33.center);
		\draw [style=blue, bend left=60] (36.center) to (35.center);
		\draw [style=blue, bend left=60] (29.center) to (30.center);
		\draw [style=blue, bend right=45] (32.center) to (31.center);
	\end{pgfonlayer}
\end{tikzpicture}
    \caption{Reducible fibres of $\cE_R$}
    \label{R_fibres}
\end{figure}

The curves in Figure \ref{R_fibres} have self-intersections $\ell_i^2 = E_i^2 = r_i^2 = -2$, and $H_i^2= -1$. By \cite[Main Theorem]{Shioda-Oguiso}, the Mordell--Weil group of $\cE_R$ is $\MW(\cE_R) = \ZZ \oplus \ZZ/3$.

We apply the base change by the cubic Galois cover ${\tau \colon \PP^1 \to \PP^1}$ totally ramified at $[0{:}1]$ and $[1{:}0]$, obtaining a K3 surface with Jacobian elliptic fibration $\cE_X:X \rightarrow \PP^1$ (see \ref{basechange_is_k3}). The fibres above the ramification points of $\tau$ are $F^X_a := \cE_X^{-1}([0{:}1])$ of type $I_0$, and $F^X_b := \cE_X^{-1}([1{:}0])$ of type $I_{18}$, and every other fibre is irreducible (see \cite{miranda}, Table VI.4.1).

\begin{prop}\label{X-X_3-isom}
    $X$ is isomorphic to $X_3$.
\end{prop}

\begin{proof}
    In Table \ref{List_K3_applied_nishiyama}, we see that the fibration 10.4 on $X_3$ has the same $ADE$-type as $\cE_X$ and its Mordell--Weil group is equal to $\ZZ \oplus \ZZ/3\ZZ$. By the Shioda-Tate formula, the rank of $\MW \cE_X$ is 1, and by Table 1 in \cite{Shimada} the torsion part is either trivial or equal to $\ZZ/3\ZZ$. Since $\cE_R$ has torsion sections which are lifted by the base change, we conclude that Mordell--Weil group of $\cE_X$ is also equal to $\ZZ \oplus \ZZ/3\ZZ$. Since both fibrations have the same reducible fibres and Mordell--Weil groups, by Theorem 1.1 in \cite{shioda1989mordell}, the Néron--Severi groups of $X$ and $X_3$ are isomorphic. Then, as a consequence of the Global Torelli Theorem for K3 surfaces (\cite{Piatetskii-Shafarevich}, Section 1, \cite{huybrechts}, Theorem 5.3), $X$ and $X_3$ are isomorphic.
\end{proof}

In view of this proposition, we refer to $X_3$ by $X$ through the rest of this section. Notice that since $\tau$ is a Galois covering, the base change induces a non-symplectic automorphism $\sigma$ in $\Aut(X)$ of order 3.

\begin{figure}[ht]
    \centering
    \begin{tikzpicture}
	\begin{pgfonlayer}{nodelayer}
		\node [style=none] (6) at (-4.5, 6.5) {};
		\node [style=none] (7) at (-5.75, 3.75) {};
		\node [style=none] (9) at (-5.5, 4.5) {};
		\node [style=none] (10) at (-6.75, 1.75) {};
		\node [style=none] (12) at (-6.5, 2.5) {};
		\node [style=none] (13) at (-7.75, -0.25) {};
		\node [style=none] (20) at (-4.5, -6.5) {};
		\node [style=none] (21) at (-5.75, -3.75) {};
		\node [style=none] (22) at (-5.5, -4.5) {};
		\node [style=none] (23) at (-6.75, -1.75) {};
		\node [style=none] (24) at (-6.5, -2.5) {};
		\node [style=none] (25) at (-7.75, 0.25) {};
		\node [style=none] (26) at (4.75, 6.25) {};
		\node [style=none] (27) at (1.25, 6.25) {};
		\node [style=none] (28) at (1.75, 6.25) {};
		\node [style=none] (29) at (-1.75, 6.25) {};
		\node [style=none] (30) at (-1.25, 6.25) {};
		\node [style=none] (31) at (-4.75, 6.25) {};
		\node [style=none] (32) at (4.5, 6.5) {};
		\node [style=none] (33) at (5.75, 3.75) {};
		\node [style=none] (34) at (5.5, 4.5) {};
		\node [style=none] (35) at (6.75, 1.75) {};
		\node [style=none] (36) at (6.5, 2.5) {};
		\node [style=none] (37) at (7.75, -0.25) {};
		\node [style=none] (38) at (4.75, -6.25) {};
		\node [style=none] (39) at (1.25, -6.25) {};
		\node [style=none] (40) at (1.75, -6.25) {};
		\node [style=none] (41) at (-1.75, -6.25) {};
		\node [style=none] (42) at (-1.25, -6.25) {};
		\node [style=none] (43) at (-4.75, -6.25) {};
		\node [style=none] (44) at (4.5, -6.5) {};
		\node [style=none] (45) at (5.75, -3.75) {};
		\node [style=none] (46) at (5.5, -4.5) {};
		\node [style=none] (47) at (6.75, -1.75) {};
		\node [style=none] (48) at (6.5, -2.5) {};
		\node [style=none] (49) at (7.75, 0.25) {};
		\node [style=none] (50) at (0, 6.25) {};
		\node [style=none] (51) at (-5, -0.75) {};
		\node [style=none] (52) at (5, -0.75) {};
		\node [style=none] (53) at (-2.75, -0.75) {};
		\node [style=none] (54) at (0, 2) {};
		\node [style=none] (55) at (2.75, -0.75) {};
		\node [style=none] (56) at (-1.25, 0.75) {};
		\node [style=none] (57) at (1.25, 0.75) {};
		\node [style=none] (58) at (0.25, 1.25) {};
		\node [style=none] (59) at (0.25, 3) {};
		\node [style=none] (60) at (0, -6.5) {};
		\node [style=smallblackdot] (61) at (0, 2) {};
		\node [style=none] (62) at (-6.25, -3) {};
		\node [style=none] (63) at (-0.5, -1.75) {};
		\node [style=none] (64) at (0, -1.75) {};
		\node [style=none] (65) at (3, -1.5) {};
		\node [style=none] (66) at (3.5, -0.25) {};
		\node [style=none] (67) at (6.25, 3.25) {};
		\node [style=smallblackdot] (68) at (2.75, -0.75) {};
		\node [style=none] (69) at (-6.25, 3.25) {};
		\node [style=none] (71) at (6.5, -3.25) {};
		\node [style=none] (72) at (-3, 0.25) {};
		\node [style=none] (73) at (2.75, -2.75) {};
		\node [style=none] (74) at (-0.5, -3.25) {};
		\node [style=none] (75) at (-2.25, -2.25) {};
		\node [style=none] (76) at (-2.5, -1.75) {};
		\node [style=none] (77) at (0, -3.25) {};
		\node [style=none] (78) at (-2, -1.25) {};
		\node [style=smallblackdot] (79) at (-2.75, -0.75) {};
		\node [style=smallblackdot] (80) at (-7.5, 0) {};
		\node [style=smallblackdot] (81) at (7.5, 0) {};
		\node [style=smallblackdot] (82) at (-4.5, 6) {};
		\node [style=smallblackdot] (83) at (4.5, 6) {};
		\node [style=smallblackdot] (84) at (4.5, -6) {};
		\node [style=smallblackdot] (85) at (-4.5, -6) {};
		\node [style=none] (86) at (-3.25, 6.75) {$\Theta_{17}$};
		\node [style=none] (87) at (3.25, 6.5) {$\Theta_{1}$};
		\node [style=none] (88) at (0, 7.25) {$\Theta_{0}$};
		\node [style=none] (89) at (5.5, 5.5) {$\Theta_{2}$};
		\node [style=none] (90) at (7, 3.5) {$\Theta_{3}$};
		\node [style=none] (91) at (7.75, 1.5) {$\Theta_{4}$};
		\node [style=none] (92) at (7.5, -1.25) {$\Theta_{5}$};
		\node [style=none] (93) at (7, -3.5) {$\Theta_{6}$};
		\node [style=none] (94) at (5.75, -5.25) {$\Theta_{7}$};
		\node [style=none] (95) at (3.25, -7) {$\Theta_{8}$};
		\node [style=none] (96) at (0, -7.25) {$\Theta_{9}$};
		\node [style=none] (97) at (-3.25, -6.75) {$\Theta_{10}$};
		\node [style=none] (98) at (-5.75, -5.5) {$\Theta_{11}$};
		\node [style=none] (99) at (-7, -3.5) {$\Theta_{12}$};
		\node [style=none] (100) at (-7.75, -1.25) {$\Theta_{13}$};
		\node [style=none] (101) at (-7.75, 1.25) {$\Theta_{14}$};
		\node [style=none] (102) at (-7.25, 3.5) {$\Theta_{15}$};
		\node [style=none] (103) at (-5.75, 5.5) {$\Theta_{16}$};
		\node [style=none] (104) at (6.75, 7.25) {$F^X_{b}$};
		\node [style=none] (105) at (2.25, 2.25) {$\color{red}F^X_{a}$};
		\node [style=none] (106) at (-1, 4.75) {$\color{blue}\Sigma_0$};
		\node [style=none] (107) at (-5.5, 1) {$\color{blue}\Sigma_5$};
		\node [style=none] (108) at (-4.5, -3) {$\color{blue}\Sigma_4$};
		\node [style=none] (109) at (0.25, -4.75) {$\color{blue}\Sigma_3$};
		\node [style=none] (110) at (5, -2) {$\color{blue}\Sigma_2$};
		\node [style=none] (111) at (4.25, 2.75) {$\color{blue}\Sigma_1$};
		\node [style=blackdot] (112) at (-4.5, 6) {};
		\node [style=blackdot] (113) at (4.5, 6) {};
		\node [style=blackdot] (114) at (0, 2) {};
		\node [style=blackdot] (115) at (-2.75, -0.75) {};
		\node [style=blackdot] (116) at (2.75, -0.75) {};
		\node [style=blackdot] (117) at (7.5, 0) {};
		\node [style=blackdot] (118) at (4.5, -6) {};
		\node [style=blackdot] (119) at (-4.5, -6) {};
		\node [style=blackdot] (120) at (-7.5, 0) {};
		\node [style=blackdot] (121) at (-2.75, -0.75) {};
		\node [style=blackdot] (122) at (-4.5, 6) {};
	\end{pgfonlayer}
	\begin{pgfonlayer}{edgelayer}
		\draw [bend left, looseness=0.75] (6.center) to (7.center);
		\draw [line width=0.6mm, bend left, looseness=0.75] (9.center) to (10.center);
		\draw [bend left, looseness=0.75] (12.center) to (13.center);
		\draw [bend right, looseness=0.75] (20.center) to (21.center);
		\draw [line width=0.6mm, bend right, looseness=0.75] (22.center) to (23.center);
		\draw [bend right, looseness=0.75] (24.center) to (25.center);
		\draw [bend left=45, looseness=0.75] (26.center) to (27.center);
		\draw [bend left=45, looseness=0.75] (28.center) to (29.center);
		\draw [bend left=45, looseness=0.75] (30.center) to (31.center);
		\draw [bend right, looseness=0.75] (32.center) to (33.center);
		\draw [line width=0.6mm, bend right, looseness=0.75] (34.center) to (35.center);
		\draw [bend right, looseness=0.75] (36.center) to (37.center);
		\draw [bend right=45, looseness=0.75] (38.center) to (39.center);
		\draw [bend right=45, looseness=0.75] (40.center) to (41.center);
		\draw [bend right=45, looseness=0.75] (42.center) to (43.center);
		\draw [bend left, looseness=0.75] (44.center) to (45.center);
		\draw [line width=0.6mm, bend left, looseness=0.75] (46.center) to (47.center);
		\draw [bend left, looseness=0.75] (48.center) to (49.center);
		\draw [bend right=45, looseness=0.75] (31.center) to (30.center);
		\draw [line width=0.6mm, bend right=45, looseness=0.75] (29.center) to (28.center);
		\draw [bend right=45, looseness=0.75] (27.center) to (26.center);
		\draw [bend left=45, looseness=0.75] (43.center) to (42.center);
		\draw [line width=0.6mm, bend left=45, looseness=0.75] (41.center) to (40.center);
		\draw [bend left=45, looseness=0.75] (39.center) to (38.center);
		\draw [style=red] (51.center) to (53.center);
		\draw [style=red, bend right=45, looseness=1.25] (53.center) to (56.center);
		\draw [style=red, bend left=45] (56.center) to (54.center);
		\draw [style=red, bend left=45] (54.center) to (57.center);
		\draw [style=red, bend right=45, looseness=1.25] (57.center) to (55.center);
		\draw [style=red] (55.center) to (52.center);
		\draw [style=blue, bend right=15, looseness=0.75] (50.center) to (54.center);
		\draw [style=blue] (58.center) to (54.center);
		\draw [style=blue, bend left=15, looseness=0.50] (60.center) to (54.center);
		\draw [style=blue] (54.center) to (59.center);
		\draw [style=blue, bend right=45, looseness=0.50] (67.center) to (55.center);
		\draw [style=blue, bend right] (55.center) to (65.center);
		\draw [style=blue, bend left=15, looseness=0.50] (62.center) to (63.center);
		\draw [style=blue, bend right=15] (64.center) to (55.center);
		\draw [style=blue, bend left=15] (55.center) to (66.center);
		\draw [style=blue, bend right] (71.center) to (73.center);
		\draw [style=blue, bend left=15, looseness=0.75] (73.center) to (77.center);
		\draw [style=blue, bend right=330] (74.center) to (75.center);
		\draw [style=blue] (76.center) to (53.center);
		\draw [style=blue] (53.center) to (72.center);
		\draw [style=blue, bend right=15, looseness=0.50] (69.center) to (53.center);
		\draw [style=blue] (53.center) to (78.center);
	\end{pgfonlayer}
\end{tikzpicture}
    \caption{Ramified fibres of $\cE_X$}
    \label{X_fibres}
\end{figure}

Figure \ref{X_fibres} describes the fibres $F^X_a$ and $F^X_b$ of $\cE_X$. The curves $\Theta_i$ are the components of $F^X_b$ and $\Sigma_i$ are the sections inherited from $H_i$ in $R$. The self-intersections of the portrayed curves are $\Theta_i^2 = \Sigma_i^2 = -2$, $(F_a^{X})^2 = 0$. The action of $\sigma$ is also described (see Remark \ref{remark_action_X_fibres}). 

\begin{prop}
    The automorphism $\sigma$ acts trivially on $\NS X$.
\end{prop}

\begin{proof}
    Since $\sigma$ comes from a base change, it preserves the fibre class and the sections of $\cE_X$. It remains to see that the components $\Theta_i$ are also fixed or preserved. Since $\sigma(F^X_b) = F^X_b$, the image $\sigma(\Theta_i)$ is equal to some $\Theta_j$ for every $i = 0,...,17$. Each $\Theta_{3i}$, for $i = 0,...,5$, is lifted from the components of $F_b$ in $R$, so they are fixed by $\sigma$. Every other component $\Theta_i$ intersects a unique $\Theta_{3k}$ for some value of $k$. As a consequence, the intersection point $p$ in $\Theta_i \cap \Theta_{3k}$ is fixed by $\sigma$. Suppose $\sigma(\Theta_i) = \Theta_j$ for some $j \neq i$. Then, $p = \sigma(p) \in \Theta_i \cap \Theta_j \cap \Theta_{3k} = \varnothing$, which is absurd. Therefore, $\sigma(\Theta_i) = \Theta_i$ and $\sigma$ acts trivially on $\NS X$.
\end{proof}

As a corollary $(X,\sigma)$ only has fibrations of type 1 or 2. Furthermore, this implies that $\NS X = N(\sigma)$ (i.e. the part of $\NS X$ fixed by $\sigma$). By Table \ref{order3_lattices}, $\sigma$ has 9 isolated fixed points and 6 fixed curves of genus 0. We conclude that $\Theta_{3i}$ are the only curves fixed by $\sigma$. Six of the isolated fixed points can be found in the intersections of $\Theta_{3i+1}$ and $\Theta_{3i+2}$ for $i = 0,...,5$. The three remaining fixed points lie in $F^X_a$, since the remaining fibres are all interchanged by $\sigma$.

\begin{remark}\label{remark_action_X_fibres}
   In Figure \ref{X_fibres}, the fixed curves of $\sigma$ are highlited as bold, and the isolated fixed points are marked by black dots.
\end{remark}

\subsection{Elliptic fibrations on $X_3$}\label{section_ell_divisors}

By Proposition \ref{K3_elliptic_fibration}, elliptic fibrations on $X$ are equivalent to embeddings $U \hookrightarrow \NS X$. Then, for classes $L,M \in \NS X$ such that $L^2 = 0$, $M^2 = -2$ and $LM = 1$, there is a fibration $\cE_{L,M} \colon X \to \PP^1$ such that $L$ is the fibre class, and $M$ the class of the zero-section. Our goal is to create divisors $L_i$ and $M_i$ which induce an elliptic fibrations $\cE_i \colon X \to \PP^1$ for each fibration 10.i in Table \ref{List_K3_applied_nishiyama}. We can create $L_i, M_i$ using only $\Theta_j$'s and $\Sigma_k$'s as components.

\begin{prop}\label{induced_fibration}
    For $i = 1,...,6$, each pair $L_i,M_i$ in Table \ref{Divisors_elliptic} induces an elliptic fibration $\cE_i \colon X \to \PP^1$ in the $\cJ_1$-class as fibration 10.i in Table \ref{List_K3_applied_nishiyama}.

    \begin{table}[ht]
    \centering
    \begin{tabular}{|c|c|c|}
        \hline
        No. & $L_i$ & $M_i$ \\
        \hline
        $1$ & $2\Theta_{16} + 4 \Theta_{17} + 6\Theta_0 + 3 \Sigma_0 + 5 \Theta_1 + 4\Theta_2 + 3 \Theta_3 + 2 \Theta_4 + \Theta_5$ & $\Theta_{6}$ \\
        \hline
        $2$ & $\Theta_{17} + \Sigma_0 + 2 \sum_{j=0}^{12} \Theta_j + \Sigma_4 + \Theta_{13}$ & $\Theta_{16}$\\
        \hline
        $3$ & $\Theta_{15} + 2 \Theta_{16} + 3 \Theta_{17} + 4\Theta_0 + 2\Sigma_0 + 3 \Theta_1 + 2 \Theta _2 + \Theta_3$ & $\Theta_{4}$ \\
        \hline
        $4$ & $\sum_{j=0}^{17} \Theta_j$ & $\Sigma_0$\\
        \hline
        $5$ & $3 \Theta_0 + 2\Theta_1 + 2 \Theta_{17} + 2 \Sigma_0 + \Theta_2 + \Theta_{16} + \Sigma_3$ & $\Theta_{3}$\\
        \hline
        $6$ & $\Theta_{17} + \Sigma_0 + 2 \Theta_0 + 2 \Theta_1 + 2 \Theta_2 + 2 \Theta_3 + \Sigma_1 + \Theta_4$ & $\Theta_{5}$ \\
        \hline
    \end{tabular}
    \caption{Divisors inducing elliptic fibrations on $X_3$}
    \label{Divisors_elliptic}
\end{table}
\end{prop}
\begin{proof}
    For each $i$, the divisor $L_i$ lies in the fibre class of $\cE_i$. By construction, this divisor represents a reducible fibre of $\cE_i$, and its root type as a lattice must appear in the $ADE$-type of the fibration. For example, $L_1$ is a fibre of type $II^*$, which induces an $E_8$ in the $ADE$-type of $\cE_1$. The only fibration with this $ADE$-type possible in $X$ by Table \ref{List_K3_applied_nishiyama} is 10.1. Equivalently, $L_2$ induces a $D_{16}$, $L_3$ an $E_7$, $L_4$ an $A_{17}$, $L_5$ an $E_6$ and $L_6$ a $D_7$. These root lattices are all unique to the $ADE$-types of their respective fibrations.
\end{proof}

In order to apply Theorem \ref{class_linearsys}, we need to describe $\tilde{X}$, given by the blow-up of the isolated fixed points of $\sigma$, and $\tilde{R}$, the quotient of $\tilde{X}$ by the lifting $\tilde{\sigma}$ of $\sigma$. The surface $\tilde{X}$ has 9 new components given by the exceptional curves of the blow-up $\eta \colon \tilde{X} \to X$. We denote the 6 curves above the intersection points of $\Theta_{3i+1}$ and $\Theta_{3i+2}$ by $\Theta_{3i+1,3i+2}$ (for $i = 0, {.}{.}{.},5$), and the 3 curves above the intersection of $\Sigma_i$ and $\Sigma_{i+3}$ by $\Sigma_{i,i+3}$ (for $i = 0,1,2$). For each curve $C$ in $X$, $\tilde{C}$ denotes the strict transform of $C$ by $\eta$. The self intersections of each component are given by $\tilde{\Theta}_i^2 = -2$ if $i \equiv 0 \mod 3$, $\tilde{\Theta}_i^2 = -3$ if $i \not\equiv 0 \mod 3$, $\tilde{\Sigma}_i^2 = -3$, $\Theta_{i,i+1}^2 = \Sigma_{i,i+3}^2 = -1$ and $\tilde{F}_a^2 = -2$. The automorphism $\tilde{\sigma}$ of $\tilde{X}$ maintains the action of $\sigma$ on each strict transform and fixes the exceptional curves.

\newpage
\begin{figure}[ht]
    \centering
    \begin{tikzpicture}
	\begin{pgfonlayer}{nodelayer}
		\node [style=none] (86) at (-3, 8.5) {$\tilde{\Theta}_{17}$};
		\node [style=none] (87) at (3, 8.5) {$\tilde{\Theta}_{1}$};
		\node [style=none] (88) at (0, 8.75) {$\tilde{\Theta}_{0}$};
		\node [style=none] (89) at (8.25, 6.5) {$\tilde{\Theta}_{2}$};
		\node [style=none] (90) at (9.5, 5) {$\tilde{\Theta}_{3}$};
		\node [style=none] (91) at (10.25, 2.75) {$\tilde{\Theta}_{4}$};
		\node [style=none] (92) at (10.25, -2.75) {$\tilde{\Theta}_{5}$};
		\node [style=none] (93) at (10, -5.25) {$\tilde{\Theta}_{6}$};
		\node [style=none] (94) at (8.5, -7) {$\tilde{\Theta}_{7}$};
		\node [style=none] (95) at (3.25, -9.75) {$\tilde{\Theta}_{8}$};
		\node [style=none] (96) at (0, -9.75) {$\tilde{\Theta}_{9}$};
		\node [style=none] (97) at (-3, -9.75) {$\tilde{\Theta}_{10}$};
		\node [style=none] (98) at (-8, -7) {$\tilde{\Theta}_{11}$};
		\node [style=none] (99) at (-9, -5) {$\tilde{\Theta}_{12}$};
		\node [style=none] (100) at (-10, -3) {$\tilde{\Theta}_{13}$};
		\node [style=none] (101) at (-10, 2.5) {$\tilde{\Theta}_{14}$};
		\node [style=none] (102) at (-9.25, 4.75) {$\tilde{\Theta}_{15}$};
		\node [style=none] (103) at (-8, 6.5) {$\tilde{\Theta}_{16}$};
		\node [style=none] (106) at (1.5, 6.5) {$\color{blue}\tilde{\Sigma}_0$};
		\node [style=none] (107) at (-6.5, 4) {$\color{blue}\tilde{\Sigma}_5$};
		\node [style=none] (108) at (-7.25, -3) {$\color{blue}\tilde{\Sigma}_4$};
		\node [style=none] (109) at (-0.75, -7.75) {$\color{blue}\tilde{\Sigma}_3$};
		\node [style=none] (110) at (6.75, -4.5) {$\color{blue}\tilde{\Sigma}_2$};
		\node [style=none] (111) at (7.5, 2.75) {$\color{blue}\tilde{\Sigma}_1$};
		\node [style=none] (123) at (-6.5, 7.5) {};
		\node [style=none] (124) at (-7.75, 4.75) {};
		\node [style=none] (125) at (-7.5, 5.5) {};
		\node [style=none] (126) at (-8.75, 2.75) {};
		\node [style=none] (127) at (-8.5, 3.5) {};
		\node [style=none] (128) at (-9.75, 0.75) {};
		\node [style=none] (129) at (-6.5, -8) {};
		\node [style=none] (130) at (-7.75, -5.25) {};
		\node [style=none] (131) at (-7.5, -6) {};
		\node [style=none] (132) at (-8.75, -3.25) {};
		\node [style=none] (133) at (-8.5, -4) {};
		\node [style=none] (134) at (-9.75, -1.25) {};
		\node [style=none] (135) at (4.75, 8) {};
		\node [style=none] (136) at (1.25, 8) {};
		\node [style=none] (137) at (1.75, 8) {};
		\node [style=none] (138) at (-1.75, 8) {};
		\node [style=none] (139) at (-1.25, 8) {};
		\node [style=none] (140) at (-4.75, 8) {};
		\node [style=none] (141) at (7, 7.5) {};
		\node [style=none] (142) at (8.25, 4.75) {};
		\node [style=none] (144) at (9, 3.5) {};
		\node [style=none] (145) at (10.25, 0.75) {};
		\node [style=none] (146) at (4.75, -9) {};
		\node [style=none] (147) at (1.25, -9) {};
		\node [style=none] (148) at (1.75, -9) {};
		\node [style=none] (149) at (-1.75, -9) {};
		\node [style=none] (150) at (-1.25, -9) {};
		\node [style=none] (151) at (-4.75, -9) {};
		\node [style=none] (152) at (7, -8) {};
		\node [style=none] (153) at (8.25, -5.25) {};
		\node [style=none] (154) at (8, -6) {};
		\node [style=none] (155) at (9.25, -3.25) {};
		\node [style=none] (156) at (9, -4) {};
		\node [style=none] (157) at (10.25, -1.25) {};
		\node [style=none] (158) at (8, 5.5) {};
		\node [style=none] (159) at (9.25, 2.75) {};
		\node [style=none] (160) at (-7, 7) {};
		\node [style=none] (161) at (-4, 8) {};
		\node [style=none] (162) at (4.25, 8) {};
		\node [style=none] (163) at (7.5, 7) {};
		\node [style=none] (164) at (-9.5, 1.25) {};
		\node [style=none] (165) at (-9.5, -1.75) {};
		\node [style=none] (166) at (10, 1.5) {};
		\node [style=none] (167) at (10, -1.75) {};
		\node [style=none] (168) at (-7, -7.5) {};
		\node [style=none] (169) at (-4.25, -9) {};
		\node [style=none] (170) at (4.25, -9) {};
		\node [style=none] (171) at (7.5, -7.5) {};
		\node [style=none] (172) at (0, 8.25) {};
		\node [style=none] (173) at (0, 4.75) {};
		\node [style=none] (174) at (0, -9.25) {};
		\node [style=none] (175) at (0, -6.5) {};
		\node [style=none] (176) at (-8.5, -4.75) {};
		\node [style=none] (177) at (-5, -2.75) {};
		\node [style=none] (178) at (-8.5, 4.25) {};
		\node [style=none] (180) at (8.75, 4.25) {};
		\node [style=none] (181) at (5.25, 2.5) {};
		\node [style=none] (182) at (9, -4.75) {};
		\node [style=none] (183) at (5, -2.75) {};
		\node [style=none] (184) at (-5, 2.5) {};
		\node [style=none] (185) at (0.75, 5.5) {};
		\node [style=none] (186) at (-1, 4.25) {};
		\node [style=none] (187) at (-1.25, -5.5) {};
		\node [style=none] (188) at (0.75, -7) {};
		\node [style=none] (189) at (-6, 2.5) {};
		\node [style=none] (190) at (-4.5, 3) {};
		\node [style=none] (191) at (5.5, -2.25) {};
		\node [style=none] (192) at (5.5, -4) {};
		\node [style=none] (193) at (-0.5, 1) {};
		\node [style=none] (194) at (-1.5, 1.5) {};
		\node [style=none] (195) at (-6, -2.25) {};
		\node [style=none] (196) at (-4.5, -3.25) {};
		\node [style=none] (197) at (5.5, 1.5) {};
		\node [style=none] (198) at (-1.75, -2) {};
		\node [style=none] (199) at (-0.75, -1.5) {};
		\node [style=none] (200) at (1.25, -0.5) {};
		\node [style=none] (201) at (2.25, 0) {};
		\node [style=none] (202) at (5.5, 3.5) {};
		\node [style=none] (209) at (6, 8.5) {$\Theta_{1,2}$};
		\node [style=none] (210) at (11, 0) {$\Theta_{4,5}$};
		\node [style=none] (211) at (6.25, -9) {$\Theta_{7,8}$};
		\node [style=none] (212) at (-6.25, -9) {$\Theta_{10,11}$};
		\node [style=none] (213) at (-11, 0) {$\Theta_{13,14}$};
		\node [style=none] (214) at (-6, 8.5) {$\Theta_{16,17}$};
		\node [style=none] (215) at (0, 3.75) {};
		\node [style=none] (216) at (-3.25, -4) {};
		\node [style=none] (217) at (5.75, 0.75) {$\color{red}\Sigma_{1,4}$};
		\node [style=none] (218) at (-0.25, -5) {$\color{red}\Sigma_{0,3}$};
		\node [style=none] (219) at (3, -2.5) {$\color{red}\Sigma_{2,5}$};
		\node [style=none] (220) at (1.5, 3.5) {$\color{red}\tilde{F}_{a}$};
	\end{pgfonlayer}
	\begin{pgfonlayer}{edgelayer}
		\draw [bend left, looseness=0.75] (123.center) to (124.center);
		\draw [bend left, looseness=0.75] (125.center) to (126.center);
		\draw [bend left, looseness=0.75] (127.center) to (128.center);
		\draw [bend right, looseness=0.75] (129.center) to (130.center);
		\draw [bend right, looseness=0.75] (131.center) to (132.center);
		\draw [bend right, looseness=0.75] (133.center) to (134.center);
		\draw [bend left=45, looseness=0.75] (135.center) to (136.center);
		\draw [bend left=45, looseness=0.75] (137.center) to (138.center);
		\draw [bend left=45, looseness=0.75] (139.center) to (140.center);
		\draw [bend right, looseness=0.75] (141.center) to (142.center);
		\draw [bend right, looseness=0.75] (144.center) to (145.center);
		\draw [bend right=45, looseness=0.75] (146.center) to (147.center);
		\draw [bend right=45, looseness=0.75] (148.center) to (149.center);
		\draw [bend right=45, looseness=0.75] (150.center) to (151.center);
		\draw [bend left, looseness=0.75] (152.center) to (153.center);
		\draw [bend left, looseness=0.75] (154.center) to (155.center);
		\draw [bend left, looseness=0.75] (156.center) to (157.center);
		\draw [bend right=45, looseness=0.75] (140.center) to (139.center);
		\draw [bend right=45, looseness=0.75] (138.center) to (137.center);
		\draw [bend right=45, looseness=0.75] (136.center) to (135.center);
		\draw [bend left=45, looseness=0.75] (151.center) to (150.center);
		\draw [bend left=45, looseness=0.75] (149.center) to (148.center);
		\draw [bend left=45, looseness=0.75] (147.center) to (146.center);
		\draw [bend right, looseness=0.75] (158.center) to (159.center);
		\draw (160.center) to (161.center);
		\draw (162.center) to (163.center);
		\draw (165.center) to (164.center);
		\draw (167.center) to (166.center);
		\draw (169.center) to (168.center);
		\draw (171.center) to (170.center);
		\draw [style=blue] (172.center) to (173.center);
		\draw [style=blue] (180.center) to (181.center);
		\draw [style=blue] (182.center) to (183.center);
		\draw [style=blue] (174.center) to (175.center);
		\draw [style=blue] (176.center) to (177.center);
		\draw [style=blue] (178.center) to (184.center);
		\draw [style=red, bend left=45, looseness=1.25] (188.center) to (187.center);
		\draw [style=red] (187.center) to (186.center);
		\draw [style=red, bend left=45] (186.center) to (185.center);
		\draw [style=red, bend left=45, looseness=1.75] (189.center) to (190.center);
		\draw [style=red, bend right=60, looseness=1.50] (192.center) to (191.center);
		\draw [style=red] (190.center) to (194.center);
		\draw [style=red] (193.center) to (191.center);
		\draw [style=red, bend right=45, looseness=1.25] (195.center) to (196.center);
		\draw [style=red] (196.center) to (198.center);
		\draw [style=red] (199.center) to (200.center);
		\draw [style=red] (201.center) to (197.center);
		\draw [style=red, bend right=60, looseness=1.50] (197.center) to (202.center);
		\draw [style=red, bend left, looseness=1.25] (216.center) to (215.center);
	\end{pgfonlayer}
\end{tikzpicture}
    \caption{Components of $\tilde{X}$}
    \label{X_tilde-R_tilde}
\end{figure}

Let $\tilde{R} = \tilde{X}/\tilde{\sigma}$, and $\tilde{\pi}: \tilde{X} \rightarrow \tilde{R}$ the quotient map. For each component above, we calculate the pushforward by $\tilde{\pi}$ given as follows.

\begin{center}
\begin{tabular}{llll}
   
   $\tilde{\pi}_{*}\tilde{\Theta}_{0} = \tilde{E}_1$    &
   $\tilde{\pi}_{*}\tilde{\Theta}_{9} = \tilde{\ell}_1$ &
   $\tilde{\pi}_{*}\tilde{\Sigma}_{0} = 3\tilde{H}_4$    &
   $\tilde{\pi}_{*}\Theta_{1,2} = R_1$  \\
   $\tilde{\pi}_{*}\tilde{\Theta}_{1} = 3S_1$            &
   $\tilde{\pi}_{*}\tilde{\Theta}_{10} = 3S_7$           &
   $\tilde{\pi}_{*}\tilde{\Sigma}_{1} = 3\tilde{H}_3$    &
   $\tilde{\pi}_{*}\Theta_{4,5} = R_2$  \\
   $\tilde{\pi}_{*}\tilde{\Theta}_{2} = 3S_2$            &
   $\tilde{\pi}_{*}\tilde{\Theta}_{11} = 3S_8$           &
   $\tilde{\pi}_{*}\tilde{\Sigma}_{2} = 3\tilde{H}_5$    &
   $\tilde{\pi}_{*}\Theta_{7,8} = R_3$  \\
   $\tilde{\pi}_{*}\tilde{\Theta}_{3} = \tilde{\ell}_3$ &
   $\tilde{\pi}_{*}\tilde{\Theta}_{12} = \tilde{E}_3$   &
   $\tilde{\pi}_{*}\tilde{\Sigma}_{3} = 3\tilde{H}_1$    &
   $\tilde{\pi}_{*}\Theta_{10,11} = R_4$\\
   $\tilde{\pi}_{*}\tilde{\Theta}_{4} = 3S_3$            &
   $\tilde{\pi}_{*}\tilde{\Theta}_{13} = 3S_9$           &
   $\tilde{\pi}_{*}\tilde{\Sigma}_{4} = 3\tilde{H}_6$    &
   $\tilde{\pi}_{*}\Theta_{13,14} = R_5$\\
   $\tilde{\pi}_{*}\tilde{\Theta}_{5} = 3S_4$            &
   $\tilde{\pi}_{*}\tilde{\Theta}_{14} = 3S_{10}$        &
   $\tilde{\pi}_{*}\tilde{\Sigma}_{5} = 3\tilde{H}_2$    &
   $\tilde{\pi}_{*}\Theta_{16,17} = R_6$\\
   $\tilde{\pi}_{*}\tilde{\Theta}_{6} = \tilde{E}_2$    &
   $\tilde{\pi}_{*}\tilde{\Theta}_{15} = \tilde{\ell}_2$&
   $\tilde{\pi}_{*}\Sigma_{0,3} = \tilde{r}_1$          &
   $\tilde{\pi}_{*}\tilde{F}_a = 3E_O$\\
   $\tilde{\pi}_{*}\tilde{\Theta}_{7} = 3S_5$            &
   $\tilde{\pi}_{*}\tilde{\Theta}_{16} = 3S_{11}$        &
   $\tilde{\pi}_{*}\Sigma_{1,4} = \tilde{r}_3$&\\
   $\tilde{\pi}_{*}\tilde{\Theta}_{8} = 3S_6$            &
   $\tilde{\pi}_{*}\tilde{\Theta}_{17} = 3S_{12}$        &
   $\tilde{\pi}_{*}\Sigma_{2,5} = \tilde{r}_2$&\\
\end{tabular}
\end{center}

The intersection pattern between distinct components is maintained, but the self-intersections are changed. The pushforwards of curves fixed by $\tilde{\sigma}$ have the self-intersection multiplied by 3, and the pushforward of curves preserved by $\tilde{\sigma}$ but not fixed have self intersection divided by 3. We obtain $\tilde{E}_i^2 = \tilde{\ell}_i^2 = -6$, $R_i^2 = \tilde{r}_i^2 = -3$, and $S_i^2 = \tilde{H}_i^2 = E_O^2 = -1.$ Notice that even though $g(\tilde{F}_a) = 1$, since $\tilde{\sigma}$ fixes 3 distinct points in $\tilde{F}_a$, by Riemann-Hurwitz we have $g(E_O) = 0$.

If we contract the $(-1)$-curves $S_1,...,S_{12}$ and $E_O$, and subsequently contract the curves $R_1,...,R_6$, we obtain a birational morphism $\varepsilon \colon \tilde{R} \to R$, with $\varepsilon_{*} \tilde{E}_i = E_i, \varepsilon_{*} \tilde{\ell}_i = \ell_i, \varepsilon_{*} \tilde{H}_i = H_i$ and $\varepsilon_{*} \tilde{r}_i = r_i$. Since the canonical divisor of $R$ is equal to $-F$ in $\NS R$, we can use $\varepsilon$ to calculate the canonical divisor of $\tilde{R}$

\begin{align*}
    K_{\tilde{R}} = \textstyle\sum_{i=1}^{6}R_i + 2\sum_{i=1}^{12}S_i -\sum_{i=1}^3 \tilde{r_i} - 2E_O.
\end{align*}

\begin{thm}\label{main_classification}
    In $(X,\sigma)$, the elliptic fibrations $\cE_1$ and $\cE_5$ in Proposition \ref{induced_fibration} are of type 1 and $\cE_2, \cE_3,\cE_4,\cE_6$ are of type 2.
\end{thm}

\begin{proof}
    For every $L_i$ in Table \ref{Divisors_elliptic}, apply the first pullback by $\eta$, then the pushforward by $\tilde{\pi}$. The divisor $\tilde{\pi}_*(\eta^*L_i)$ induces a linear system in $\tilde{R}$, which by Theorem \ref{class_linearsys} can be used to classify $\cE_i$ with respect to the action of $\sigma$. If $\cE_i$ is of type 1, then $\tilde{\pi}_*(\eta^*L_i) = 3\tilde{L}_i$, where $\tilde{L_i}$ is a generalized conic bundle class in $\tilde{R}$, so $\tilde{\pi}_*(\eta^*L_i)\cdot K_{\tilde{R}} = 3\tilde{L_i} \cdot K_{\tilde{R}} = 3 \cdot (-2) = -6$. If $\cE_i$ is of type 2, then $\tilde{\pi}_*(\eta^*L_i) = \tilde{L}_i$ is a splitting genus 1 pencil, and $\tilde{L_i}\cdot K_{\tilde{R}} = 0$. Since $\sigma$ acts trivially on $\NS X$, there are no fibrations of type 3. The explicit calculations of $\tilde{\pi}_*(\eta^*L_i)$ and its intersection with $K_{\tilde{R}}$ are presented in Table \ref{tilde_Li}.
\end{proof}

\begin{table}[ht]
    \centering
    \begin{tabular}{|c|c|c|}
    \hline
        No. & $\tilde{\pi}_*(\eta^*L_i)$ & $\tilde{\pi}_*(\eta^*L_i)\cdot K_{\tilde{R}}$ \\
    \hline
        $1$ & \small${\begin{matrix}
            & \\
            3\tilde{L}_1 = & 3(2S_{11} + 2R_6 + 4S_{12} + 2\tilde{E}_1 + 3 \tilde{H}_4 + \tilde{r}_1 \\
            & + 5S_1 + 3R_1 + 4S_2 + \tilde{\ell}_3 + 2S_3 + R_2 + S_4)\\
        \end{matrix}}$ & $-6$\\
        \hline
        $\begin{matrix}
        2\\
        \\
        \\
        \end{matrix}$& \small$\begin{matrix}
            &\\
            \tilde{L}_2 = & R_6 + 3S_{12} + \tilde{r}_1 + 3\tilde{H}_4 + 2\tilde{E}_1 + 6S_1 + 4R_1\\
            & + 6S_2 + 2\tilde{\ell}_3 + 6S_3 + 4R_2 + 6S_4 + 2\tilde{E}_2\\
            & + 6S_5 + 4R_3 + 6S_6 + 2\tilde{\ell}_1 + 6S_7 + 4R_4\\
            & + 6S_8 + 2\tilde{E}_3 + 3\tilde{H}_6 + \tilde{r}_3 + 3S_9 + R_5\\
        \end{matrix}$ & $0$\\
        \hline
        $3$ & \small$\begin{matrix}
            &\\
            \tilde{L}_3 = & \tilde{\ell}_{2} + 6S_{11} + 5R_6 + 9S_{12} + 4\tilde{E}_1 + 6\tilde{H}_4\\
            & + 2\tilde{r}_1 + 9S_1 + 5R_1 + 6S_2 + \tilde{\ell}_3
        \end{matrix}$ & $0$\\
        \hline
        $\begin{matrix}
        \\
        4\\
        \\
        \end{matrix}$& \small$\begin{matrix}
            &\\
            \tilde{L}_4 = & \sum_{i=1}^{3} \tilde{E}_i + \sum_{i=1}^{3} \tilde{\ell}_i + 3\sum_{i=1}^{12} S_i + 2 \sum_{i=1}^{6} \tilde{R}_i\\
            &\\
        \end{matrix}$& $0$\\
        \hline
        $5$ & \small$\begin{matrix}
            &\\
            3\tilde{L}_5 = & 3(\tilde{E}_1 + 2S_1 + R_1 + S_2 + 2S_{12}\\
            & + R_6 + S_{11} + 2\tilde{H}_4 + \tilde{r}_1 + \tilde{H}_1)\\
        \end{matrix}$ & $-6$\\
        \hline
        $\begin{matrix}
        6\\
        \\
        \end{matrix}$ & \small$\begin{matrix}
            &\\
            \tilde{L}_6 = & R_6 + 3S_{12} + \tilde{r}_1 + 3\tilde{H}_4 + 2\tilde{E}_1 + 6S_1 + 4 R_1\\
            & + 6S_2 + 2\tilde{\ell}_3 + 3\tilde{H}_3 + \tilde{r}_3 + 3S_3 + R_2\\
            & \\
        \end{matrix}$ & $0$\\
        \hline
    \end{tabular}
    \caption{Divisors induced in $\tilde{R}$}
    \label{tilde_Li}
\end{table}
\subsection{Weierstrass Equations of the elliptic fibrations}\label{section_pencils}

In this section, we apply the method described in Section \ref{weierstrass_section} to find Weierstrass equations for each $\cE_i$ in Table \ref{Divisors_elliptic}. Let $\Gamma_i$ be the pencils of curves in $\PP^2$ induced by each $\cE_i$. We describe the geometry of $\Gamma_i$ by applying the following substitutions for each $\tilde{L}_i$ in Table \ref{tilde_Li}, where $\tilde{h}$ is the pullback of the line class $h \in \Pic(\PP^2)$ to $\Pic(\tilde{R})$.
{\footnotesize
\begin{align*}
    \tilde{l_1} = \ &\tilde{h} - \tilde{E_2} - \tilde{E_3} - \tilde{H_1} - \tilde{H_5} - \tilde{H_6} - R_2 - 2R_3 - 2R_4 - R_5\\ &- S_3 - 2S_4 - 3S_5 - 3S_6 - 3S_7 - 3S_8 - 2S_9 - S_{10},\\
    \tilde{l_2} = \ &\tilde{h} - \tilde{E_1} - \tilde{E_3} - \tilde{H_2} - \tilde{H_4} - \tilde{H_6} - R_1 - R_4 - 2R_5 - 2R_6\\ &- 2S_1 - S_2 - S_7 - 2S_8 - 3S_9 - 3S_{10} - 3S_{11} - 3S_{12},
\end{align*}
\begin{align*}
    \tilde{l_3} = \ &\tilde{h} - \tilde{E_1} - \tilde{E_2} - \tilde{H_3} - \tilde{H_4} - \tilde{H_5} - 2R_1 - 2R_2 - R_3 - R_6\\ &- 3S_1 - 3S_2 - 3S_3 - 3S_4 - 2S_5 - S_6 - S_{11} - 2S_{12},\\
    \tilde{r_1} = \ &\tilde{h} - E_{O} - \tilde{E_1} - \tilde{H_1} - 2\tilde{H_4} - R_1 - R_6 - 2S_1 - S_2 - S_{11} - 2S_{12},\\
    \tilde{r_2} = \ &\tilde{h} - E_{O} - \tilde{E_2} - \tilde{H_2} - 2\tilde{H_5} - R_2 - R_3 - S_3 - 2S_4 - 2S_5 - S_{6},\\
    \tilde{r_3} = \ &\tilde{h} - E_{O} - \tilde{E_3} - \tilde{H_3} - 2\tilde{H_6} - R_4 - R_5 - S_7 - 2S_8 - 2S_{9} - S_{10}.
\end{align*}}
As an example, applying these substitutions to $\tilde{L_1}$ we obtain $\tilde{L_1} =  2\tilde{h} - E_O - \tilde{E_2} - \tilde{H_1} - \tilde{H_3} - \tilde{H_5} - R_2 - R_3 - S_3 - 2S_4 - 2S_5 - S_6$. The component $2\tilde{h}$ indicates that the linear system $\Gamma_1$ in $\PP^2$ is composed of conics. The negative components $E_O, E_2, H_1$ and $H_3$ indicate that the conics in $\Gamma_1$ pass through the points $O, P_2, Q_1$ and $Q_3$ in $\PP^2$, and the remaining components come from the pullback of the former. Consequently, $\Gamma_1$ is the pencil of conics through $O, P_2, Q_1, Q_3$. We can do the same for the rest of the divisors $\tilde{L_i}$ (see Table \ref{Gamma_i}).

\begin{table}[ht]
    \centering
    \begin{tabular}{|c|c|c|}
    \hline
        No. & $\tilde{L_i}$ & $\Lambda_i$ \\
    \hline
        $\begin{matrix}
        1\\ \\
        \end{matrix}$ & \small${\begin{matrix}
            & \\
            \tilde{L_1} = \ &2\tilde{h} - E_O - \tilde{E_2} - \tilde{H_1}\\
            & - \tilde{H_3} - \tilde{H_5}
             - R_2 - R_3\\
            & - S_3 - 2S_4 - 2S_5 - S_6\\
        \end{matrix}}$ & $\begin{matrix}
            \text{Conics through}\\
            P_2, Q_1,Q_3, O.
        \end{matrix}$\\
        \hline
        $\begin{matrix}
        2\\
        \\ \\ \\ \\
        \end{matrix}$& \small$\begin{matrix}
            &\\
            \tilde{L_2} = \ &6 \tilde{h} - 2E_O - \tilde{E_1} - 2\tilde{E_2} - \tilde{E_3}\\
            & - 3\tilde{H_1} - 3\tilde{H_3} - \tilde{H_4} - 4 \tilde{H_5} - \tilde{H_6}\\
            & - R_1 - 2R_2 - 2R_3 - R_4- 2R_5\\
            & - 2R_6 - 2S_1 - S_2 - 2S_3 - 4S_4 \\
            & - 4S_5 - 2S_6 - S_7 - 2S_8 - 3S_9\\
            & - 3S_{10} - 3S_{11} - 3S_{12}\\
        \end{matrix}$ & $\begin{matrix}
            \text{Sextics through}\\
            Q_1, Q_3 \ \text{(multiplicity 3)},\\
            P_2 \ \text{(tacnode tangent to $r_2$)},\\
            O \ \text{(multiplicity 2)},\\
            P_1, P_3 \ \text{(tangent to $l_2$).}
        \end{matrix}$\\
        \hline
        $\begin{matrix}
        3\\ \\ \\ \\
        \end{matrix}$ & \small$\begin{matrix}
            &\\
            \tilde{L_3} = \ &4\tilde{h} - 2E_O - \tilde{E_2} - \tilde{E_3}
            \\ & - 2\tilde{H_1} - \tilde{H_2} - \tilde{H_3} - \tilde{H_5} - \tilde{H_6}\\
            & - 2R_2 - R_3 - R_4 - 2R_5\\
            &- 3S_3 - 3S_4 - 2S_5 - S_6\\
            & - S_7 - 2S_8 - 3S_9 -3S_{10}\\
        \end{matrix}$ & $\begin{matrix}
            \text{Quartics through}\\
            O, Q_1 \ \text{(multiplicity 2)},\\
            P_2, P_3, Q_2, Q_3 \ \text{(multiplicity 1).}
        \end{matrix}$\\
        \hline
        $\begin{matrix}
        \\ 4 \\ \\ \\ \\ \\
        \end{matrix}$&
        \small$\begin{matrix}
            &\\
            \tilde{L_4} =  &3\tilde{h} - \tilde{E_1} - \tilde{E_2} - \tilde{E_3} - \tilde{H_1}
            - \tilde{H_2}\\
            &- \tilde{H_3} - 2\tilde{H_4} - 2\tilde{H_5} - 2\tilde{H_6}
            - R_1 - R_2\\
            & - R_3 -R_4 - R_5 - R_6 - 2S_1 - S_2\\
            & - S_3 - 2S_4 - 2S_5 - S_6 - S_7\\
            & - 2S_8 - 2S_9 - S_{10} - S_{11} - 2S_{12}\\   
        \end{matrix}$ & $\begin{matrix}
            \text{Cubics through } Q_1,Q_2,Q_3 \\
            P_1 \ \text{(tangent to } r_1), \\
            P_2 \ \text{(tangent to } r_2), \\
            P_3 \ \text{(tangent to } r_3). \\
        \end{matrix}$\\
        \hline
        $5$ & \small $\begin{matrix}
        \\ \tilde{L_5} = \tilde{h} - E_O\\ \\ 
        \end{matrix}$ & Lines through $O$.\\
        \hline
        $\begin{matrix}
        \\
        6\\ \\ \\ \\ \\
        \end{matrix}$ & \small$\begin{matrix}
            &\\
            \tilde{L_6} = & 4\tilde{h} - 2E_O - \tilde{E_1} - 2\tilde{E_2} - \tilde{E_3}\\
            &- \tilde{H_1} -\tilde{H_4} - 2\tilde{H_5} - 2\tilde{H_6} - R_1\\
            & - 2R_2 - 2R_3 - R_4 - R_5 -2R_6\\
            & -2S_1 - S_2 - 3S_3 - 6S_4 - 4S_5 - 2S_6\\
            & - S_7 - 2S_8 - 2S_9 - S_{10} - 2S_{11} - S_{12}\\
        \end{matrix}$ & $\begin{matrix}
            \text{Quartics through}\\
            O, P_2 \ \text{(multiplicity 2)},\\
            P_1, Q_1, \ \text{(multiplicity 1)},\\
            P_3 \ \text{(tangent to } r_3).\\
            \end{matrix}$\\
        \hline
    \end{tabular}
    \caption{Linear systems $\Lambda_i$}
    \label{Gamma_i}
\end{table} 

The description of these systems allows us to use Propositions \ref{equation_type2} and \ref{equation_type1} and find Weierstrass equations for each fibration in Proposition \ref{Divisors_elliptic}.

\begin{prop}\label{weierstrass_eq}
    The elliptic fibrations $\cE_2$, $\cE_3$ and $\cE_6$ are given by the following equations in Weierstrass form:
    \small\begin{align*}
        \cE_2: y^2 &+ \tfrac{2t^6}{t^3 - 3}xy + \tfrac{2t^{15}}{t^9 - 9t^6 + 27t^3 - 27}y\\
          &= x^3 + \tfrac{-3t^{12} + 2t^9}{t^6 - 6t^3 + 9}x^2 + \tfrac{t^{24} - 6t^{21} + t^{18}}{t^{12} - 12t^9 + 54t^6 - 108t^3 + 81}x + \tfrac{t^{30}}{t^{15} - 15t^{12} + 90t^9 - 270t^6 + 405t^3 - 243},\\
        \cE_3: y^2 &+ \tfrac{2}{t^6 + t^3}xy - \tfrac{2}{t^{15} + 3t^{12} + 3t^9 + t^6}y = x^3 - \tfrac{2t^3 + 1}{t^{12} + 2t^9 + t^6}x^2 + \tfrac{1}{t^{18} + 3t^{15} +    3t^{12} + t^9}x,\\
        \cE_6: y^2 &- \tfrac{1}{t^3 - 1}xy - \tfrac{t^3}{t^6 - 2t^3 + 1}y=x^3 + \tfrac{3t^6 - t^3}{t^6 - 2t^3 + 1}x^2 + \tfrac{3t^{12} - 2t^9}{t^{12} - 4t^9 + 6t^6 - 4t^3 + 1}x + \tfrac{t^{15}}{t^{15} - 5t^{12} + 10t^9 - 10t^6 + 5t^3 - 1}.
    \end{align*}
\end{prop}

\begin{proof}
    By Proposition \ref{equation_type2}, for each $\cE_i$, we need to find $F_{i,a}$ and $F_{i,b}$ the fibres above the ramification points of the base change, and then calculate the curves $\cC_{i,a}$ and $\cC_{i,b}$ they induce in $\Gamma_i$. Notice that the ADE-types of each fibration determines the Kodaira types of the ramified fibres, as noted in Remark \ref{remark_ADE}. We can take $F_{i,a}$ to be the divisor $L_i$ (see Table \ref{Divisors_elliptic}). Then, $F_{i,b}$ must be disjointed from $F_{i,a}$ and the same Kodaira type of the ramified fibre. For each $\cE_i$, we can choose as follows:
    \begin{align*}
        & F_{2,b} = \Theta_{15} + \Sigma_5 + \Sigma',\\
        & F_{3,b} = \Theta_5 + \Sigma_2 + 2(\textstyle\sum_{i=6}^{12} \Theta_i) + \Theta_{13} + \Sigma_4,\\
        & F_{6,b} = \Sigma_2 + \textstyle\sum_{i=6}^{15} \Theta_i + \Sigma_5.
    \end{align*}

    Here, $\Sigma'$ is a section of $\cE_X$ intersecting $F_b^X$ in the component $\Theta_{15}$ and $F_a^{X}$ in the same point as $\Sigma_5$. Next, we calculate $\tilde{\pi}_*(\eta^* F_{i,a})$ and $\tilde{\pi}_*(\eta^* F_{i,b})$ in $\Pic (\tilde{R})$. The components coming from curves in $\PP^2$ will determine $\cC_{i,a}$ and $\cC_{i,b}$. For example, $\tilde{\pi}_*(\eta^* F_{3,a}) = \tilde{L}_3 = \tilde{\ell}_2 + 2 \tilde{r}_1 + \tilde{\ell}_3 + (4\tilde{E}_1 + 6\tilde{H}_4 + 9R_1 + 5 R_6 + 9S_1 + 6S_2 + 6S_2 + 6S_{11} + 9 S_{12})$. Then, the induced curve $\cF_3$ is given by $\ell_2 \cdot r_1^2 \cdot \ell_3^2 = x(y-z)^2z = 0$.
    
    This process is straightforward for every $F_{i,a}$, as well as for $F_{3,b}$ and $F_{6,b}$. For $F_{2,b}$, we first have to calculate $\tilde{\pi}_*(\eta^*(\Sigma'))$. Since $\Sigma'$ passes through the same fixed point of $\sigma$ as $\Sigma_5$, it follows that $\eta^*(\Sigma') = \tilde{\Sigma}' + \Sigma_{2,5}$, and $\tilde{\Sigma}^{'2} = -3$. As $\tilde{\sigma}$ preserves $\tilde{\Sigma}'$, we calculate $\tilde{\pi}_*(\eta^*(\Sigma')) = 3H' + \tilde{r}_2$, where $H' = \tilde{\pi}(\tilde{\Sigma'})$ corresponds with the section of $\tilde{R}$ coming from the line $x+z= y$ in $\PP^2$ passing through $Q_1$ and $Q_2$, and $H^{'2} = -1$. Consequently, we have $\tilde{\pi}_*(\eta^*F_{2,b}) = \tilde{\ell}_2 + 3 \tilde{H}_2 + 2\tilde{r}_2 + 3H'$.

    We arrive at the following equations for $\cE_2, \cE_3$ and $\cE_6:$
    \begin{align*}
        \cE_2: \ &x^2z^2(x-y)(y-z) + t^3y(z-x)^2(x-y+z)^3 = 0,\\
        \cE_3: \ &xy(x-y)^2 + t^3 z^2(z-x)(y-z) = 0, \\
        \cE_6: \ & xz(y-z)^2 +t^3 y^2(z-x)(x-y) = 0. 
    \end{align*}
    Putting these equations in Weierstrass form, we obtain the result.
\end{proof}

\begin{prop}\label{weierstrass_eq_2}
    The elliptic fibrations $\cE_1$ and $\cE_5$ are given by the following equations in Weierstrass form:
    \small
    
    \begin{align*}
        \cE_1 \colon y^2 &- \tfrac{v^5 - 3v^4 + 3v^3 + v^2}{v^7 - 10v^6 + 45v^5 - 117v^4 + 189v^3 - 189v^2 + 108v - 27}y\\
        & = x^3 - \tfrac{\scriptstyle v^5}{\scriptscriptstyle v^{12} - 17v^{11} + 36v^{10} - 675v^9 + 2310v^8 - 5733v^7 + 10566v^6 - 14553v^5 + 14850v^4 - 10935v^3 + 5508v^2 - 1701v + 243},\\
        \cE_5 \colon y^2 &+ \tfrac{-2v^4 + 4v^3}{\stackanchor{\scriptstyle v^{13} - 16v^{12} + 120v^{11} - 554v^{10} + 1742v^9 - 3903v^8 + 6337v^7}{\scriptstyle - 7435v^6 + 6171v^5 - 3470v^4 + 1229v^3 - 246v^2 + 25v - 1}}y\\
        &= x^3 + \tfrac{v^5}{\stackanchor{\scriptscriptstyle v^{22}{-}27v^{21}{+}351v^{20}{-}2915v^{19}{+}17310v^{18} - 77975v^{17}{+}275920v^{16}{-}783765v^{15}{+}1811095v^{14}{-}3429800v^{13}{+}5338045v^{12}}{\scriptscriptstyle -6819500v^{11}{+}7115140v^{10}{-}6008210v^9{+}4051240v^8{-}2141208v^7{+}865711v^6{-}259643v^5{+}55665v^4{-}8165v^3{+}771v^2{-}42v{+}1}}.
    \end{align*}
    
\end{prop}

\begin{proof}
    By Proposition \ref{equation_type1}, we need to calculate the restriction of $\cE_{\tilde{R}} \colon \tilde{R} \to \PP^1$ to each $D_{i,v}$ in $\Lambda_i$, which we write as $f_{i,v} \colon D_{i,v} \to \PP^1$. To do this, we determine, for each $C_{i,v} \subset \PP^2$ induced by $D_{i,v}$ in $\Gamma_i$, a map $\rho_{i,v} \colon \PP^1 \to \PP^2$ such that $\rho_{i,v} (\PP^1) = C_{i,v}$. We can write these maps as follows.
    \begin{align*}
        \rho_{1,v}([u_1{:}u_2]) &= [u_1 - vu_2 : u_1-vu_1 : u_2 - vu_2],\\
        \rho_{5,v}([u_1{:}u_2]) &= [u_1^2 - vu_1u_2 : u_1^2 - vu_2^2 : u_1u_2 - vu_2^2].
    \end{align*}
    By the construction of $R$ as a pencil of cubics, the elliptic fibration $\cE_{\tilde{R}}$ is equivalent in an open set of $\tilde{R}$ to blowing down to $\PP^2$ and composing with the rational map $\varphi([x{:}y{:}z]) = [xyz : (x{-}y)(y{-}z)(z{-}x)]$. Composing each $\rho_{i,v}$ with $\varphi$, we obtain maps $\varphi \circ \rho_{i,v} \colon \PP^1 \cong C_{i,v} \dashrightarrow \PP^1$, which we write as
    \begin{align*}
        \varphi \circ \rho_{1,v}([u_1{:}u_2]) &= [(v-1)(u_1-vu_2)u_1u_2 : v(u_1-u_2)^3],\\
        \varphi \circ \rho_{5,v}([u_1{:}u_2]) &= [(u_1-vu_2)^2(u_1^2 - vu_2^2)u_1u_2 : v(u_1-u_2)^3(u_1-vu_2)u_1u_2].
    \end{align*}
    The maps $f_{1,v}$ and $f_{5,v}$ are the resolutions of indeterminacy of the maps above. Since the coordinates in $\varphi \circ \rho_{1,v}$ have no common factors, we have that $f_{1,v} = \varphi \circ \rho_{1,v}$. The coordinates of $\varphi \circ \rho_{5,v}$ have a common factor of $(u_1-vu_2)u_1u_2$, so we have
    \begin{align*}
        f_{5,v}([u_1{:}u_2]) = [(u_1-vu_2)(u_1^2-vu_2^2) : v(u_1-u_2)^3].
    \end{align*}
    The equations for the fibrations are given by $f_{i,v}([u_1{:}u_2]) = \tau([s{:}t])$ in $\PP^1 \times \PP^1$ over $k(v)$, thus we have
    \begin{align*}
        \cE_1 &\colon s^3 v (u_1-u_2)^3 = (v-1) t^3 (u_1-vu_2)u_1u_2,\\
        \cE_5 &\colon s^3 v (u_1-u_2)^3 = t^3(u_1-vu_2)(u_1^2-vu_2^2).
    \end{align*}
    Putting these equations in Weierstrass form, we obtain the result.
\end{proof}

\begin{remark}
    The action of $\sigma$ can be given explicitly in the Weierstrass equations in the previous propositions. For the equations in Proposition \ref{weierstrass_eq}, $\sigma$ is given by $(x,y,t) \mapsto (x,y,\zeta_3 t)$, and for those in Proposition \ref{weierstrass_eq_2}, $\sigma$ is $(x,y,t) \mapsto (\zeta_3 x,y,t)$.
\end{remark}

\subsection{Other automorphism of $X_3$} \label{other_aut_X3}

In the previous sections, we have studied the automorphism $\sigma \in \Aut(X)$ obtained by constructing $X$ as the base change of a specific rational elliptic surface $R$. This is not the only non-trivial automorphism in $X$; in fact, it admits infinitely many automorphisms (see \cite{Shioda-Inose}, Theorem 5). We have already seen different ways of constructing $X$ as a base change in Proposition \ref{weierstrass_eq}. However, since these constructed fibrations are of type 2 with relation to $\sigma$, the automorphisms induced by them are equal to $\sigma$. In this section, we explore different ways to construct $X$ by base change, and provide an example leading to a distinct non-symplectic automorphism of order 3.

\begin{prop}
    It is not possible to obtain the elliptic fibration $\cE_1 \colon X \to \PP^1$ as a base change of degree 3 of a rational elliptic surface.
\end{prop}

\begin{proof}
    The fibration $\cE_1$ has two fibres of type $II^*$, so by Proposition \ref{fibres_type_2} it cannot be of type 2 with relation to any non-symplectic automorphism of order 3.
\end{proof}

\begin{remark}
    It is possible to construct $\cE_1$ as a base change of order 2 of a rational elliptic surface. If $\cE'_1 \colon R_1 \to \PP^1$ is a rational elliptic surface with one $II^*$ and two $I_1$ fibres, we can take the base change by a map of degree 2 ramifying over an $I_0$ and an $I_1$ fibre, obtaining $\cE_1 \colon X \to \PP^1$. This also induces a non-symplectic involution in $X$.
\end{remark}

\begin{prop}\label{ce5_type_2}
    We can construct $\cE_5 \colon X \to \PP^1$ as the base change of a rational elliptic surface $\cE'_5 \colon R_5 \to \PP^1$ such that the induced automorphism $\sigma_5$ does not act trivially on $\NS X$. Consequently, $\sigma_5$ is different from $\sigma$ as an automorphism of $X$.
\end{prop}

\begin{proof}
    Let $\Lambda'_5$ be a cubic pencil in $\PP^2$ generated by the cubics $z^3 = 0$ and $xy(x-y) = 0$. This pencil induces a rational elliptic surface $\cE'_5 \colon R_5 \to \PP^1$ with reducible fibres $F_{5,a}$ of type $IV^*$ and $F_{5,b}$ of type $IV$. The ADE-type of this fibration is $E_6 \oplus A_2$, and by \cite{Shioda-Oguiso} we have $\MW \cE^{'}_5 = \ZZ / 3$.

    \begin{figure}[ht]
        \centering
        \begin{tikzpicture}
	\begin{pgfonlayer}{nodelayer}
		\node [style=none] (0) at (0, -4.25) {};
		\node [style=none] (1) at (0, 3.25) {};
		\node [style=none] (2) at (0.5, 0) {};
		\node [style=none] (3) at (-1.5, 0) {};
		\node [style=none] (4) at (0.5, -2.25) {};
		\node [style=none] (5) at (0.5, 2.25) {};
		\node [style=none] (6) at (-1.5, 2.25) {};
		\node [style=none] (7) at (-1.5, -2.25) {};
		\node [style=none] (8) at (-1, 2.5) {};
		\node [style=none] (9) at (-1, 1) {};
		\node [style=none] (10) at (-1, 0.25) {};
		\node [style=none] (11) at (-1, -1.5) {};
		\node [style=none] (12) at (-1, -2) {};
		\node [style=none] (13) at (-1, -3.75) {};
		\node [style=none] (32) at (-3.5, -2.25) {};
		\node [style=none] (33) at (-3.5, 2.25) {};
		\node [style=none] (42) at (-4.25, 3.25) {};
		\node [style=none] (43) at (-4.25, -4.25) {};
		\node [style=none] (44) at (-4.25, -3) {};
		\node [style=none] (45) at (-4.25, 1.5) {};
		\node [style=none] (46) at (-5, -1.5) {};
		\node [style=none] (47) at (-5, 0.5) {};
		\node [style=none] (54) at (-4.25, 3.25) {};
		\node [style=none] (59) at (-7, 1.5) {};
		\node [style=none] (60) at (-0.25, 1.5) {};
		\node [style=none] (61) at (0.25, 1.5) {};
		\node [style=none] (62) at (6.25, 1.5) {};
		\node [style=none] (68) at (3.25, -0.5) {};
		\node [style=none] (69) at (0.25, -0.5) {};
		\node [style=none] (70) at (-0.25, -0.5) {};
		\node [style=none] (71) at (-7, -0.5) {};
		\node [style=none] (76) at (-7, -3) {};
		\node [style=none] (77) at (-0.25, -3) {};
		\node [style=none] (78) at (0.25, -3) {};
		\node [style=none] (79) at (3.75, -3) {};
		\node [style=none] (82) at (-4.25, -5.25) {$F_u$};
		\node [style=none] (84) at (0, -5.25) {$F_{5,a}$};
		\node [style=none] (86) at (4, -5.25) {$F_{5,b}$};
		\node [style=none] (87) at (7, 1.5) {$H_2$};
		\node [style=none] (88) at (7, -0.5) {$H_1$};
		\node [style=none] (89) at (7, -3) {$H_0$};
		\node [style=none] (91) at (0, 4) {$\Theta_{6,a}$};
		\node [style=none] (94) at (-2.5, 2.25) {$\Theta_{5,a}$};
		\node [style=none] (98) at (-1.25, 3) {$\Theta_{2,a}$};
		\node [style=none] (100) at (-2.5, 0.25) {$\Theta_{4,a}$};
		\node [style=none] (103) at (-1.75, -1.5) {$\Theta_{1,a}$};
		\node [style=none] (106) at (-2.5, -2.25) {$\Theta_{3,a}$};
		\node [style=none] (109) at (-1.25, -4.25) {$\Theta_{0,a}$};
		\node [style=none] (118) at (4.75, -0.5) {};
		\node [style=none] (119) at (5.25, -3) {};
		\node [style=none] (121) at (6.25, -0.5) {};
		\node [style=none] (122) at (6.25, -3) {};
		\node [style=none] (132) at (2.25, 3) {$\Theta_{2,b}$};
		\node [style=none] (133) at (4, 3) {$\Theta_{1,b}$};
		\node [style=none] (134) at (6, 3) {$\Theta_{0,b}$};
		\node [style=none] (135) at (4, 1.25) {};
		\node [style=none] (136) at (4, -3.75) {};
		\node [style=none] (137) at (2.25, 2.25) {};
		\node [style=none] (138) at (5.25, -3.75) {};
		\node [style=none] (139) at (2.75, -3.75) {};
		\node [style=none] (140) at (5.25, 1.25) {};
		\node [style=none] (141) at (3.75, -0.5) {};
		\node [style=none] (142) at (4.25, -0.5) {};
		\node [style=none] (143) at (4.75, -3) {};
		\node [style=none] (145) at (4.25, -3) {};
		\node [style=none] (146) at (4, 1.75) {};
		\node [style=none] (147) at (4, 2.25) {};
		\node [style=none] (148) at (5.5, 1.75) {};
		\node [style=none] (149) at (5.75, 2.25) {};
	\end{pgfonlayer}
	\begin{pgfonlayer}{edgelayer}
		\draw (0.center) to (1.center);
		\draw (5.center) to (6.center);
		\draw (2.center) to (3.center);
		\draw (4.center) to (7.center);
		\draw (8.center) to (9.center);
		\draw (10.center) to (11.center);
		\draw (12.center) to (13.center);
		\draw [style=blue] (59.center) to (60.center);
		\draw [style=blue] (61.center) to (62.center);
		\draw [style=blue] (71.center) to (70.center);
		\draw [style=blue] (69.center) to (68.center);
		\draw [style=blue] (76.center) to (77.center);
		\draw [style=blue] (78.center) to (79.center);
		\draw (54.center) to (45.center);
		\draw [bend left, looseness=1.25] (45.center) to (47.center);
		\draw [bend right=75] (47.center) to (46.center);
		\draw [bend left, looseness=1.25] (46.center) to (44.center);
		\draw (44.center) to (43.center);
		\draw [style=blue] (119.center) to (122.center);
		\draw [style=blue] (118.center) to (121.center);
		\draw (136.center) to (135.center);
		\draw (138.center) to (137.center);
		\draw (139.center) to (140.center);
		\draw [style=blue] (141.center) to (142.center);
		\draw [style=blue] (145.center) to (143.center);
		\draw (146.center) to (147.center);
		\draw (148.center) to (149.center);
	\end{pgfonlayer}
\end{tikzpicture}
        \caption{Reducible fibres of $\cE'_5$}
        \label{R5_fibres}
    \end{figure}
    
    If $\tau' \colon \PP^1 \to \PP^1$ is a cubic Galois cover totally ramified at $b$ and another point $u \in \PP^1$, then the base change of $\cE^{'}_5$ by $\tau'$ is a K3 surface $\cE^{''}_5 \colon X_5 \to \PP^1$ with three reducible fibres $F^{X}_{5,a_1}, F^{X}_{5,a_2}, F^{X}_{5,a_3}$ of type $IV^*$, where $\tau'(a_i) = a$.  The fibre $F^{X}_{5,b_1}$ with $\tau'(b_1) = b$ is of type $I_0$, that is, it is a smooth fibre (see \cite{miranda}, Table VI.4.1). The $ADE$-type of $\cE^{''}_5$ is $E_6^{\oplus 3}$, and by \cite{Shimada} we have $\MW \cE^{''}_5 = \ZZ/3$. By the same argument as \ref{X-X_3-isom}, $X_5$ is isomorphic to $X$.

\begin{figure}[ht]
    \centering
    \begin{tikzpicture}
	\begin{pgfonlayer}{nodelayer}
		\node [style=none] (0) at (0.5, -4.5) {};
		\node [style=none] (1) at (0.5, 3.5) {};
		\node [style=none] (2) at (1, 0) {};
		\node [style=none] (3) at (-1, 0) {};
		\node [style=none] (4) at (1, -2.5) {};
		\node [style=none] (5) at (1, 2.5) {};
		\node [style=none] (6) at (-1, 2.5) {};
		\node [style=none] (7) at (-1, -2.5) {};
		\node [style=none] (8) at (-0.5, 2.75) {};
		\node [style=none] (9) at (-0.5, 1) {};
		\node [style=none] (10) at (-0.5, 0.25) {};
		\node [style=none] (11) at (-0.5, -1.5) {};
		\node [style=none] (12) at (-0.5, -2.25) {};
		\node [style=none] (13) at (-0.5, -4) {};
		\node [style=none] (14) at (5, -4.5) {};
		\node [style=none] (15) at (5, 3.5) {};
		\node [style=none] (16) at (5.5, 0) {};
		\node [style=none] (17) at (3.5, 0) {};
		\node [style=none] (18) at (5.5, -2.5) {};
		\node [style=none] (19) at (5.5, 2.5) {};
		\node [style=none] (20) at (3.5, 2.5) {};
		\node [style=none] (21) at (3.5, -2.5) {};
		\node [style=none] (22) at (4, 2.75) {};
		\node [style=none] (23) at (4, 1) {};
		\node [style=none] (24) at (4, 0.25) {};
		\node [style=none] (25) at (4, -1.5) {};
		\node [style=none] (26) at (4, -2.25) {};
		\node [style=none] (27) at (4, -4) {};
		\node [style=none] (28) at (-4, -4.5) {};
		\node [style=none] (29) at (-4, 3.5) {};
		\node [style=none] (30) at (-3.5, 0) {};
		\node [style=none] (31) at (-5.5, 0) {};
		\node [style=none] (32) at (-3.5, -2.5) {};
		\node [style=none] (33) at (-3.5, 2.5) {};
		\node [style=none] (34) at (-5.5, 2.5) {};
		\node [style=none] (35) at (-5.5, -2.5) {};
		\node [style=none] (36) at (-5, 2.75) {};
		\node [style=none] (37) at (-5, 1) {};
		\node [style=none] (38) at (-5, 0.25) {};
		\node [style=none] (39) at (-5, -1.5) {};
		\node [style=none] (40) at (-5, -2.25) {};
		\node [style=none] (41) at (-5, -4) {};
		\node [style=none] (42) at (-8, 3.5) {};
		\node [style=none] (43) at (-8, -4.5) {};
		\node [style=none] (44) at (-8, -3) {};
		\node [style=none] (45) at (-8, 2) {};
		\node [style=none] (46) at (-8.75, -1.75) {};
		\node [style=none] (47) at (-8.75, 0.75) {};
		\node [style=none] (48) at (8, 3.5) {};
		\node [style=none] (49) at (8, -4.5) {};
		\node [style=none] (50) at (8, -3) {};
		\node [style=none] (51) at (8, 2) {};
		\node [style=none] (52) at (7.25, -1.75) {};
		\node [style=none] (53) at (7.25, 0.75) {};
		\node [style=none] (54) at (-8, 3.5) {};
		\node [style=blackdot] (55) at (6.75, -0.5) {};
		\node [style=blackdot] (56) at (8, 2) {};
		\node [style=blackdot] (57) at (8, -3) {};
		\node [style=none] (58) at (-4.25, 2) {};
		\node [style=none] (59) at (-3.75, 2) {};
		\node [style=none] (60) at (0.25, 2) {};
		\node [style=none] (61) at (0.75, 2) {};
		\node [style=none] (62) at (4.75, 2) {};
		\node [style=none] (63) at (5.25, 2) {};
		\node [style=none] (64) at (-10, 2) {};
		\node [style=none] (65) at (9, 2) {};
		\node [style=none] (66) at (9, -0.5) {};
		\node [style=none] (67) at (5.25, -0.5) {};
		\node [style=none] (68) at (4.75, -0.5) {};
		\node [style=none] (69) at (0.75, -0.5) {};
		\node [style=none] (70) at (0.25, -0.5) {};
		\node [style=none] (71) at (-3.75, -0.5) {};
		\node [style=none] (72) at (-4.25, -0.5) {};
		\node [style=none] (73) at (-10, -0.5) {};
		\node [style=none] (74) at (-10, -3) {};
		\node [style=none] (75) at (-4.25, -3) {};
		\node [style=none] (76) at (-3.75, -3) {};
		\node [style=none] (77) at (0.25, -3) {};
		\node [style=none] (78) at (0.75, -3) {};
		\node [style=none] (79) at (4.75, -3) {};
		\node [style=none] (80) at (5.25, -3) {};
		\node [style=none] (81) at (9, -3) {};
		\node [style=none] (82) at (-8, -5.5) {$F^{X}_{u_1}$};
		\node [style=none] (83) at (-4, -5.5) {$F^{X}_{5,a_1}$};
		\node [style=none] (84) at (0.5, -5.5) {$F^{X}_{5,a_2}$};
		\node [style=none] (85) at (5, -5.5) {$F^{X}_{5,a_3}$};
		\node [style=none] (86) at (8, -5.5) {$F^{X}_{5,b_1}$};
		\node [style=none] (87) at (9.5, 2) {$\Sigma_2$};
		\node [style=none] (88) at (9.5, -0.5) {$\Sigma_1$};
		\node [style=none] (89) at (9.5, -3) {$\Sigma_0$};
		\node [style=none] (90) at (-4, 4.25) {$\Theta_{6,a_1}$};
		\node [style=none] (91) at (0.5, 4.25) {$\Theta_{6,a_2}$};
		\node [style=none] (92) at (5, 4.25) {$\Theta_{6,a_3}$};
		\node [style=none] (93) at (-6.5, 2.5) {$\Theta_{5,a_1}$};
		\node [style=none] (94) at (-2, 2.5) {$\Theta_{5,a_2}$};
		\node [style=none] (95) at (2.5, 2.5) {$\Theta_{5,a_3}$};
		\node [style=none] (96) at (-6.5, 0.25) {$\Theta_{4,a_1}$};
		\node [style=none] (97) at (-5.5, 3.25) {$\Theta_{2,a_1}$};
		\node [style=none] (98) at (-1, 3.25) {$\Theta_{2,a_2}$};
		\node [style=none] (99) at (3.5, 3.25) {$\Theta_{2,a_3}$};
		\node [style=none] (100) at (-2, 0.25) {$\Theta_{4,a_2}$};
		\node [style=none] (101) at (2.5, 0.25) {$\Theta_{4,a_3}$};
		\node [style=none] (102) at (-5.75, -1.5) {$\Theta_{1,a_1}$};
		\node [style=none] (103) at (-1.5, -1.5) {$\Theta_{1,a_2}$};
		\node [style=none] (104) at (3.25, -1.5) {$\Theta_{1,a_3}$};
		\node [style=none] (105) at (-6.5, -2.5) {$\Theta_{3,a_1}$};
		\node [style=none] (106) at (-1.75, -2.5) {$\Theta_{3,a_2}$};
		\node [style=none] (107) at (2.75, -2.5) {$\Theta_{3,a_3}$};
		\node [style=none] (108) at (-5.25, -4.25) {$\Theta_{0,a_1}$};
		\node [style=none] (109) at (-0.75, -4.25) {$\Theta_{0,a_2}$};
		\node [style=none] (110) at (3.75, -4.25) {$\Theta_{0,a_3}$};
	\end{pgfonlayer}
	\begin{pgfonlayer}{edgelayer}
		\draw (0.center) to (1.center);
		\draw (5.center) to (6.center);
		\draw (2.center) to (3.center);
		\draw (4.center) to (7.center);
		\draw (8.center) to (9.center);
		\draw (10.center) to (11.center);
		\draw (12.center) to (13.center);
		\draw (14.center) to (15.center);
		\draw (19.center) to (20.center);
		\draw (16.center) to (17.center);
		\draw (18.center) to (21.center);
		\draw (22.center) to (23.center);
		\draw (24.center) to (25.center);
		\draw (26.center) to (27.center);
		\draw (28.center) to (29.center);
		\draw (33.center) to (34.center);
		\draw (30.center) to (31.center);
		\draw (32.center) to (35.center);
		\draw (36.center) to (37.center);
		\draw (38.center) to (39.center);
		\draw (40.center) to (41.center);
		\draw (48.center) to (51.center);
		\draw (49.center) to (50.center);
		\draw [bend left, looseness=1.25] (51.center) to (53.center);
		\draw [bend right, looseness=1.25] (50.center) to (52.center);
		\draw [style=blue] (64.center) to (58.center);
		\draw [style=blue] (59.center) to (60.center);
		\draw [style=blue] (61.center) to (62.center);
		\draw [style=blue] (63.center) to (65.center);
		\draw [style=blue] (73.center) to (72.center);
		\draw [style=blue] (71.center) to (70.center);
		\draw [style=blue] (69.center) to (68.center);
		\draw [style=blue] (67.center) to (66.center);
		\draw [style=blue] (74.center) to (75.center);
		\draw [style=blue] (76.center) to (77.center);
		\draw [style=blue] (78.center) to (79.center);
		\draw [style=blue] (80.center) to (81.center);
		\draw (54.center) to (45.center);
		\draw [bend left, looseness=1.25] (45.center) to (47.center);
		\draw [bend right=75] (47.center) to (46.center);
		\draw [bend left, looseness=1.25] (46.center) to (44.center);
		\draw (44.center) to (43.center);
		\draw [in=-90, out=150, looseness=0.75] (52.center) to (55);
		\draw [in=90, out=-150] (53.center) to (55);
	\end{pgfonlayer}
\end{tikzpicture}
    \caption{Reducible fibres of $\cE_5$ and the action of $\sigma_5$}
    \label{X5_fibres}
\end{figure}

    Let $\sigma_5$ be the non-symplectic automorphism induced by this base change. By construction, we have $\sigma_5(F^{X}_{5,a_1}) = F^{X}_{5,a_2}$. Then, for any component $\Theta_{i,a_1}$ of $F^{X}_{5,a_1}$, $\sigma_5(\Theta_{i,a_1})$ must be a component of $F^{X}_{5,a_2}$. However, by Shioda-Tate, irreducible components of distinct reducible fibres have distinct classes in $\NS X$.
\end{proof}

Remember that since $\sigma$ acts trivially in $\NS X$, then $(X,\sigma)$ does not admit elliptic fibrations of type 3. This is not true for $(X,\sigma_5)$, and in what follows we provide an example.

\begin{prop}
    Let $L$ and $M$ be the following divisors of $X$
    \begin{align*}
        L &= \sigma_0 + 2 \Theta_{0,a_1} + 3 \Theta_{3,a_1} + 4 \Theta_{6,a_1} + 2 \Theta_{5,a_1} + 3 \Theta_{4,a_1} + 2 \Theta_{1,a_1} + \Sigma_1,\\
        M &= \Theta_{0,a_2}.
    \end{align*}
Then, the pair $(L,M)$ defines an elliptic fibration of type 3 in $(X,\sigma_5)$.
\end{prop}

\begin{proof}
    Since $L^2 = 0$, $M^2 = -2$ and $LM = 1$, $(L,M)$ induces an elliptic fibration $\varphi_{|L|} \colon X \to \PP^1$ where $L$ is the class of the fibres in $\NS X$. Through the action of $\sigma_5$, we obtain
    \begin{align*}
        \sigma_5(L) = \sigma_0 + 2 \Theta_{0,a_2} + 3 \Theta_{3,a_2} + 4 \Theta_{6,a_2} + 2 \Theta_{5,a_2} + 3 \Theta_{4,a_2} + 2 \Theta_{1,a_2} + \Sigma_1.
    \end{align*}
    We can check that $L$ and $\sigma_5(L)$ are distinct in $\NS X$ by checking their intersection. Indeed, $L \cdot \sigma_5(L) = 4$, while $L^2 = 0$. Then, $\sigma_5$ does not fix the fibre class of $\varphi_{|L|}$, so $\varphi_{|L|}$ is of type 3 in $(X,\sigma_5)$.
\end{proof}

\bibliographystyle{amsalpha}
\bibliography{bibliography}

\newpage

\appendix
\section{Orthogonal complements of primitive embeddings}\label{appendix}

In Section \ref{section_nishiyama}, part of the application of the Kneser--Nishiyama method requires the calculation of the orthogonal complements of primitive embeddings of lattices. In this appendix we show these calculations explicitly.

Firstly, we establish some notation. Let $a_i, d_i$ and $e_i$ denote the canonical generators of the $A_n$, $D_n$ and $E_n$ lattices, respectively. For ease of notation, we define elements $\alpha_i$ in lattices of type $A_n$ as follows.

\begin{align*}
    &\alpha_i = a_{3i-2} + 2a_{3i-1} + 3a_{3i} + 2a_{3i+1} + a_{3i+2},\\
    &\alpha'_i = a_{3i-2} + 2a_{3i-1} + 3a_{3i}.\\
\end{align*}

Then, letting $\delta_0 = d_1$, we define the following elements of $D_n$ recursively.

\begin{align*}
    &\delta_i = \delta_{i-1} + d_{3i-1} + 2d_{3i} + 2d_{3i+1} + d_{3i+2},\\
    &\delta'_i = 2\delta_{i-1} + d_{3i-1} + d_{3_i} + d_{3i+1},\\
    &\delta''_i = \delta_{i-1} + d_{3i-1} + 2d_{3i} + 2d_{3i+1}.\\
\end{align*}

Finally, we denote a general element of $E_n$ as follows.

\begin{align*}
    \tiny{\stackanchor{}{\lambda_2}\stackanchor{}{\lambda_3}\stackanchor{\lambda_1}{\lambda_4}\stackanchor{}{\lambda_5}\stackanchor{}{\lambda_6}} &= \sum_{i=1}^6 \lambda_i e_i,\\
    \tiny{\stackanchor{}{\lambda_2}\stackanchor{}{\lambda_3}\stackanchor{\lambda_1}{\lambda_4}\stackanchor{}{\lambda_5}\stackanchor{}{\lambda_6}\stackanchor{}{\lambda_7}} &= \sum_{i=1}^7 \lambda_i e_i,\\
    \tiny{\stackanchor{}{\lambda_2}\stackanchor{}{\lambda_3}\stackanchor{\lambda_1}{\lambda_4}\stackanchor{}{\lambda_5}\stackanchor{}{\lambda_6}\stackanchor{}{\lambda_7}\stackanchor{}{\lambda_8}} &= \sum_{i=1}^8 \lambda_i e_i.
\end{align*}

In Table \ref{embeddings}, for pairs of lattices $N, L$, we explicitly present the unique primitive embedding $\phi \colon N \to L$ (up to an action of the Weyl group of $L$), the orthogonal complement $\phi(N)^{\perp L}$ and its root type $(\phi(N)^{\perp L})_\mathrm{root}$.

\def\block(#1,#2)#3{\multicolumn{#2}{c}{\multirow{#1}{*}{$ #3 $}}}

\begin{landscape}
\begin{table}[ht]
    \centering
    \caption{Primitive embeddings and orthogonal complements}
    \label{embeddings}
    \begin{tabular}{|c|c|c|c|c|}
        \hline
        $L$ & $N$ & $\phi \colon N \to L$ & $\phi(N)^{\perp L}$ & $(\phi(N)^{\perp L})_\mathrm{root}$ \\
        \hline
        $A_{11}$ & $A_2$ & $\langle a_1, a_2 \rangle$ &
        $\left( \begin{array}{c|cccc}
        -12 & 3 & 0 & \cdots & 0\\
        \hline
        3 & \block(4,4){\text{\Large $A_8$}}\\
        0 & & & &\\
        \vdots & & & &\\
        0 & & & &
        \end{array} \right) =
             \langle \alpha'_1,a_4,{.}{.}{.}, a_{11} \rangle$ & $A_8$\\
        \hdashline
        & $A_2^{\oplus 2}$ & $\langle a_1, a_2 \rangle{\oplus}\langle a_4, a_5 \rangle$ &
        $\left( \begin{array}{cc|cccc}
        -6 & 3 & 0 & 0 & \cdots & 0\\
        3 & -12 & 3 & 0 & \cdots & 0\\
        \hline
        0 & 3 & \block(4,4){\text{\Large $A_5$}}\\
        0 & 0 & & & &\\
        \vdots & \vdots & & & &\\
        0 & 0 & & & &
        \end{array} \right) = \langle \alpha_1, \alpha'_2,a_7, {.}{.}{.} , a_{11} \rangle$ & $A_5$\\
        \hdashline
        & $A_2^{\oplus 3}$ & $\langle a_1, a_2 \rangle{\oplus}\langle a_4, a_5 \rangle{\oplus}\langle a_7,a_8 \rangle$ & $\left( \begin{array}{ccc|cc}
        -6 & 3 & 0 & 0 & 0 \\
        3 & -6 & 3 & 0 & 0 \\
        0 & 3 & -12 & 3 & 0 \\
        \hline
        0 & 0 & 3 & \block(2,2){\text{\Large $A_2$}}\\
        0 & 0 & 0 & & \\
        \end{array} \right) = \langle \alpha_1,\alpha_2,\alpha'_3, a_{10}, a_{11}\rangle$ & $A_2$\\
        \hdashline
        & $A_2^{\oplus 4}$ & $\bigoplus_{i=0}^{3} \langle a_{3i+1}, a_{3i+2} \rangle$ & $A_3(3) = \left( \begin{array}{cccc}
        -6 & 3 & 0\\
        3 & -6 & 3\\
        0 & 3 & -6\\
        \end{array} \right) = \langle \alpha_1, \alpha_2, \alpha_3 \rangle$ & $0$\\
        \hline
        $A_{17}$ & $A_2$ & $\langle a_1, a_2 \rangle$ & $\left( \begin{array}{c|cccc}
        -12 & 3 & 0 & \cdots & 0\\
        \hline
        3 & \block(4,4){\text{\Large $A_{14}$}}\\
        0 & & & &\\
        \vdots & & & &\\
        0 & & & &
        \end{array} \right) =\langle \alpha'_1,a_4,{.}{.}{.}, a_{17} \rangle$ & $A_{14}$\\
        \hline
    \end{tabular}
\end{table}

\newpage

\begin{table}
    \begin{tabular}{|c|c|c|c|c|}
        \hline
        $L$ & $N$ & $\phi \colon N \to L$ & $\phi(N)^{\perp L}$ & $(\phi(N)^{\perp L})_\mathrm{root}$ \\
        \hline
       $A_{17}$ & $A_2^{\oplus 2}$ & $\langle a_1, a_2 \rangle{\oplus}\langle a_4, a_5 \rangle$ & $\left( \begin{array}{cc|cccc}
        -6 & 3 & 0 & 0 & \cdots & 0\\
        3 & -12 & 3 & 0 & \cdots & 0\\
        \hline
        0 & 3 & \block(4,4){\text{\Large $A_{11}$}}\\
        0 & 0 & & & &\\
        \vdots & \vdots & & & &\\
        0 & 0 & & & &
        \end{array} \right) = \langle \alpha_1, \alpha'_2,a_7, {.}{.}{.} , a_{17} \rangle$ & $A_{11}$\\
        \hdashline
        & $A_2^{\oplus 3}$ & $\langle a_1, a_2 \rangle{\oplus}\langle a_4, a_5 \rangle{\oplus}\langle a_7,a_8 \rangle$ & $\left( \begin{array}{ccc|cccc}
        -6 & 3 & 0 & 0 & 0 & \cdots & 0 \\
        3 & -6 & 3 & 0 & 0 & \cdots & 0 \\
        0 & 3 & -12 & 3 & 0 & \cdots & 0 \\
        \hline
        0 & 0 & 3 & \block(4,4){\text{\Large $A_{8}$}}\\
        0 & 0 & 0 & & & & \\
        \vdots & \vdots & \vdots & & & &\\
        0 & 0 & 0 & & & &
        \end{array} \right) = \langle \alpha_1,\alpha_2,\alpha'_3, a_{10}, {.}{.}{.}, a_{17}\rangle$ & $ A_8$\\
        \hdashline
        & $A_2^{\oplus 4}$ & $\bigoplus_{i=0}^{3} \langle a_{3i+1}, a_{3i+2} \rangle$ & $\left( \begin{array}{cccc|cccc}
        -6 & 3 & 0 & 0 & 0 & 0 & \cdots & 0 \\
        3 & -6 & 3 & 0 & 0 & 0 & \cdots & 0 \\
        0 & 3 & -6 & 3 & 0 & 0 & \cdots & 0 \\
        0 & 0 & 3 & -12 & 3 & 0 & \cdots & 0 \\
        \hline
        0 & 0 & 0 & 3 & \block(4,4){\text{\Large $A_{5}$}}\\
        0 & 0 & 0 & 0 & & & & \\
        \vdots & \vdots & \vdots & \vdots & & & &\\
        0 & 0 & 0 & 0 & & & &
        \end{array} \right) = \begin{array}{c}
            \langle \alpha_1, \alpha_2, \alpha_3, \alpha'_4,\\
            a_{13},{.}{.}{.}, a_{17}\rangle
        \end{array}$ & $A_5$\\
        \hdashline
        & $A_2^{\oplus 5}$ & $\bigoplus_{i=0}^{4} \langle a_{3i+1}, a_{3i+2} \rangle$ & $\left( \begin{array}{ccccc|cc}
        -6 & 3 & 0 & 0 & 0 & 0 & 0\\
        3 & -6 & 3 & 0 & 0 & 0 & 0\\
        0 & 3 & -6 & 3 & 0 & 0 & 0\\
        0 & 0 & 3 & -6 & 3 & 0 & 0\\
        0 & 0 & 0 & 3 & -12 & 3 & 0\\
        \hline
        0 & 0 & 0 & 0 & 3 & \block(2,2){\text{\Large $A_{2}$}}\\
        0 & 0 & 0 & 0 & 0 & & \\
        \end{array} \right) = \begin{array}{c}
            \langle \alpha_1, \alpha_2, \alpha_3, \alpha_4,\\ \alpha'_5,a_{16},a_{17}\rangle
        \end{array}$ & $A_2$\\
        \hline
    \end{tabular}
\end{table}

\newpage

\begin{table}
    \begin{tabular}{|c|c|c|c|c|}
        \hline
        $L$ & $N$ & $\phi \colon N \to L$ & $\phi(N)^{\perp L}$ & $(\phi(N)^{\perp L})_\mathrm{root}$ \\
        \hline
        $A_{17}$ & $A_2^{\oplus 6}$ & $\bigoplus_{i=0}^{5} \langle a_{3i+1}, a_{3i+2} \rangle$ & $A_5(3) = \begin{pmatrix}
            -6 & 3 & 0 & 0 & 0\\
            3 & -6 & 3 & 0 & 0\\
            0 & 3 & -6 & 3 & 0\\
            0 & 0 & 3 & -6 & 3\\
            0 & 0 & 0 & 3 & -6
        \end{pmatrix} = \begin{array}{c}
            \langle \alpha_1, \alpha_2, \alpha_3,\\
            \alpha_4, \alpha_5\rangle
        \end{array}$ & $0$\\
        \hline
        $D_7$ & $A_2$ & $\langle d_2, d_3 \rangle$ & $\left( \begin{array}{c|cccc}
            -4 & -1 & 1 & 0 & 0 \\
            \hline
            -1 & \block(4,4){\text{\Large $D_4$}}\\
            1 & & & &\\
            0 & & & &\\
            0 & & & &\\
        \end{array}\right) = \langle \delta'_1, \delta_1, d_5, d_6, d_7\rangle$ & $D_4$\\
        \hdashline
        & $A_2^{\oplus 2}$ & $\langle d_2, d_3 \rangle \oplus \langle d_5, d_6 \rangle$ & $\begin{pmatrix}
            -4 & -1 & 0 \\
            -1 & -4 & -2\\
            0 & -2 & -4
        \end{pmatrix} = \langle \delta'_1, \delta'_2, \delta''_2\rangle$ & $0$ \\
        \hline
        $D_{10}$ & $A_2$ & $\langle d_2, d_3 \rangle$ & $\left( \begin{array}{c|ccccc}
            -4 & -1 & 1 & 0 & \cdots & 0 \\
            \hline
            -1 & \block(5,5){\text{\Large $D_7$}}\\
            1 & & & & &\\
            0 & & & & &\\
            \vdots & & & & &\\
            0 & & & & &\\
        \end{array}\right) = \langle \delta'_1, \delta_1, d_5, \ldots, d_{10}\rangle$ & $D_7$\\
        \hdashline
        & $A_2^{\oplus 2}$ & $\langle d_2, d_3 \rangle \oplus \langle d_5, d_6 \rangle$ & $\left( \begin{array}{cc|cccc}
            -4 & -1 & 0 & 0 & 0 & 0\\
            -1 & -4 & -1 & 1 & 0 & 0\\
            \hline
            0 & -1 & \block(4,4){\text{\Large $D_4$}}\\
            0 & 1 & & & &\\
            0 & 0 & & & &\\
            0 & 0 & & & &\\
        \end{array} \right) = \begin{array}{c}
            \langle \delta'_1 , \delta'_2, \delta_2,\\
            d_8, d_9, d_{10} \rangle
        \end{array}$ & $D_4$ \\
        \hdashline
        & $A_2^{\oplus 3}$ & $\langle d_2, d_3 \rangle \oplus \langle d_5, d_6 \rangle \oplus \langle d_8, d_9 \rangle$ & $\begin{pmatrix}
            -4 & -1 & 0 & 0\\
            -1 & -4 & -1 & 0\\
            0 & -1 & -4 & -2\\
            0 & 0 & -2 & -4\\
        \end{pmatrix} = \langle \delta'_1, \delta'_2, \delta'_3, \delta''_3\rangle$ & $0$\\
        \hline
    \end{tabular}
\end{table}

\newpage

\begin{table}
    \begin{tabular}{|c|c|c|c|c|}
        \hline
        $L$ & $N$ & $\phi \colon N \to L$ & $\phi(N)^{\perp L}$ & $(\phi(N)^{\perp L})_\mathrm{root}$ \\
        \hline
        $D_{16}$ & $A_2$ & $\langle d_2, d_3 \rangle$ & $\left( \begin{array}{c|ccccc}
            -4 & -1 & 1 & 0 & \cdots & 0 \\
            \hline
            -1 & \block(5,5){\text{\Large $D_{13}$}}\\
            1 & & & & &\\
            0 & & & & &\\
            \vdots & & & & &\\
            0 & & & & &\\
        \end{array}\right) = \langle \delta'_1, \delta_1, d_5, \ldots, d_{16}\rangle$ & $D_{13}$\\
        \hdashline
        & $A_2^{\oplus 2}$ & $\langle d_2, d_3 \rangle \oplus \langle d_5, d_6 \rangle$ & $\left( \begin{array}{cc|ccccc}
            -4 & -1 & 0 & 0 & 0 & \cdots & 0\\
            -1 & -4 & -1 & 1 & 0 & \cdots & 0\\
            \hline
            0 & -1 & \block(5,5){\text{\Large $D_{10}$}}\\
            0 & 1 & & & & &\\
            0 & 0 & & & & &\\
            \vdots & \vdots & & & & &\\
            0 & 0 & & & & &\\
        \end{array} \right) = \begin{array}{c}
            \langle \delta'_1 , \delta'_2, \delta_2,\\
            d_8, \ldots, d_{16} \rangle
        \end{array}$ & $D_{10}$\\
        \hdashline
        & $A_2^{\oplus 3}$ & $\langle d_2, d_3 \rangle \oplus \langle d_5, d_6 \rangle \oplus \langle d_8, d_9 \rangle$ & $\left( \begin{array}{ccc|ccccc}
            -4 & -1 & 0 & 0 & 0 & 0 & \cdots & 0\\
            -1 & -4 & -1 & 0 & 0 & 0 & \cdots & 0\\
            0 & -1 & -4 & -1 & 1 & 0 & \cdots & 0\\
            \hline
            0 & 0 & -1 & \block(5,5){\text{\Large $D_{7}$}}\\
            0 & 0 & 1 & & & & &\\
            0 & 0 & 0 & & & & &\\
            \vdots & \vdots & \vdots & & & & &\\
            0 & 0 & 0 & & & & &\\
        \end{array} \right) = \begin{array}{c}
            \langle \delta'_1 , \delta'_2, \delta'_3, \delta_3,\\
            d_{11}, \ldots, d_{16} \rangle
        \end{array}$ & $D_7$ \\
        \hdashline
        & $A_2^{\oplus 4}$ & $\bigoplus_{i=1}^{4} \langle d_{3i-1}, d_{3i} \rangle$ & $\left( \begin{array}{cccc|cccc}
            -4 & -1 & 0 & 0 & 0 & 0 & 0 & 0\\
            -1 & -4 & -1 & 0 & 0 & 0 & 0 & 0\\
            0 & -1 & -4 & -1 & 0 & 0 & 0 & 0\\
            0 & 0 & -1 & -4 & -1 & 1 & 0 & 0\\
            \hline
            0 & 0 & 0 & -1 & \block(4,4){\text{\Large $D_{4}$}}\\
            0 & 0 & 0 & 1 & & & &\\
            0 & 0 & 0 & 0 & & & &\\
            0 & 0 & 0 & 0 & & & &\\
        \end{array} \right) = \begin{array}{c}
            \langle \delta'_1 , \delta'_2, \delta'_3, \delta'_4\\
            \delta_4, d_{14}, d_{15}, d_{16} \rangle
        \end{array}$ & $D_4$ \\
        \hline
    \end{tabular}
\end{table}

\newpage

\begin{table}
    \begin{tabular}{|c|c|c|c|c|}
        \hline
        $L$ & $N$ & $\phi \colon N \to L$ & $\phi(N)^{\perp L}$ & $(\phi(N)^{\perp L})_\mathrm{root}$ \\
        \hline
        $D_{16}$ & $A_2^{\oplus 5}$ & $\bigoplus_{i=1}^{5} \langle d_{3i-1}, d_{3i} \rangle$ & $\begin{pmatrix}
            -4 & -1 & 0 & 0 & 0 & 0\\
            -1 & -4 & -1 & 0 & 0 & 0\\
            0 & -1 & -4 & -1 & 0 & 0\\
            0 & 0 & -1 & -4 & -1 & 0\\
            0 & 0 & 0 & -1 & -4 & -2\\
            0 & 0 & 0 & 0 & -2 & -4\\
        \end{pmatrix} = \begin{array}{c}
            \langle \delta'_1, \delta'_2, \delta'_3,\\
            \delta'_4, \delta''_5, \delta''_5 \rangle
        \end{array}$ & $0$\\
        \hline
        $E_6$ & $A_2$ & $\langle e_2, e_3 \rangle$ & $\begin{matrix}\\A_2^{\oplus 2} = \langle e_5, e_6 \rangle \oplus \langle e_1, \tiny{\stackanchor{}{1}\stackanchor{}{2}\stackanchor{2}{3}\stackanchor{}{2}\stackanchor{}{1}} \rangle \\ & \end{matrix}$ & $A_2^{\oplus 2}$\\
        \hdashline
        & $A_2^{\oplus 2}$ & $\langle e_2, e_3 \rangle \oplus \langle e_5, e_6 \rangle$ & $\begin{matrix}\\ A_2 = \langle e_1, \tiny{\stackanchor{}{1}\stackanchor{}{2}\stackanchor{2}{3}\stackanchor{}{2}\stackanchor{}{1}} \rangle \\ & \end{matrix}$ & $A_2$\\
        \hdashline
        & $E_6$ & $\langle e_1, {.}{.}{.}, e_6 \rangle$ & $0$ & $0$ \\
        \hline
        $E_7$ & $A_2 $ & $\langle e_2, e_3 \rangle$ & $\begin{matrix}\\ A_5 = \langle e_5, e_6, e_7,\tiny{\stackanchor{}{1}\stackanchor{}{2}\stackanchor{2}{3}\stackanchor{}{2}\stackanchor{}{1}\stackanchor{}{0}}, e_1 \rangle \\ & \end{matrix}$ & $A_5$\\
        \hdashline
        & $A_2^{\oplus 2}$ & $\langle e_2, e_3 \rangle \oplus \langle e_5, e_6 \rangle$ & $\begin{matrix}\\ (-6) \oplus A_2 = \tiny{\langle \stackanchor{}{2}\stackanchor{}{4}\stackanchor{3}{6}\stackanchor{}{5}\stackanchor{}{4}\stackanchor{}{3}}\rangle \oplus \langle e_1, \tiny{\stackanchor{}{1}\stackanchor{}{2}\stackanchor{2}{3}\stackanchor{}{2}\stackanchor{}{1}\stackanchor{}{0}}\rangle \\ & \end{matrix}$ & $A_2$\\
        \hdashline
        & $E_6$ & $\langle e_1, {.}{.}{.}, e_6 \rangle$ & $\begin{matrix}\\ (-6) = \langle \tiny{\stackanchor{}{2}\stackanchor{}{4}\stackanchor{3}{6}\stackanchor{}{5}\stackanchor{}{4}\stackanchor{}{3}}\rangle \\ & \end{matrix}$ & $0$ \\
        \hline
        $E_8$ & $A_2 $ & $\langle e_2, e_3 \rangle$ & $\begin{matrix}\\ E_6 = \langle e_8, e_5, e_6, e_7, \tiny{\stackanchor{}{1}\stackanchor{}{2}\stackanchor{2}{3}\stackanchor{}{2}\stackanchor{}{1}\stackanchor{}{0}\stackanchor{}{0}}, e_1 \rangle \\ & \end{matrix}$ & $E_6$\\
        \hdashline
        & $A_2^{\oplus 2}$ & $\langle e_2, e_3 \rangle \oplus \langle e_5, e_6 \rangle$ & $\begin{matrix}\\ A_2^{\oplus 2} = \langle e_1, \tiny{\stackanchor{}{1}\stackanchor{}{2}\stackanchor{2}{3}\stackanchor{}{2}\stackanchor{}{1}\stackanchor{}{0}\stackanchor{}{0}}\rangle \oplus \langle e_8, \tiny{\stackanchor{}{2}\stackanchor{}{4}\stackanchor{3}{6}\stackanchor{}{5}\stackanchor{}{4}\stackanchor{}{3}\stackanchor{}{2}}\rangle \\ & \end{matrix}$ & $A_2^{\oplus 2}$\\
        \hdashline
        & $E_6$ & $\langle e_1, {.}{.}{.}, e_6 \rangle$ & $\begin{matrix}\\ A_2 = \langle e_8, \tiny{\stackanchor{}{2}\stackanchor{}{4}\stackanchor{3}{6}\stackanchor{}{5}\stackanchor{}{4}\stackanchor{}{3}\stackanchor{}{2}}\rangle \\ & \end{matrix}$ & $A_2$\\
        \hline
    \end{tabular}
\end{table}
\end{landscape}

\end{document}